\documentclass[10pt]{article}
\usepackage{amsmath,amssymb,amsfonts,amsthm,amscd,graphicx,psfrag,epsfig}
\usepackage{color}
\definecolor{blue}{rgb}{0,0,0.7}
\definecolor{red}{rgb}{0.75, 0, 0}
\usepackage{stmaryrd}
\usepackage[title]{appendix}
\usepackage{hyperref}
\hypersetup{colorlinks, linkcolor=blue, citecolor=cyan}
\usepackage[all]{xy}           
\usepackage{marginnote}  
\usepackage{float}
\usepackage{soul}

\newtheorem{theorem}{Theorem}[section]
\newtheorem{lemma}[theorem]{Lemma}
\newtheorem{proposition}[theorem]{Proposition}
\newtheorem{corollary}[theorem]{Corollary}
\newtheorem{conjecture}[theorem]{Conjecture}
\newtheorem{definition}[theorem]{Definition}

\newtheorem{remark}[theorem]{Remark}

\newcommand{\bpf}{\begin{proof}}
\newcommand{\epf}{\end{proof}}
\newcommand{\bs}{\begin{split}}
\newcommand{\es}{\end{split}}
\newcommand{\be}{\begin{equation}}
\newcommand{\ee}{\end{equation}}
\newcommand{\bt}{\begin{theorem}}
\newcommand{\et}{\end{theorem}}
\newcommand{\bd}{\begin{definition}}
\newcommand{\ed}{\end{definition}}
\newcommand{\bp}{\begin{proposition}}
\newcommand{\ep}{\end{proposition}}
\newcommand{\bl}{\begin{lemma}}
\newcommand{\el}{\end{lemma}}
\newcommand{\bc}{\begin{corollary}}
\newcommand{\ec}{\end{corollary}}
\newcommand{\bcon}{\begin{conjecture}}
\newcommand{\econ}{\end{conjecture}}
\newcommand{\la}{\label}

\newcommand{\Z}{{\mathbb Z}}
\newcommand{\R}{{\mathbb R}}
\newcommand{\Q}{{\mathbb Q}}
\newcommand{\C}{{\mathbb C}}
\newcommand{\G}{{\rm G}}

\newcommand{\lra}{\longrightarrow}
\newcommand{\lms}{\longmapsto}
\newcommand{\bS}{{\Bbb S}}

\newcommand{\bg}{\begin{equation}\begin{gathered}}
\newcommand{\eg}{\end{gathered}\end{equation}}

\newcommand{\F}{{\rm F}}

\newcommand{\Y}{{\rm Y}}

\newcommand{\rE}{{\rm E}}

\newcommand{\K}{{\rm K}}

\newcommand{\N}{{\rm N}}

\newcommand{\X}{{\rm X}}

\newcommand{\old}[1]{}

\newcommand{\D}{{\rm  D}}

\newcommand{\rC}{{\rm  C}}
\newcommand{\T}{{\rm T}}

\newcommand{\A}{{\rm A}}
\newcommand{\B}{{\rm B}}
\newcommand{\U}{{\rm {U}}}
\renewcommand{\P}{\mathbb{P}}

\renewcommand{\H}{{\rm H}}
\renewcommand{\P}{{\rm P}}

\newcommand{\bsp}{\begin{split}}
\newcommand{\esp}{\end{split}}
\newcommand{\epr}{\end{proof}}
\newcommand{\bpr}{\begin{proof}}

\renewcommand{\L}{{\rm L}}

\input xy
\xyoption{all}

\begin{document}

\date{}

\title {Exponential volumes  of moduli spaces of hyperbolic surfaces}
\author{Alexander B. Goncharov, Zhe Sun}

\maketitle

\begin{abstract}
A decorated surface $\bS$ is an oriented topological surface  with marked points on the boundary considered modulo the isotopy. 
We consider the moduli space ${\cal M}_\bS$ of hyperbolic structures on $\bS$ with geodesic boundary, such that the 
hyperbolic structure near each marked point is a cusp,  equipped with  a horocycle.   The  space ${\cal M}_\bS$ carries a volume form $\Omega$. 
Let us fix the set  ${\K}$ of the   
 distances between  the horocycles at   the adjacent  cusps, and the set $\L$ of the geodesic lengths of  boundary circles without cusps. 
We get a subspace  ${\cal M}_\bS({\K}, \L)$  
 with the induced 
volume form $\Omega_\bS({{\K}, \L})$.  However,  
 if the cusps are present,  the volume of  the space ${\cal M}_\bS({\K}, \L)$, or its   variant without horocycles, is infinite. \vskip 2mm
 
 We introduce the {\it exponential  volume form} $e^{-{W}}\Omega$, where 
${W}$ is the {\it potential} - a positive function on  ${\cal M}_\bS$, given by the sum over the cusps of the hyperbolic areas under the horocycles. 
We show that  the following {\it exponential volume} is finite:
\be \la{VF}
\int_{{\cal M}_\bS({\K}, \L)}e^{-{W}}\Omega_\bS({{\K}, \L}).
\ee

 We suggest that  moduli spaces ${\cal M}_\bS({\K}, \L)$  with the exponential volume forms  are the true 
 analogs of  moduli spaces ${\cal M}_{g,n}$ with the Weil--Petersson volume form, e.g.  relevant to the open string theory. \vskip 2mm

 We  prove    
  {\it unfolding formulas}, expressing   integrals $\int_{{\cal M}_\bS({\K}, \L)}f\ e^{-{W}}\Omega_\bS({{\K}, \L})$, where $f$ is an integrable  function, 
  as  finite sums of similar integrals 
  for  the three  {\it elementary  decorated  surfaces}.  They  generalise Mirzakhani's recursions for the 
 volumes of  moduli spaces of hyperbolic surfaces. 
 \vskip 1mm
 
 We show that exponential volumes  for  elementary decorated surfaces give rise to a commutative  algebra ${\cal E}$, which we call 
 the {\it  positive Hecke-Whittaker algebra for ${\rm PGL}_2(\R)$}.   
Exponential volumes  for all decorated surfaces and  unfolding formulas   extend the algebra ${\cal E}$ to all decorated surfaces.

 \end{abstract}

\tableofcontents

   \section{Introduction}
\subsection{Exponential volumes of  moduli spaces of ideal hyperbolic surfaces}

\subsubsection{Ideal hyperbolic surfaces.} A {\em decorated surface} $\bS$ is   
 a connected oriented topological surface  with boundary circles, and  a finite number of marked points on the boundary,  considered modulo isotopy. We sometimes refer to the boundary circles without  marked points as {\it punctures}, see Figure \ref{figure:dsurf}, and to the marked points as {\it cusps}.
\begin{figure}[ht]
	\centering
	\includegraphics[scale=0.4]{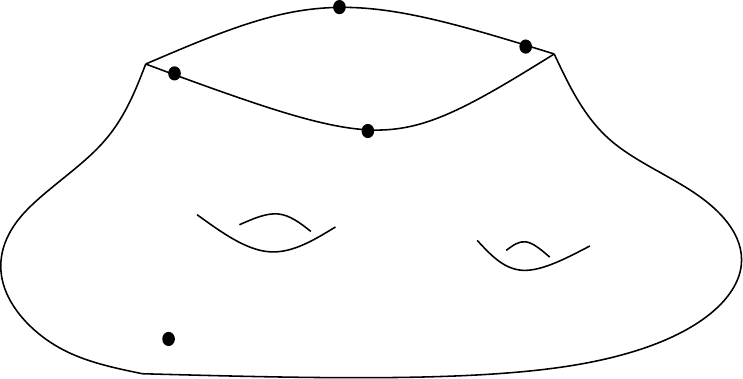}
	\small
	\caption{A decorated surface with a puncture and four marked points.}
	\label{figure:dsurf}
\end{figure}

Denote by $\D_n^*$  a once punctured disc with $n$ marked boundary points, see the left picture on Figure \ref{figure:crown}. 

A {\em hyperbolic crown} is the decorated surface $\D_n^*$  with a hyperbolic structure, which  is bordered on the one side by the {\em crown end} - a collection of bi-infinite geodesics such that each adjacent pair forms a cusp -  and  on the other side by the unique geodesic, called the {\em neck geodesic}, plus a horocycle near each cusp, see the middle picture on Figure \ref{figure:crown}. The universal cover of the hyperbolic crown is shown on the right of Figure \ref{figure:crown}. A fundamental domain for the deck transformation action of the fundamental group $\Z$ 
 is a geodesic $(n+3)-$gon in the hyperbolic plane, bounded by the universal cover $\widetilde \gamma$ of the neck geodesic $\gamma$. 

\begin{figure}[ht]
	\centering
	\includegraphics[scale=0.4]{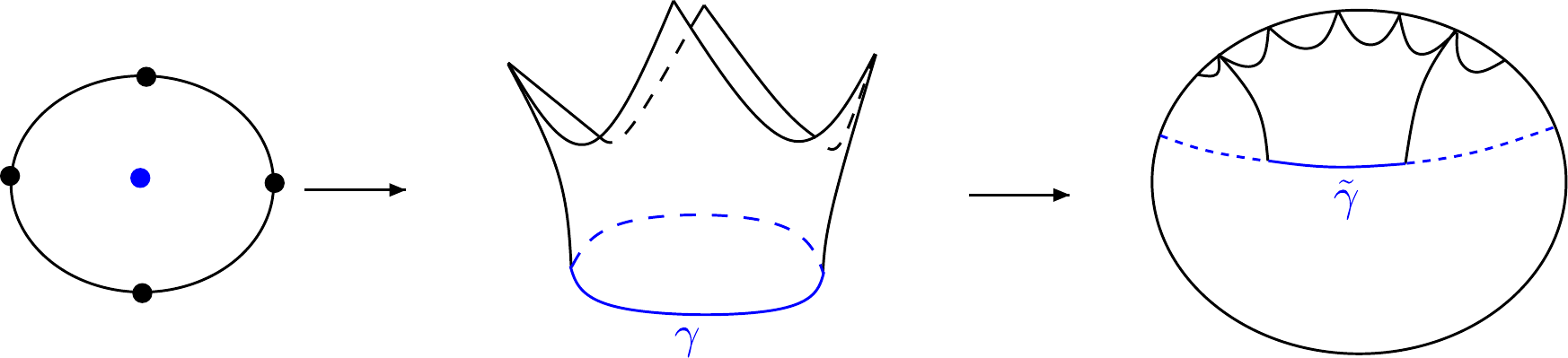}
	\small
	\caption{An ideal hyperbolic structure on a punctured disc with $4$ marked points is a hyperbolic  crown with $4$ cusps and the neck geodesic $\gamma$, and  a choice of a horocycle at each cusp, omitted on the picture. }
	\label{figure:crown}
\end{figure}

\bd \la{D1.1} An {\em ideal hyperbolic structure} on a decorated surface $\bS$ is  a hyperbolic metric 
 of curvature $-1$ on $\bS$ with the following boundary structure, illustrated on Figure \ref{figure:crsurf}:  

\begin{itemize}

 \item Each boundary circle without marked points  is a geodesic.

\item  The hyperbolic metric near each boundary component with marked points is a crown end. 

\item A choice of a horocycle $h_p$ near each cusp $p$.

\end{itemize}
\ed

\begin{figure}[ht]
	\centering
	\includegraphics[scale=0.35]{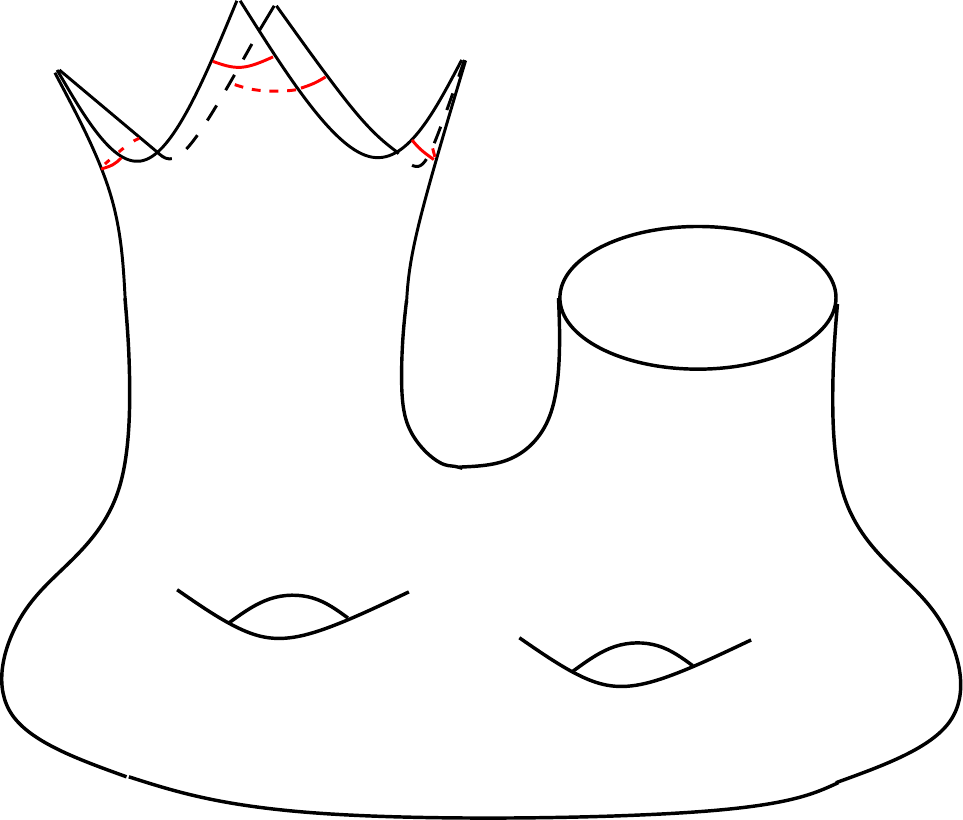}
	\small
	\caption{An ideal hyperbolic surface. The horocycles around the cusps are shown in red.}
	\label{figure:crsurf}
\end{figure}


Let us now introduce  the  moduli spaces which we study in the paper.

\begin{definition} \la{Thor}

The {Teichm\"uller space} $\mathcal{T}_\bS$   of a decorated surface $\bS$ parametrises ideal hyperbolic structures on $\bS$, up to orientation and boundary preserving homeomorphisms of $\bS$ isotopic to the identity.

The {pure mapping class group} ${\rm Mod}(\bS)$  is the quotient of the group of orientation $\&$ boundary preserving homeomorphisms of $\bS$ modulo the ones   isotopic to the  identity. 
The {moduli space}  of    $\mathcal{M}_\bS$ is the quotient
$$
\mathcal{M}_\bS:= \mathcal{T}_\bS/{\rm Mod}(\bS).
$$
\end{definition}

Here are a few examples and comments.

\begin{enumerate}

\item 
 When $\bS$ is a disc with $n$ marked points, an ideal  hyperbolic structure on $\bS$ describes a collection of $n$ horocycles  in the hyperbolic disc. 
Each horocycle defines a point on the boundary circle, and we  connect each pair of  adjacent points by a bi-infinite geodesic. 
The Teichm\"uller space ${\cal T}_\bS$ parametrises ideal $n-$gons  with  horocycles at the  vertices, see Figure \ref{SZhe1}. The group ${\rm Mod}(\bS)$ is trivial. 
So $\mathcal{M}_\bS= \mathcal{T}_\bS$. 
\begin{figure}[ht]
\centerline{\epsfbox{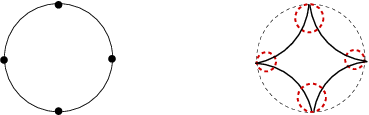}}
\caption{An ideal hyperbolic structure  with (red) horocycles  on a disc with marked points.}
\label{SZhe1}
\end{figure} 

\item  When $\bS=\D^*_n$ is a punctured disc with $n$ marked points, the Teichm\"uller space ${\cal T}_{\D_n^*}$  parametrises hyperbolic crowns   with $n$ cusps, see Figure \ref{figure:crown}.  The group ${\rm Mod}(\bS)$ is trivial. 
So $\mathcal{M}_\bS= \mathcal{T}_\bS$.

 \item  When $\bS$ does not have any crown end, the Teichm\"uller space $\mathcal{T}_\bS$ is the usual Teichm\"uller space. 

\item The classical Teichmuller space ${\cal T}^{\rm cl}_\bS$  parametrises hyperbolic structures on $\bS$ with geodesic boundary and crown ends, see  \cite[Chapter 4]{CB}.  
There is a canonical ${\rm Mod}(\bS)-$equivariant  isomorphism  
\be
{\cal T}_\bS \stackrel{\sim}{\lra} {\cal T}^{\rm cl}_\bS\times \R^{\{\mbox{\rm cusps $p$ on $\bS$}\}}.
\ee
 The projection onto  ${\cal T}^{\rm cl}_\bS$  forgets the horoarcs.    The map onto the second factor is given by the areas  of domains enclosed between the cusps $p$ and the horoarcs $h_p$. 
 It plays the crucial role in the story. 

\item Consider the decorated  Teichmuller space ${\cal T}^{\rm d}_\bS$ parametrising ideal hyperbolic structures on $\bS$  where  geodesic boundary circles are reduced to  punctures, that is have zero  length, 
and  equipped with horocycles. 
For a punctured surface without boundary we get Penner's decorated Teichmuler space \cite{Pen92}. 
 The Teichmuller spaces ${\cal T}_\bS$ and ${\cal T}^{\rm d}_\bS$ are dual to each 
other. In particular ${\rm dim}{\cal T}_\bS = {\rm dim}{\cal T}^{\rm d}_\bS$. This duality is a manifestation of the cluster duality \cite{FG03b}, see  Section \ref{Sec2}.  
\end{enumerate}
 
  A decorated surface $\bS$ carries a collection of  isotopy classes of simple loops,   isotopic to the boundary circles, with or without marked points. We call them {\it boundary loops} of $\bS$. Given an ideal hyperbolic structure on $\bS$, 
each boundary loop is represented by the {\it neck geodesic}. If a boundary component is not a crown, the neck geodesic is  the {\it boundary geodesic loop}. 
We denote by $m$ the number of boundary components of $\bS$,  by 
$l_1, \ldots, l_m$ the lengths of the neck geodesics, and  by $r$ the number of boundary components with marked points. 
 
 For each boundary geodesic   connecting two adjacent cusps  at the ends of a boundary interval $\F$, we define its length $\kappa_{\F}\in \R$ as the  distance between the two horocycles at the cusps, taken with the minus sign if the horocycles overlap. We call it the {\it signed distance}. 
 Let us define the {\it $\K-$coordinate}\footnote{The reason for the factor $-1$ will be clarified later, see Definition \ref{OMD}.} 

\be \la{KF}
{\rm K}_{\F}:= e^{-\kappa_{\F}}.
\ee

Take an ideal hyperbolic structure on a decorated surface $\bS$ with  ideal crowns $\rC_1,\ldots, \rC_r$. Given an ideal crown $\rC_i$  with  $n_i$ cusps, there are   ${\rm K}-$coordinates ${\rm K}_{\rC_i, j}$, $1 \leq j \leq n_i$ at the boundary intervals of the crown.  They are positive numbers. We denote by ${\K}$ the collection of these numbers:
\be \la{KC}
{\K}:= \{{\K}_{{\rC_i}}\}, \qquad {\K}_{{\rC_i}}:=({\K}_{{\rC}_{i, 1}},...,  {\K}_{{\rC}_i, n_i})\in {\Bbb R}^{n_i}_{>0}. 
\ee

 We denote the  geodesic lengths of the  boundary geodesic circles   by 
 \be \la{spL}
 \L=(l_{r+1},...,l_{m})\in \R_{\geq 0}^{m-r}.
 \ee

\begin{definition} \la{DEFMKL}
The Teichm\"uller space  ${\cal T}_\bS({\K}, \L)$ is the subspace of the Teichm\"uller space ${\cal T}_\bS$  with the given set ${\K}$ of  ${\rm K}-$coordinates (\ref{KC}), 
 and  the given set $\L$ of  lengths of the  boundary geodesic circles (\ref{spL}).
 
 The 
moduli space $\mathcal{M}_\bS(\K, \L)$ is the quotient:
\be \la{MS}
\begin{split}
&\mathcal{M}_\bS({\K}, \L):= \mathcal{T}_\bS({\K}, \L)/{\rm Mod}(\bS).\\
\end{split}
\ee
We denote by $\mathcal{T}_\bS({\K})$ and $\mathcal{M}_\bS({\K})$  the Teichm\"uller and moduli space when $\L$ is not fixed. 
\ed
If $\bS$ is a genus $g$ surface without  marked points and with 
$m$ boundary circles of zero lengths, we get the moduli space     ${\cal M}_{g,m}$.

  \subsubsection{The exponential volumes.} 
  
  Recall that an ideal hyperbolic structure on $\bS$ carries a horoarc $h_p$ at each cusp $p$. We denote by $W_p$ the area enclosed by the horoarc  $h_p$. We call it   the {\it local potential at the cusp $p$}, or just the potential at the cusp $p$. The {\it potential} $W_\bS$ is a function on the Teichm\"uller space ${\cal T}_\bS$ given by  the sum of the  potentials at all cusps  on $\bS$:
 $$
 W_\bS := \sum_{p}W_p.
 $$
 Both the local potentials $W_p$ and the potential $W_\bS$ are ${\rm Mod}(\bS)-$invariant. Therefore they are functions on the moduli space (\ref{MS}).

 Given an orientation of the Teichm\"uller space $\mathcal{T}_\bS$, there  is canonical  
 volume form $\Omega_\bS$  on  $\mathcal{T}_\bS$, with positive integrals over  discs. 
 It has a simple expression (\ref{CVF})  in the cluster Poisson coordinates. 
It is ${\rm Mod}(\bS)-$invariant, and  induces a volume form $\Omega_\bS({{\K}, \L})$ on the moduli space $\mathcal{M}_\bS({\K}, \L)$ as in equation (\ref{RVF2}).
When $\bS$ has no marked points, it is propotional to the  Weil--Petersson volume form. The exponential volume form is a positive measure on the oriented Teichm\"uller space $\mathcal{T}_\bS$ given by 
\be \la{6}
e^{-W_\bS}\Omega_\bS({{\K}, \L}).
\ee

\bd The  {exponential  volume}  
of the moduli space $\mathcal{M}_\bS({\K}, \L)$ is given by
\be \la{EXV}
{\rm Vol}_{{\mathcal{E}}}\mathcal{M}_\bS({\K}, \L):=\int_{\mathcal{M}_\bS({\K}, \L)} e^{-W_\bS}\Omega_\bS({{\K}, \L}).
\ee
\end{definition}

\subsubsection{Mirzakhani's volume formula.} When $\bS=S$ has no    marked points, the potential ${W}_S$ is zero, and the set $\K$ is empty. Then we write $\Omega_S(\L)$ for the cluster form ({\ref{6}). 
It is proportinal to  the Weil--Petersson volume form: 
\be \la{WPCL}
\Omega_S(\L) = d_S\Omega^{\rm WP}_S(\L), \ \ \ \ d_S \in \Q^*_+,
\ee
 for some positive rational constant $d_S$.\footnote{
Using \cite[Lemma 3.8 and Theorem 3.9]{ABCGLW20}, we have $d_S=2^{-2g+3-n}= 2^{\chi_S+1}$.
If $g=0$,  see also formula (\ref{WPCLn}) in Appendix \ref{APPB}.} So we get  the Weil-Peterssen volume of  the moduli space $\mathcal{M}_S(\L)$ of  genus $g$ hyperbolic surfaces with fixed boundary geodesic lengths $\L = (l_1, ..., l_m)$:
$$
{\rm Vol}_{\rm WP}(\mathcal{M}_S)(\L):= \int_{\mathcal{M}_S(\L)} \Omega^{\rm WP}_S(\L).
$$ 
By Mirzakhani's theorem \cite[Theorem 1.1]{Mir07b} it  is an even polynomial   in $l_i$ of degree $6g-6+2m$: 
\be \la{V1}
{\rm Vol}_{\rm WP}(\mathcal{M}_S(\L)) = \sum_{d_1+ \ldots + d_m \leq d} {\cal V}_{g, d_1, ..., d_m}l_1^{2d_1} \ldots  l_m^{2d_m}, \ \ \ \ d:= 3g-3+m = \frac{1}{2}{\rm dim}_\R{\cal M}_S(\L).
\ee
Here $d_i\geq 0$ are  integers. Setting $|d|:= d_1+\ldots + d_m$ and $q:=d-|d|$, we have
\be \la{NV}
 {\cal V}_{g, d_1, ..., d_m} \in \pi^{2q}{\Bbb Q}.
\ee
 By \cite{Mir07b},  we have 
\be \la{WCM}
 {\cal V}_{g, d_1, ..., d_m}  = \frac{1}{2^{|d|}|d|! q!}\cdot \int_{\overline {\cal M}_{g,m}}\psi_1^{d_1}\cdots \psi_m^{d_m}\omega^{q}.
\ee
Here $\psi_i$ is the first Chern class of the line bundle on $\overline {\cal M}_{g,m}$ given by the tangent space at the $i-$th puncture on the curve, and $\omega$ is the Weil--Petersson symplectic form on $\overline {\cal M}_{g,m}$.

So  the leading coefficients  are rational numbers - 
the intersection numbers of the $\psi-$classes on $\overline {\cal M}_{g,m}$.

\subsubsection{The first main result.} If $\bS$ has marked points, the volume  is always infinite. In contrast with this we prove  in Section \ref{Sec3}:

\bt \la{MTHO1} For any decorated surface $\bS$, the exponential volume ${\rm Vol}_{{\mathcal{E}}}\mathcal{M}_\bS({\K}, \L)$  is finite. 
\et 

 Theorem \ref{MTHO1} supports the main idea of this paper: 
\vskip 1mm

{\it The moduli spaces ${\cal M}_{\bS}({\K}, \L)$ with the exponential volume forms are the true analogs of  the moduli spaces ${\cal M}_{g,n}$ and  ${\cal M}_S(\L)$ 
with the Weil--Petersson volume forms}. 
\vskip 1mm

For instance the Weil--Petersson volume forms are essential in the string theory \cite{P81} as  the background measures for the 
correlation functions/volume forms   on ${\cal M}_{g,m}$. 
The exponential volume form should play a similar role in the open string theory. 

\subsection{The Laplace transform of exponential volumes and ${\cal B}-$functions} \la{SECTT1.4}

\subsubsection{The Laplace transform of  exponential volumes.}    
 Take a collection of  complex numbers 
 $$ 
  s=(s_1, \ldots, s_m)\in \C^m.   
 $$
Recall the set of parameters $\K$ from  (\ref{KC}). Denote by
 $$
\L= (l_1, \ldots, l_m)\in \R_{+}^m
$$ 
the lengths of all neck geodesics, including the boundary geodesics.   
We denote by $\Omega_\bS(\K)$ the volume form on the space $\mathcal{M}_\bS({\K})$, where $\L$ is not fixed. 
Consider the following  integral:
 \be \la{EIIa}
\begin{split}
{\mathcal{L}}_\bS({\K};  s_1, ..., s_m, \hbar) :=&   
\int_{{\cal M}_\bS}e^{-W/\hbar}\ e^{-(l_1 s_1 + \ldots +l_m s_m)/2}\Omega_\bS({\K}). \\
\end{split}
\ee

  \bl \la{MTHOEE*} For any decorated surface $\bS$,  the   integral  (\ref{EIIa}) is convergent for ${\rm Re}(s_i)\geq 0$, and has an analytic continuation 
  to a meromorphic function in $s$.
 \el
 
 \bpr This is proved in Theorem \ref{MTHOE}. Alternatively, the convergence of the integral for ${\rm Re}(s_i)\geq 0$ follows from the finiteness of the exponential volume, see Theorem \ref{MTHO1}. Then the analytic continuation and its properties follow from 
  the standard properties of the distribution $x_+^\lambda$ on $\R$, applied to $\L_i^{s_i}$ \cite{GF1}. 
 \epr

The function (\ref{EIIa}) generalizes  the  Laplace transform of the classical  volumes  ${\rm Vol}(\mathcal{M}_S(\L))$ in (\ref{V1}). 
 
 Let us note that Mirzakhani's recursions  for these  
  volumes   
   are equivalent 
  to the    Eynard--Orantin    Topological Recursion  for the Laplace transform of the volumes, for a certain spectral curve \cite{EO07}. 
  The complex variables $s_1, \ldots , s_m$ become  the points on the spectral curve.

 \subsubsection{${\cal B}-$function of a decorated surface $\bS$.}  It is  a similar integral:
\be \la{EIa}
\begin{split}
{\mathcal{B}}_\bS({\K};  s_1, ..., s_m; \hbar) :=
&
\int_{{\cal M}^{\circ \circ}_\bS}e^{-W/\hbar}\ e^{-(l_1 s_1 + \ldots +l_m s_m)/2}\Omega_\bS({\K}).\\
\end{split}
\ee 
 The integration is over the $2^{r}:1$ ramified cover  ${\cal M}^{\circ \circ}_\bS$ of $ {\cal M}_\bS$, obtained by specifying 
 an eigenvalue of the monodromy around each crown neck geodesic.   
Their logarithms $l_1, ..., l_r$ can be any real numbers.  So
 $$
(l_1, ..., l_r, l_{r+1}, ..., l_{m}) \in \R^r\times \R_+^{m-r}.
$$

  The ${\cal B}-$function  is the sum of  ${\cal L}-$functions, over all $2^r$ ways to put the signs:
  $$
  {\mathcal{B}}_\bS({\K}; s_1, ..., s_m;\hbar)  = \sum {\mathcal{L}}_\bS({\K}; \pm s_1, ..., \pm s_r,  s_{r+1}, ..., s_m; \hbar).   
  $$ 
Below we often specialize  $\hbar=1$, skipping $\hbar$ from the notation. 
 
\vskip 2mm

   Recall the modified Bessel function of the second kind, which we refer to below as the Bessel function\footnote{The modified Bessel function of the second kind is often denoted by $ K_s(z)$. We have 
    $J_s(z) = 2 K_{-s}(2\sqrt{z})$.}: 
\be \la{f22}
\begin{split}
J_s(z):=&\int_{0}^\infty {\rm exp}\Bigl({-\sqrt{z}(\lambda+\lambda^{-1}} )\Bigr)\lambda^{s} d\log {\lambda}\\
=&z^{-s/2}\cdot \int_{0}^\infty {\rm exp}\Bigl({-t-\frac{z}{t}} \Bigr)t^{s} d\log {t}.\\
\end{split}
\ee
Note that it is an even function in $s$: $J_s(z)=J_{-s}(z)$. \vskip 2mm

There are important special cases of the function ${\cal B}_\bS$, when it reduces to the Bessel function.  

\begin{enumerate}

\item 
  When $\bS= \D_1^*$ is   a punctured disc with one cusp,  we get  the Bessel function  
\be \la{213}
  \mathcal{B}_{\D^*_1}({\K}; s) =  2 J_{s}(\K).
\ee
Indeed, the moduli spaces of enhanced ideal hyperbolic structures on $\D_1^*$ is parametrised by pairs of
real numbers $(\kappa, l)$, where $\kappa $ is the length of the crown geodesic between its intersections with the horocycle, and $l\in \R$ is the signed length of the neck geodesic. We use the exponential coordinates 
\be
\K= e^{-\kappa}, \ \L=e^{l}.
\ee
 The potential function  at the cusp is calculated in Proposition \ref{EXVV} below: 
\be \la{EXVtI*}
\begin{split}
W_{  {\D_1^*} }(\K, l)  = \ &\K^{1/2} (e^{l/2} +e^{-l/2} ).\\
\end{split}
\ee
The volume form $\Omega_{  {\D_1^*}}= d\log \L^{1/2}$. Therefore by the very definition of the   ${\cal B}-$function:
\be \la{FBFa}
  \mathcal{B}_{\D^*_1}({\K}; s) := 2   \int_{-\infty}^\infty {\rm exp}(-\K^{1/2} (e^{l/2} +e^{-l/2} ))e^{-ls/2}d(l/2) =2J_{s}(\K).
\ee
 
 \item  When  $\bS= \D_n^*$ is   a punctured disc with $m$ cusps,  the  ${\cal B}-$function is a product of  Bessel functions:
\be \la{GC}
 \mathcal{B}_{\D^*_n}({\K_1, ..., \K_n}; s) =  2 J_{s}(\K_1)\cdots J_{s}(\K_n). 
\ee
We prove this in Example 1 of Section \ref{SecEVII}.

\end{enumerate}

\subsubsection{The function ${\cal B}_\bS$ via $\A-$model.}  a) According to Givental \cite{G96}, see also \cite{G97}, the Bessel function $J_s(\K)$ solves the quantum differential equation for the ${\Bbb C}^\times-$ equivariant quantum cohomology 
 ${\rm QH}^*_{{\Bbb C}^\times}({\Bbb C}{\rm P}^1)$, where $\C^\times$ acts by preserving  points $\{0, \infty\} \subset \C\P^1$.  Here $\K$ is the K\"ahler parameter, and $s$ the $\C^\times-$equivariant parameter.\footnote{Givental \cite{G96}  addressed the case of the flag variety for ${\rm GL}_k$,  rather than just $\P^1$.}
 \vskip 2mm
 
b)  Givental's theorem, combined with  (\ref{GC}),  tells that the function  $\mathcal{B}_{\D^*_n}({\K_1, ..., \K_n}, s)$  solves 
 the quantum differential equation for the ${\Bbb C}^\times-$equivariant quantum cohomology
  $$
 {\rm QH}^*_{{\Bbb C}^\times}({\Bbb C}{\rm P}^1\times \cdots \times {\Bbb C}{\rm P}^1)
 $$ 
  of the product of $n$ copies of ${\Bbb C}\P^1$. Here $\K_1, ..., \K_n$ are  K\"ahler parameters, and $s$ the equivariant parameter. 
  In particular,  the function ${\cal B}_\bS$ is related to an $\A-$model interpretation. 
   \vskip 2mm

c) If $\bS=S$ is a genus $g$ surface with $m$ punctures,   Mirzakhani's formula  (\ref{WCM}) calculates the volume  in terms of the $\A-$model   on ${\cal M}_{g,m}$.  \vskip 2mm
 
  It would be very interesting to find an $\A-$model interpretation 
 of the function ${\cal B}_\bS$ for the general $\bS$.
 The key problem is that  in general the group ${\rm Mod}(\bS)$ is infinite, while  in the  examples a), b) it is trivial. \\

\subsection{Neck recursion formula} \la{UF} 
Our next goal is a recursion formula which allows us to calculate the exponential volume integrals.  

Let us  cut   $\bS$ along a neck geodesic $\ell$:
 \be
 \bS = {\rm D}_{\ell} \cup \bS_{\ell}; \quad \bS_{\ell}:= \bS-{\rm D}_{\ell}. 
 \ee
 Here  ${\rm D}_{\ell}$ is a decorated surface  given by a punctured disc with $k$ cusps. 
 Denote by $l$ the length of   $\ell$. 
 Since the neck geodesic $\ell$ is a boundary component of both surfaces   ${\rm D}_{\ell}$ and $\bS_{\ell}$, 
  the exponential volumes ${\rm Vol}_{\cal E}({\cal M}_{{{\rm D}}_{\ell}}) $ and  ${\rm Vol}_{\cal E}({\cal M}_{\bS_{\ell}}) $  depend on $l$.  
 
\bt \la{NRF2}
Let $\rC_{\bS, \ell}:=2^{-\mu_\ell} c_{\bS, \ell}$ where $\mu_\ell$ is the number of one-holed tori cut off by $\ell$ and $c_{\bS, \ell}$ is a positive constant depending only on $\bS$ and $\ell$ defined in Proposition \ref{propcc}.
 One has the following {neck recursion formula}:
 \be \la{NRF1}
 \begin{split}
& {\rm Vol}_{\cal E}({\cal M}_\bS)(\K,\L) = \rC_{\bS, \ell} \int^\infty_{0} {\rm Vol}_{\cal E}({\cal M}_{{{\rm D}}_{\ell}})(\K_{\rm D_\ell})  \cdot {\rm Vol}_{\cal E}({\cal M}_{\bS_{\ell}})(\K',\L') \ ldl,\\
\end{split}
 \ee
  \be \la{NRF2*}
 \begin{split}
  & {\cal L}_\bS({\K}; s_1, ..., s_m) = \rC_{\bS, \ell} \int^\infty_{0} {\rm Vol}_{\cal E}({\cal M}_{{{\rm D}}_{\ell}})(\K_{\rm D_\ell})  \cdot {\cal L}_{\bS_{\ell}}(\K'; s_1, ...,  s_{m}) \ l dl.\\
\end{split}
 \ee
 \et

 Theorem \ref{NRF2} is proved in Section \ref{ssscs}.

 Theorem \ref{NRF2}, combined with Mirzakhani's formula (\ref{V1}), has the following two applications.\vskip 2mm

1. It 
allows us to express   
the function ${\cal B}_\bS({\K}; s)$ as a  linear combination with coefficients in the ring \be \la{RR}
 \Q[\pi^2, s_{r+1}^{-1}, ... , s_m^{-1}]
\ee
of  the product of odd derivatives  of Bessel functions, generalizing formula (\ref{GC}). 
Precisely, consider the following  differential operator in $s_1, ..., s_r$, where  
  ${\cal V}_{g, d_1, ..., d_m}\in \Q[\pi^2]$ were defined in (\ref{V1}):
\be \la{DS}
 {\cal D}_\bS:=  \sum_{d_1+ ... + d_m\leq 3g-3+m} 2^{m} \ {\cal V}_{g, d_1, ..., d_m} 
\Bigl(-2\frac{d}{ds_1}\Bigr)^{2d_1+1}\cdots  \Bigl(-2\frac{d}{ds_r}\Bigr)^{2 d_r+1}\cdot  \prod_{j=r+1}^m\frac{(2d_j)!}{s_j^{2d_j+1}}.  
\ee
 Denote by ${\bf K}_j$ the product of the ${\rm K}-$coordinates at the $j-$th boundary component of $\bS$.

 \bt \la{MTHOEE} The function ${\cal B}_\bS$  is obtained by applying  operator (\ref{DS}) to the product of  Bessel functions:
 \be \la{MFOEaaa}
\begin{split}
 &{\cal B}_{\bS}({\K}, s_1, ..., s_r) =   {\rm C}_{\bS} \cdot {\cal D}_{\bS} 
 \ \prod_{j=1}^r 
    J_{s_j}({\bf K}_j), \ \ \ \  {\rm C}_\bS:= \prod_{i=1}^r \rC_{\bS,\ell_i}.\\
\end{split}
 \ee
 \et

 Theorem \ref{MTHOEE} is proved in Section   \ref{SecEVII}.   \vskip 2mm
 
2.  Setting $s_1=...=s_r=0$,  Theorem \ref{MTHOEE} allows us to express the exponential volumes 
 as linear combinations with coefficients in the ring (\ref{RR}) 
   of  odd  derivatives   of   {\it Bessel  like} integrals at $s=0$, see Theorem \ref{MTHO} combined with formula (\ref{UsI}).   
 Each of the terms is manifestly a period of  exponential mixed motive, discussed in Section \ref{ssec1.3}. \vskip 2mm

   Our main result is a different way to express the exponential volume integrals via integrals over the products of the moduli spaces with simpler topology, which we call {\it unfolding formulas.}

 \subsection{Unfolding  exponential volumes} \la{UF}

 The unfolding  starts from a choice of a cusp $p$ on $\bS$, and consists of two steps.
   
  \subsubsection{Reduction to the case when $p$ is single cusp on a crown.} There is a unique up to an isotopy  arc $\pi_p$ from the cusp $p$ to itself, which is homotopic to the 
   boundary component containing  $p$. Cutting  $\bS$ along  it, we get a disjoint union of a polygon $\P_p$ and a decorated surface $\bS_p = \bS-\P_p$:    
$$
   \bS - \pi_p = \P_p \cup \bS_p.
$$
   The surface $\bS_p$  inherits  a single cusp $p$  on the crown ${\rm C}_p$ bounded by $\pi_p$. 
 If the crown of  $\bS$ containing the cusp $p$ has $n$ sides with the $\K-$coordinates $\K_1, ..., \K_n$ assigned to them, the polygon $\P_p$ is an $(n+1)-$gon with the $\K-$coordinates $\K_1, ..., \K_n, \K$. The decorated surface $\bS_p$ has a single boundary interval on the crown ${\rm C}_p$ containing $p$, with the $\K-$coordinate $\K$ assigned to it.

   \bp \la{2.9}
   One has 
   \begin{equation} \la{recurP}
\begin{aligned}
&{\rm Vol}_{\cal E}({\cal M}_\bS)= \int_{\mathbb{R}_+}  
{\rm{Vol}}_{\mathcal{E}}({\cal M}_{  {\P_p}})(\K_1, \ldots , \K_n, \K) \cdot{ \rm{Vol}}_{\mathcal{E}}({\cal M}_{\bS_p})(\K, \ldots)\   d\log \K. \\
\end{aligned}
\end{equation}
   \ep
   
 Proposition \ref{2.9} follows immediately from  
  the cutting and gluing formulas for the exponential volume forms in Sections \ref{subsectioncg} and \ref{cutting}. 
  \vskip 1mm
  
  From now on, we assume that $p$ is a single cusp on a certain crown of $\bS$, denoted by ${\rm C}_p$. 
\subsubsection{Trouser legs and ideal triangles at a  cusp $p$.}

  \begin{figure}[ht]
\centerline{\epsfbox{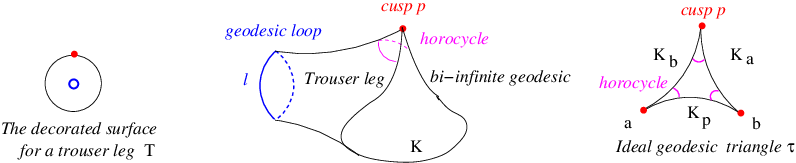}}
\caption{ A {\it trouser leg ${\rm T}$} is a decorated surface given by an annulus  with one special boundary  point. A {\it geodesic trouser leg ${\rm T}$} carries an ideal hyperbolic structure with a geodesic loop $\ell_\T$ and a cusp with a horocycle. The length of $\ell_\T$ is $l$. The {\it horocycle length} of the bi-infinite  geodesic is $\log \K^{-1}$.  An {\it ideal geodesic triangle $\tau$} with horocycles has sides of  {\it horocycle lengths}  $(\log \K^{-1}_a, \log \K^{-1}_b, \log \K^{-1}_p)$. }
\label{sz7}
\end{figure}

 We call a decorated surface $\bS$ is {\it elementary}, if the  moduli space ${\cal M}_\bS(\K, \L)$ is a point. There are three of them: 
a punctured disc $\D_1^*$ with a special point,  a triangle, and a pair of pants. 
We refer to  the first two, embedded as decorated surfaces into a  decorated surface $\bS$,  as {\it trouser legs} and {\it ideal triangles}, respectively. 
 So the {elementary decorated surfaces}  containing the  cusp $p$, see   Figure \ref{sz7}, are:

\begin{enumerate}

\item    {\it Trouser legs} $\T\subset \bS$,    
 containing the cusp $p$.
 
\item {\it Ideal  triangles} $\tau\subset \bS$ and containing the cusp $p$.

\end{enumerate}

 \subsubsection{Unfolding for a crown with a  single cusp $p$.}

  The moduli spaces   ${\cal M}_\tau(\K_a, \K_b, \K_p) $ and   ${\cal M}_{\T}(\K, l)$ of ideal hyperbolic structures on a triangle $\tau(\K_a, \K_b, \K_p)$ and a 
  trouser leg $\T(\K,l)$ with  $\K-$coordinates $(\K_a, \K_b, \K_p)$ and  $(\K, l)$,   see Figure \ref{sz7},  are  points.  
 Yet their exponential volumes   are non-trivial functions  of  $\K-$coordinates denoted by ${\cal E}_\tau$ and ${\cal E}_{\T}$. 
Precisely, let  $W_\tau$ and $W_\T$ be 
 the potentials of a triangle $\tau$ and a trouser leg $\T$. Then by Corollaries \ref{cortri} and \ref{corcrown}
 \be
\la{eqelevol}
\begin{split}
&{\cal E}_{\tau}(\K_a, \K_b, \K_p):=  e^{-W_{\tau}};  \ \ \ \ \ \ \ \ \ \ 
  {\cal E}_{\rm T}(\K, l):=   e^{-W_{\rm T}}. \\
 \end{split}
 \ee

 Next, let   $W_{\tau, p}$  be the partial potential of a triangle $\tau$ at the cusp $p$. 
 For a trouser leg $\T$ with a cusp  $p$, let  $Q_{{\rm T}}(\K, l)$ be  the area of the triangle cut by the horocycle $h_p$ at  $p$, 
 the boundary geodesic $\beta_p$, and the geodesic $\ell_{p\beta^+}$ from $p$ circling around the geodesic boundary loop $\beta = \ell_\T$ of $\T$, see Figure \ref{figure:hp1*}.  
 We have:

\begin{figure}[ht]
	\centering
	\includegraphics[scale=0.3]{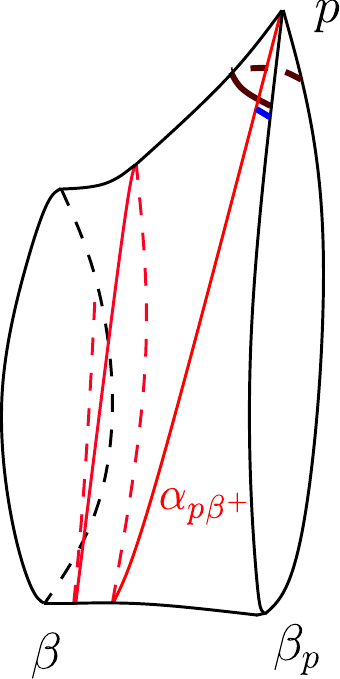}
	\small
	\caption{The trouser leg $\T= \T(\K,l)$ at the cusp $p$ and the  function $Q_{{\rm T}}(\K, l)$.}
	\label{figure:hp1*}
\end{figure}

\bp \la{EXVV} For an ideal triangle $\tau(\K_a, \K_b, \K_p)$ and a trouser leg $\T(\K,l)$ with a cusp $p$ 
 one has 
\be \la{gapfunctI}
\begin{split}
 {W}_{\tau, p}(\K_a, \K_b, \K_p)\ &=  \  \Bigl(\frac{\K_a\K_b}{\K_p}\Bigr)^{1/2}, \\
W_{\tau}(\K_a, \K_b, \K_p)  \ &:= \ {W}_{\tau, p} + {W}_{\tau, a} +{W}_{\tau, b},  \\
Q_{{\rm T}}(\K, l)  \ &=\K^{1/2}e^{-l/2},\\
W_{  {\T} }(\K, l) \ &= \ \K^{1/2} (e^{l/2} +e^{-l/2} ).\\
 \end{split}
\ee
\ep
 For the formula in the first line see   (\ref{FPo}); it goes back to \cite{GS13}, see also \cite[Example 3.16]{GS19}. The formula in the second line is just a definition. The last two formulas are proved in  Lemmas \ref{L5.31} and \ref{L5.3}. \vskip 2mm

Given a cusp $p$, we introduce   {\it recursion kernel functions ${\cal R}_{\zeta, p}$} for the ideal triangles $\zeta=\tau(\K_a, \K_b, \K_p)$ and trouser legs $\zeta ={\rm T}({\K},l)$ containing $p$:  
\be \la{gapfunctII}
\begin{split}
&{\cal R}_{\tau, p}(\K_a, \K_b, \K_p):=  {W}_{\tau, p}\ \ \stackrel{(\ref{gapfunctI})}{=}  \  \ \Bigl(\frac{\K_p}{\K_a\K_b}\Bigr)^{-1/2}, \\\
&{\cal R}_{{\rm T}, p}(\K, l):=  
\left\{ \begin{array}{lll} 
 { \K}^{1/2} (e^{l/2} +e^{-l/2} ), \ \ \  \  \ \ \  \ \  \mbox{if the geodesic loop $\ell_\T \subset  \partial \bS$.}\\ 
2^{1-\mu_{{\rm T}}}{\K}^{1/2}e^{-l/2}, \   \  \ \ \ \    \qquad  \mbox{if  the geodesic loop $\ell_\T\not \subset \partial \bS$. } \\ 
\end{array}\right.\\
 \end{split}
\ee
Here $\mu_{{\rm T}}$ is the number of one-holed tori cut out by a trouser leg $\rm T$. 
\vskip 2mm

A function $f$ on the moduli space ${\cal M}_\bS({\K}, \L)$ is just a ${\rm Mod}(\bS)-$invariant function  on the Teichm\"uller space ${\cal T}_\bS({\K}, \L)$. So it provides a function on Teichmuller  spaces ${\cal T}_{\bS - \tau}$ and ${\cal T}_{\bS - \T}$, also denoted by $f$.

\vskip 2mm

Given an elementary surface $\zeta\subset \bS$ containing the cusp $p$, consider the fibered product of  Teichmuller spaces   over the base ${\cal C}_{\zeta, \bS}$, 
provided by the common boundary components of $\zeta$ and $\bS-\zeta$:
\be \la{FP}
 {\cal T}_{\zeta}\times_{{\cal C}_{\zeta, \bS}} {\cal T}_{\bS-\zeta }.
 \ee
The base ${\cal C}_{\zeta, \bS}$ is parametrised as follows, where  $\ell_\T$ is the geodesic boundary loop of a trouser leg $\T\subset \bS$: 
  
\begin{itemize}

\item  If $\zeta = \tau$:  by the  $\K-$coordinates on internal edges of $\zeta$. 
  
  These are   $\K_a$ if there is one internal edge, and $(\K_a, \K_b)$ if there are two. 
  
\item  If $\zeta =\T$: by the $\K-$coordinate $\K_a$ if  $\ell_\T \subset \partial \bS$, 
and by the $(\K_a, l, \theta)$ if $\ell_\T \not \subset \partial \bS$.

In the latter case,  $(\theta, l)$ are  the  Fenchel--Nielsen coordinates related to the loop $\ell_\T$.
\end{itemize}

Let $c_{\bS,\ell_\T}$ be some positive rational constant depending only on $\bS$ and $\ell_\T$ as in 
 Proposition \ref{propcc}. There is the following  canonical volume form on the base ${\cal C}_{\zeta, \bS}$:
\be \la{formvol}
\begin{split}
& {\rm vol}_{{\cal C}_{\zeta, \bS}} =  \left\{ \begin{array}{lll} 
  d\log \K_a   \ \ \  \  \ \  \ \ \  \ \  \  \  \ \ \    \  \ \ \ \ \ \mbox{if $\zeta = \T$ and  $\ell_\T\subset \partial \bS$,}\\
 c_{\bS,\ell_\T}  \cdot dl \wedge d\theta\wedge d\log \K_a
  \   \     \mbox{if $\zeta = \T$ and   $\ell_\T \not \subset \partial \bS$,} \\ 
    d\log \K_a  \  \ \ \ \  \ \  \  \ \  \ \  \  \ \  \  \ \   \ \ \ \ \mbox{if $\zeta = \tau$ with  one   internal side with coordinate $\K_a$,}\\
   d\log \K_a \wedge d\log \K_b     \ \  \ \  \ \ \ \  \ \mbox{if $\zeta = \tau$ is with  two  internal sides with coordinates $\K_a, \K_b$.}\\
\end{array}\right.\\
\end{split}
\ee

We   use the  shorthand  for the exponential volume form on the space ${\cal M}_\bS({\K}, \L)$  from Definition \ref{DEFMKL}:  
\be \la{OKL}
{\Bbb E}_{\bS}:= {\Bbb E}_{\bS}(\K, \L) := 
 e^{-W_\bS} \ \Omega_\bS({\K}, \L).
\ee
 
The pure mapping class  group ${\rm Mod}(\bS)$ fixes the cusps. So 
 it acts  on 
  isotopy classes of elementary decorated surfaces $\zeta$ containing the cusp $p$. There are finitely many ${\rm Mod}(\bS)-$orbits. Indeed, an orbit is determined by the topology of the surface $\bS-\zeta$, and there are only finitely many topological types of surfaces. 
 Let us consider the following finite set.
  
 \begin{itemize} 
 
 \item The set  $[{\cal H}_{{\zeta}, p}]$  of  ${\rm Mod}(\bS)-$orbits on the set of isotopy classes of   elementary decorated surfaces  ${\zeta}\subset \bS$ containing the cusp $p$.  
 Here if $\zeta$ is an ideal  triangle, we assume that  its  side opposite to $p$  is a boundary interval of $\bS$. We call such triangles the {\it $p-$narrowest triangles}. 
\end{itemize}

Denote by ${\rm Stab}_\zeta$ the stabiliser of the elementary decorated surface $\zeta\subset \bS$ in ${\rm Mod}(\bS)$. Let us set 
\be \la{MZETA}
{\cal M}_{\bS, \zeta}(\K, \L):= {\cal T}_{\bS}(\K, \L)/{\rm Stab}_\zeta.
\ee

A function $f$ on the moduli space ${\cal M}_\bS({\K}, \L)$ is  a ${\rm Mod}(\bS)-$invariant function  on the Teichm\"uller space ${\cal T}_\bS({\K}, \L)$. So it provides a function on the fibered product (\ref{FP}),  also denoted by $f$.

\begin{theorem} \la{T5.5I}
For any  decorated surface $\bS$,  a crown with  a single cusp $p$, and any smooth function $f$ on the space ${\cal M}_\bS({\K}, \L)$,  we have the following recursion formula:
\begin{equation} \la{recur**}
\begin{aligned}
& \int _{{\cal M}_\bS({\K}, \L)}f  \ W_p\  {\Bbb E}_\bS \ = \ \sum_{[{\zeta}]\in [ {\cal H}_{{\zeta},p}]}   
 \int_{  {\cal M}_{\bS, \zeta}(\K, \L)}   f\cdot      {\cal R}_{{\zeta}, p}\  {\cal E}_{\zeta}
   \cdot \ {\Bbb E}_{\bS-{\zeta}}\wedge {\rm vol}_{{\cal C}_{\zeta, \bS}}.\\
  \end{aligned}
\end{equation}
\end{theorem}
Here we use the two projections from (\ref{MZETA}) onto  ${\cal M}_{\zeta}$ and ${\cal M}_{\bS-\zeta }$ to pull back the function    
${\cal R}_{{\zeta}, p} \ {\cal E}_{\zeta}$ and the form $ {\Bbb E}_{\bS-{\zeta}}$, respectively.
  We elaborate  formula (\ref{recur**}) in Theorem \ref{T5.5}.

\subsubsection{Comments on the unfolding formula and its structure.}   
\begin{enumerate}

\item  The right hand side of  formula (\ref{recur**}) is  obtained as follows. Cut out  from $\bS$ an {elementary decorated surface $\zeta$} containing the cusp $p$. Multiply $f$ by 
the exponential volume function ${\cal E}_\zeta$ and by the {recursion kernel} function ${\cal R}_{\zeta, p}$, provided by the McShane identity.   Multiply the resulting function by the exponential volume form ${\Bbb E}_{\bS-\zeta}$. Induce the obtained  form to the  fiber product (\ref{FP}), multiply by the form ${\rm vol}_{{\cal C}_{\zeta, \bS}}$ lifted from the base 
${\cal C}_{\zeta, \bS}$, and 
 integrate the resulting volume form.

\item All but one term of the right hand side of  formula (\ref{recur**}) are topologically simpler than the left hand side. Let us describe the exceptional term. 
Given a crown with a single cusp $p$, let us cut the surface along the neck geodesic loop $\ell_p$ 
 around  this cusp. The obtained surface has two components. The one containing the cusp $p$ is called the {\it neck trouser leg},  and denoted by $\T_{\ell_p}$.  
 It makes sense to subtract the term corresponding to  the {neck trouser leg}  $\T_{\ell_p}$. The resulting formula, called the {\it reduced unfolding formula},  is presented in Section \ref{Sect5.2}.

 \item When $\bS=S$  has no cusps, Mirzakhani  \cite{Mir07a} proved recursion formulas for the volumes of moduli spaces ${\cal M}_S(\L)$ of hyperbolic surfaces 
   using  McShane identities  and   a variant of unfolding.  
  Theorem \ref{T5.5I} can be viewed as a generalization of  Mirzakhani's recursions  for the volumes. Indeed:

 \begin{itemize}

\item  Mirzakhani's recursion is a sum over all topological types of embedded into  $S$ {\it pairs of pants} 
 containing a given boundary circle.
 \item Recursion (\ref{recur**}) is a sum over all topological types of   {\it  ideal triangles} and {\it trouser legs} in $\bS$   containing the cusp $p$.  
 
 \end{itemize}

So in both cases we sum over all topological types of elementary surfaces  containing either the given boundary circle, or a given cusp.

\item  If $\zeta = \T$ is a trouser leg with $\ell_\T \not \subset \partial \bS$, and the function $f$ does not depend on the angle parameter $\theta$ of $\ell_\T$, we can integrate $d\theta$, getting  the 2-form ${\rm vol}_{{\cal C}_{\zeta, \bS}} = ldl\wedge d\log \K$. In particular, this is so when $f$ is a function of the length of the neck geodesic at the cusp $p$, 
e.g. is the ${\cal B}-$function of $\bS$.  

\item  If $\zeta$ is an ideal triangle, the sum in  formula (\ref{recur**}) is over $p-$narrowest ideal triangles with the vertex  $p$. 
There are two options: 

i) The cusp $p$ and the opposite  side $ab$ lie on different   boundary components, see Figure \ref{figure:cuttri}. In this case  the number of connected components of $\bS$ does not change. 

 \begin{figure}[ht]
	\centering
	\includegraphics[scale=0.3]{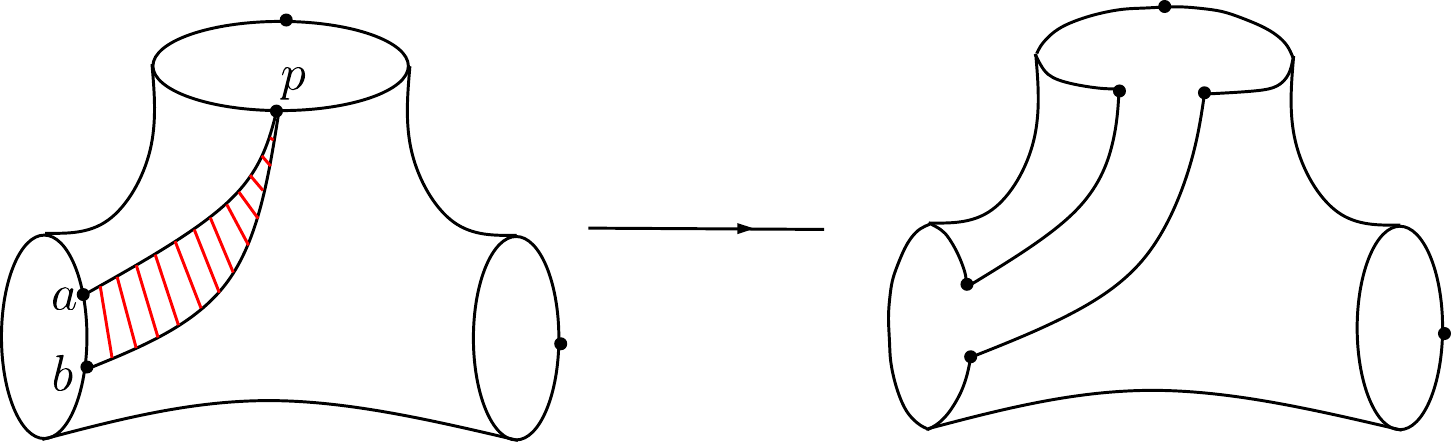}
	\small
	\caption{Cutting an ideal triangle  with the cusp $p$ on one boundary component, and the side $ab$ on the other does not change the number of components of $\bS$.}
	\label{figure:cuttri}
\end{figure}

ii) The cusp $p$ is on the same boundary component as the  side $ab$. 
In fact the  side $ab$ can have both of its vertices at $p$, see  the right picture on Figure 
\ref{figure:crpp}. The left picture helps to imagine such a cut. 
In this case we increase the number of connected components of $\bS$ by one.  

 \begin{figure}[ht]
	\centering
	\includegraphics[scale=0.3]{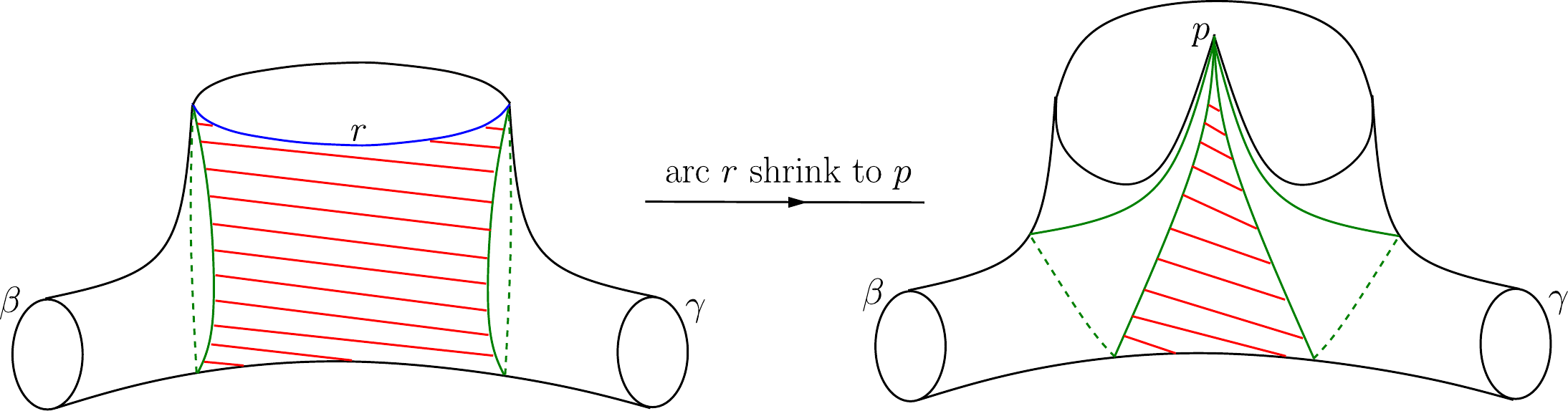}
	\small
	\caption{Cutting an ideal triangle  with the cusp $p$ and the side $ab$ on the same  boundary component. This  increases the number of connected components of $\bS$ by one.}
	\label{figure:crpp}
\end{figure}

\item Theorem \ref{T5.5I}  is proved in Section \ref{Sec4}.  The proof   consists of  
two  ingredients:
 \begin{itemize}

\item  Generalized McShane identities \cite{McS91}, \cite{McS98} for ideal hyperbolic surfaces, see Section \ref{Sec4.1}.   

\item The  {\it factorization property} of exponential volume forms, see also (\ref{FEV})  for the elaborated form:
\be
\begin{split}
&i_\zeta^*({\Bbb E}_\bS) ={\cal E}_\zeta \cdot {\Bbb E}_{\bS-\zeta} \wedge {\rm vol}_{{\cal C}_{\zeta, \bS}}.\\
\end{split}
\ee

The additivity of  the potential $W_\bS$ under the cutting of  $\bS$ 
 is crucial  
 for  both of them. 
 \end{itemize}

\item  Unfolding formulas (\ref{recur**})   are of independent interest on their own: \vskip 2mm

\begin{itemize}

\item 
They show effectively that the exponential volumes are functions of algebraic geometric origin -   periods of variations of exponential motives, see Section \ref{ssec1.3}. \vskip 2mm

\item  Mirzakhani's recursion  for the 
  volumes ${\rm Vol}({\cal M}_S)(\L)$ is  a primary example of the Topological Recursion  \cite{EO07}. 
   Its combinatorial skeleton    is the same:   the sum over all topologically different embedded pairs of pants containing a  given boundary circle.

One should have an open string analog of  Topological Recursion, with the combinatorial skeleton given by the sum over 
topological types of ideal triangles/trouser legs containing a  cusp.  
\end{itemize}

\item  The surfaces $\bS - \T$ and $\bS-\tau$ 
can have one or two  components. At least one of them has cusps. We pick a cusp to perform unfolding. If one of the components does not have cusps, it has a boundary circle, and we perform 
Mirzakhani's recursion  at this circle. Keep doing this, we   decompose  $\bS$ into a finite collection of triangles $\tau_i$, trouser legs $\T_j$ and pairs of pants ${\cal P}_k$, glued  
according to a gluing pattern $\gamma$, which tells which pairs of sides/boundary loops  have to be glued. Schematically, 
$$
\bS = \tau_1 \cup_\gamma \cdots \cup_\gamma \tau_a \cup_\gamma \T_1 \cup_\gamma \cdots \cup_\gamma \T_b \cup {\cal P}_1 \cup_\gamma \cdots \cup_\gamma{\cal P}_c. 
$$
This allows us to present the moduli space ${\cal M}_\bS$ as a fibered product of the elementary ones 
\be \la{DEC}
{\cal M}_\bS = {\cal M}_{\tau_1} \ast \cdots \ast {\cal M}_{\tau_a} \ast {\cal M}_{\T_1} \ast \cdots \ast {\cal M}_{\T_b} \ast 
{\cal M}_{{\cal P}_1} \ast \cdots \ast {\cal M}_{{\cal P}_c}. 
\ee
If $f$ is a regular function, then, by complexifying all factors in (\ref{DEC}), each of the resulting integrals is on the nose an exponential period, as  explained in Section \ref{ssec1.3}.

    \item {\it Recursion for a crown with a  single cusp $p$.} 
Formula  (\ref{recur**})  gives  an unfolding of the integral of $f{\Bbb E}_\bS$, 
 getting  factors  $f/W_p$ on the right. 
 If $f=1$, we get an unfolding for the exponential volume of $\bS$.  
 However, unlike Mirzakhani's recursion,   it      has a factor $1/W_p$ which 
  does not  factorise into a product of the ones lifted 
  from ${\cal M}_{\zeta}$ and ${\cal M}_{\bS-\zeta}$.

  Here is how we can treat this problem by  introducing parameters $\alpha_p$ at the potentials at the cusps $p$. 
 We  introduce a modification of the potential  depending on  parameters $\alpha_p\in \R_{>0}$ at the cusps $p$:
\be
\begin{split}
&\widetilde W_\bS:= \sum_p \alpha_pW_p, \ \ \ \widetilde {\Bbb E}_\bS:= e^{-\widetilde W_\bS}\Omega_\bS.\\
\end{split}
\ee
Then  
\be
-\frac{d}{d\alpha_p}\int _{{\cal M}_\bS({\K}, \L)}f \ \widetilde {\Bbb E}_\bS=   \int _{{\cal M}_\bS({\K}, \L)}f \ W_p \widetilde {\Bbb E}_\bS.
\ee

Therefore Theorem \ref{T5.5I} implies immediately the following.
\begin{theorem} \la{recur****}
Under the same assumptions as in Theorem \ref{T5.5I},  we have the following recursion:
\begin{equation} \la{recur**!}
 -\frac{d}{d\alpha_p}\int _{{\cal M}_\bS({\K}, \L)}f \ \widetilde {\Bbb E}_\bS\ = \  \mbox{\rm $\alpha_p-$modified right hand side of } \ (\ref{recur**}).
\end{equation}
Here on the right hand side   we use everywhere the modified exponential factors $e^{-\widetilde W_*}$. 

Note  
that for the cusps $p', p''$  on $\bS-\zeta$ matching the cusp $p$ on $\bS$ we have $\alpha_{p'}=\alpha_{p''}$. 
\end{theorem}
The integration over $\alpha_p$ recovers the exponential volume since,  
  as $\alpha_p\to \infty$ it exponentially decays:
\be \la{8/31/24}
-\int_{\alpha_p}^\infty \Bigl(\frac{d}{d\alpha_p}\int _{{\cal M}_\bS({\K}, \L)}f \ \widetilde {\Bbb E}_\bS\Bigr)  d\alpha_p = \int _{{\cal M}_\bS({\K}, \L)}f \ \widetilde {\Bbb E}_\bS.
\ee
So  using  recursion formula (\ref{recur**!}),   and  integrating over $\alpha_p$ using (\ref{8/31/24}), 
  we get  a recursion for the exponential volumes, 
  where the right hand side does factorise.

 Note that recursion kernels ${\cal R}_{\zeta, p}$ for all $\zeta$ but  the  one $\zeta = \T$ with $\ell_\T\subset \partial \bS$ are the local potentials of 
the elementary surface $\zeta$ at the cusp $p$, see (\ref{gapfunctII}).  
 Recursion kernels depending on the parameters $\alpha_p$ are   derivatives of the exponential volume functions for the elementary surface:
\be
\begin{split}
  &( e^{-\widetilde W_\T }Q_{\T})(\alpha_p)   =\left(-\frac{1}{2} \frac{d  }{d\alpha_p} +\frac{1}{\alpha_p}\frac{d}{dl}\right) e^{-\widetilde W_\T }(\alpha_p); \\
  &( e^{-\widetilde W_\zeta }W_{\zeta, p})(\alpha_p)   = -\frac{d  }{d\alpha_p}e^{-\widetilde W_\zeta }(\alpha_p). \\
  \end{split}
    \ee 
 
    \end{enumerate} 
    
    \subsection{Exponential volumes and   positive  Whittaker--Hecke   algebra for ${\rm SL}_2(\R)$}

 Exponential volumes of  elementary decorated surfaces with cusps -  that is a triangle $\tau$ and a trouser leg $\T$ -  are non-trivial functions ${\cal E}_\tau(\K_1, \K_2, \K_3)$ and ${\cal E}_\T(\K, \L)$. 
 We define an algebra ${\cal E}$,  consisting of functions $f$ on  $\R_{>0}$ with exponential decay at infinity,  
with the product    given by   
\be \la{PFE9xx}
(f\ast g)(\K_3):= \int_{\R_{>0}\times \R_{>0}} {\cal E}_\tau(\K_1, \K_2, \K_3)f(\K_1)g(\K_2) d\log\K_1 d\log \K_2.
\ee
    The product  is evidently commutative. Its  associativity is non-trivial, and will be discussed momentarily. 
    
    The spectrum of the algebra ${\cal E}$ is described as follows. The Mellin transform of the function ${\cal E}_\T(\K, \L)$ is the Bessel function ${\cal B}_\T(\K, s) = \int_{\R_{>0}} {\cal E}_\T(\K, \L)\L^{s/2}d\log \L $.      
 We prove that it  satisfies the  product formula 
\be \la{PF10xx}
 {\cal B}_\T(\K_1, s)\cdot {\cal B}_\T(\K_2, s) = \int_{\R_{>0}} {\cal E}_\tau(\K_1, \K_2, \K){\cal B}_\T(\K, s)d\log \K.
\ee   
 It implies that the function ${\cal B}_\T(\K, s)$ provides a homomorphism $\psi_s$ of the algebras $({\cal E}, \ast)$ to $\C$:
\be \la{PF13xx}
\psi_{s}(f):=  \int_{\R_{>0}} f(\K){\cal B}_\T(\K, s)d\log \K.
\ee

 Let $\bS$ be a decorated surface with trivial mapping class group ${\rm Mod}(\bS)$, that is either  $\D_m^*$ or a  polygon $\P_n$. Cutting $\bS$ into a collection  $\{\bS_\alpha\}$ of smaller decorated surfaces  
 gives rise to a formula expressing the exponential volume of the moduli space  ${\cal M}_\bS(\K, \L)$ as an integral of the product of exponential volumes of ${\cal M}_{\bS_\alpha}(\K', \L')$.  
 Cutting $\bS$ into two different collections $\{\bS_\alpha\}$ and$\{\bS_\beta\}$ provides an identity between the integrals. 
 The associativity of the product  and formula (\ref{PF10xx}) are examples, see Figures \ref{RT12} and \ref{RT10a}. 
   \begin{figure}[ht]
\centerline{\epsfbox{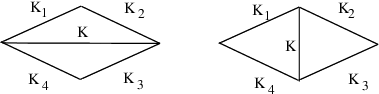}}
\caption{The associativity of the $\ast-$product results from two ways to cut a rectangle.}
\label{RT12}
\end{figure}

 
    \begin{figure}[ht]
\centerline{\epsfbox{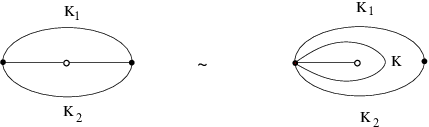}}
\caption{The product formula (\ref{PF10xx})  results from cutting  $\D_2^*$ in two ways.}
\label{RT10a}
\end{figure} 

 Identities between  integral formulas resulting    from  cutting  a general $\bS$   follow formally from this.

Formulas (\ref{PFE9xx}) - (\ref{PF13xx})
are similar  to   formulas in the classical theory of  spherical functions \cite{Ge50}.   
The Hecke algebra of spherical functions is 
given by  compactly supported functions on ${\rm SL}_2(\R)$,  invariant under the left and right actions of  ${\rm SO}(2)$, with the convolution product. 
It is a commutative algebra, and its characters are described by {\it zonal spherical functions} similarly to (\ref{PF13xx}).

In our case {\it Whittaker functions}  on  ${\rm SL}_2(\R)$ play the role of spherical functions,  covariant under the left and right action of the unipotent subgroup 
$\N(\R)\subset {\rm SL}_2(\R)$ by a given non-trivial character. 
The Bessel function  ${\cal B}_\T(\K, s)$ is the restriction of the zonal Whittaker function for the principal series representation $V_s$ of ${\rm PGL}_2(\R)$ to the Cartan subgroup. 
The key  difference is that  the convolution of such functions is ill-defined due to non-compactness of   ${\rm N}(\R)$. Remarkably, one can still define the product   
for   restrictions of  Whittaker functions to the positive diagonal matrices by formula (\ref{PFE9xx}). This suggests the name {\it positive Whittaker--Hecke algebra} for the algebra ${\cal E}$. 
We will continue the discussion in Section \ref{WHalgebra}.

    \subsection{Exponential volumes are exponential periods} \la{ssec1.3}
  
We show that the exponential volumes of moduli spaces are functions of algebraic-geometric origin. 
Here the space ${\cal P}_\bS$ is essential, providing  the complexification of the Teichm\"uller space, its volume form, and the potential $W$. \vskip 1mm

We start with a data
\be \la{DEM}
(X, W, f, \Omega; \gamma), 
\ee
where $X$ is a regular 
  variety  over $\Bbb Q$,  $W$ and $f$ is a regular function on $X$,  $\Omega_X$  is an algebraic volume form  on $X$ with logarithmic singularities,  and $\gamma$ is a possibly non-compact cycle  of real dimension ${\rm dim}(X)$ such that the map $W:\gamma \to \C$ is proper,  
and  ${\rm Re} \ W\to +\infty$. Then we can consider  the following integral
 \be \la{ETI}
 \int_{\gamma\subset X(\C)}e^{-W} f \ \Omega_X.
\ee
The conditions on the cycle $\gamma$  guarantee 
 that the integral is convergent: the integrand exponentially decays at infinity of the cycle $\gamma$. 
The numbers which we can get this way are called the {\it exponential periods over $\Q$}. If the data  (\ref{DEM}) depends on parameters $\K$ parametrised by a  variety ${\cal K}$, 
then the functions in $\K$ given by integrals (\ref{ETI}) are called the {\it periods of   variations of exponential motives}. 
\vskip 2mm

The exponential volume function ${\rm Vol}_{{\mathcal{E}}}\mathcal{M}_\bS({\K}, \L)$ is defined by an apparently similar data 
$$
({\cal P}_{\bS},  W,  1,  \Omega_\bS; \gamma={\cal T}^{\rm en}_\bS).
$$
However this data is ${\rm Mod}(\bS)-$invariant.  If the   group ${\rm Mod}(\bS)$ is finite, 
 it defines an exponential period on the nose. Otherwise the integral  diverges. So  we have to integrate over a fundamental domain for the action 
of  ${\rm Mod}(\bS)$ on the Teichm\"uller space.    
Yet there is no natural  choice of the fundamental domain. 
\vskip 2mm

Here is a classical analogy. It is  known that for the standard invariant volume form $\omega$ on ${\rm SL}_n$ we have
\be \la{ZV}
{\rm Volume}_\omega\Bigl({\rm SL}_n(\R)/{\rm SL}_n(\Z)\Bigr) := \int_{{\rm SL}_n(\R)/{\rm SL}_n(\Z)}\omega  = \zeta(2) \zeta(3) \cdots \zeta(n).
\ee
However it is not straightforward to give an algebraic-geometric interpretation of the volume. 
Yet  its value is  manifestly a mixed Tate 
period over ${\rm Spec}(\Z)$. 

\bt \la{MTHO11}The exponential volume  ${\rm Vol}_{{\mathcal{E}}}\mathcal{M}_\bS({\K}, \L)$  is a period of a variation of  exponential motives. 
\et

We give two proofs of Theorem \ref{MTHO11}, each providing an explicit way to write the exponential volume as a finite sum of 
exponential periods. 

\begin{enumerate}

\item  We  calculate the   ${\rm Vol}_{{\mathcal{E}}}\mathcal{M}_\bS({\K}, \L)$   
using  {neck  recursion formula} (\ref{NRF1})  and  Mirzakhani's formula (\ref{V1}). 

\item We  apply  {unfolding formulas} from Theorem \ref{T5.5I} to calculate inductively 
 the exponential volume  ${\rm Vol}_{{\mathcal{E}}}\mathcal{M}_\bS({\K}, \L)$ as a finite sum of  integrals of type  (\ref{ETI}). 

\end{enumerate}

Both  proofs are based on different variants of unfolding in the Teichm\"uller theory.

 \vskip 1mm

Finally, we want to  note  the  analogy between the Rankin-Selberg 
method in Number Theory and the unfolding  in the Teichm\"uller theory.

  \subsubsection{Organization of the paper.}  In Section \ref{Sec2} we prove   the crucial   
   cutting and gluing formulas for unfolding exponential volume forms. 
   
In Section \ref{SECT2.1}  we discuss  the moduli space ${\cal P}_{\bS}$ -  the algebraic-geometric avatar of the Teichm\"uller space of ideal hyperbolic structures on $\bS$, and recall  the cluster Poisson coordinates 
   $\{\B_\F, X_{\rm E}\}$ on it assigned to an ideal triangulation of $\bS$. Then we define  
   the local potentials ${W}_p$ and  the regular functions ${\rm K}_\F$  assigned to the boundary intervals of $\bS$, and calculate them in the cluster Poisson 
   coordinates. Restricting the functions ${\rm K}_\F$ to the positive locus we recover  functions (\ref{KF}). After these preparations,  we prove in Section \ref{subsectioncg} the 
   cutting and gluing formulas for the exponential volume forms. 
   
 In Section \ref{SecEVII} we calculate  ${\cal B}-$functions and 
  prove Theorem \ref{MTHO1} - the exponential volume is finite.

    In Section \ref{Sec4.1} we prove Proposition \ref{EXVV} and  use it to get the generalized McShane identity. 
       
   In Section \ref{Sec4} we prove Theorem \ref{T5.5I} - the unfolding formula for the exponential volumes.

  In Section  \ref{tropical} we study the tropicalization of exponential volumes, and prove that they are   finite. 

In Section \ref{sect2.4} we show that tropical volumes of spaces of measured laminations on  punctured surfaces  are equal to Kontsevich's volumes, and so carry the same  information as the intersection theory on ${\cal M}_{g,n}$. 

In Section \ref{SECT6} we elaborate  the simplest examples of the unfolding formula. 

In Section \ref{WHalgebra} we show that exponential volumes for the elementary decorated surfaces with cusps describe a commutative associative algebra,  which we call  {\it positive Whittaker--Hecke algebra} for 
${\rm PGL}_2(\R)$. 

\subsubsection{\bf Acknowledgement.} We are very grateful to the referee for the terrific job. Essentially all referee's comments and remarks are  incorporated in the paper. 
This work was done at IHES during the Summer of 2023, and the final draft  prepared  during the Summer of 2024. 
We are grateful to IHES for the  hospitality and support. The work of AG was supported by  the NSF grants DMS-1900743 and DMS-2153059, and by 
 the Gretchen and Barry Mazur Chair at IHES and the Simons Foundation fellowship in 2023.

\section{Moduli spaces $\mathcal{P}_{\bS}$ and  Teichm\"uller spaces} \la{Sec2}

Given a group $\G$, recall  the canonical equivalence of categories:
$$
\{\mbox{{\em $\G$-local systems} ${\cal L}$ on $S$}\} \lra \{\mbox{representations $\rho: \pi_1(S,x) \to \G$  modulo the $\G-$conjugation}\}.
$$
It assigns to a $\G-$local system ${\cal L}$ on $S$ the monodromy representation $\rho: \pi_1(S, x) \lra \G$.\vskip 2mm

Now let  $\G= {\rm PGL}_2$, and  ${\cal B} :=\P^1$. 
Let   $\U$ be a maximal unipotent subgroup of $\G$, and 
${\cal A}:=\G/\U$ the decorated flag variety, 
which can be described as   
$$
{\cal A} = \Bigr((V_2-\{0\}) \times ({\rm det} V_2^*-\{0\})\Bigr) /{\Bbb G}_m.
$$
Here $V_2$ is a  two-dimensional  vector space.  
So the moduli space ${\cal A}$ parametrises pairs $(v, \omega)$ where $v$ is a non-zero vector, and $\omega$ an area form in $V_2$, 
 considered modulo the action of the multiplicative group ${\Bbb G}_m$:
 $$
 (v, \omega) \lra (\lambda v, \lambda^{-2}\omega). 
 $$
 There is a canonical function 
 $$
 \Delta: {\cal A}   \times {\cal A}    \lra \A^1, \qquad  (v_1, \omega_1) \times (v_2, \omega_2) \lms \omega_1(v_1, v_2)\omega_2(v_1, v_2).
 $$
 Given a ${\rm PGL}_2-$bundle ${\cal L}$, let  ${\cal L}_{\P^1}$ be the associated local system of projective lines, and   ${\cal L}_{\cal A}$ the associated local system of two-dimensional  vector bundles with area forms, modulo the action of the group ${\Bbb G}_m$.

\begin{definition}  \la{DEFF3.1a} \cite{FG03a}. Let $\bS$ be a decorated surface. Let ${\cal L}$ be a ${\rm PGL}_2$-local system on $\bS$.

A {\em framing} ${\cal L}$ at a puncture  on $\bS$ is a flat section of the associated local system ${\cal L}_{\P^1}$  near the puncture.\vskip 1mm

The moduli space $\mathcal{X}_{{\rm PGL}_2, \bS} = \mathcal{X}_{\bS} $ parametrises pairs $(\cal L, \beta)$,  with   a framing $\beta$  at each puncture.

\ed

Here are better versions of the moduli space  $\mathcal{X}_{\bS}$  for decorated surfaces $\bS$ with marked boundary points, introduced and studied in  \cite{GS13} and  \cite{GS19}. 
Their key advantage is the existence of the gluing maps. 

\bd \la{DEFF3.1}  Let $\bS$ be a decorated surface. Let ${\cal L}$ be a ${\rm PGL}_2$-local system on $\bS$.

A {\em decoration} on ${\cal L}$ at a marked point $p$ on $\bS$ is a flat section of the associated local system ${\cal L}_{\cal A}$ near $p$.
We assume that the decorations at each pair of adjacent marked points are in generic position. \vskip 1mm

The moduli space ${\cal P}_{{\rm PGL}_2, \bS} = \mathcal{P}_{\bS} $   parametrises triples $(\cal L, \alpha, \beta)$ where $\cal L$ is a ${\rm PGL}_2$-local system on $\bS$ with a decoration  $\alpha$  at each marked point, and  a framing $\beta$  at each puncture.

The space $\mathrm{Loc}_{{\rm PGL}_2,\bS} = \mathrm{Loc}_{\bS}$ parametrises pairs $({\cal L}, {\alpha})$,  with a decoration $\alpha$ at each marked point.

\end{definition}

So  the moduli space $\mathcal{P}_{\bS}$  parametrises  local systems  of  two-dimensional vector spaces on $\bS$ with non-zero volume forms, modulo the action of ${\Bbb G}_m$, equipped with  a flat section of the local system of lines  near each puncture,  and a pair $(v_p,\omega_p)$   at the fiber at each marked point $p$, considered up to a common  simultaneous rescaling in all fibers. The vectors $v_{p}, v_{p'}$ at adjacent marked boundary points $p, p'$ are in the generic position:  their parallel transports to the middle of the segment $p p'$ are not collinear. When the
boundary component has a single marked point $p$, the vector $v_{p}$ is  in generic  position means that monodromy around the boundary component transforms  $v_p$  to a vector $v_p'$ which is not collinear to $v_p$.

 Just like the space $\mathcal{X}_{\bS}$, the space 
 ${\cal P}_\bS$  is equipped with the action of  the group ${\rm Mod}(\bS)$, and  carries a canonical ${\rm Mod}(\bS)-$equivariant  cluster Poisson  structure, given by an infinite collection of rational coordinate systems, related by cluster Poisson transformations, which we review in Section \ref{SECT2.1}.

\subsection{Cluster Poisson coordinates, local potentials, and the cluster volume form  on  $\mathcal{P}_{\bS}$} \la{SECT2.1}

\subsubsection{Cluster Poisson coordinates on the space $\mathcal{P}_{\bS}$.} Given any local system on $\bS$, its fiber over any simply-connected domain makes sense since the fibers at any two points of the domain are canonically isomorphic by parallel transport via a path in the domain connecting the points.  In particular, we can talk about the fibers of the local system of projective lines ${\cal L}_{\P^1}$ over an edge, or over a rectangle of a triangulation of $\bS$

Pick a point $(\cal L, \alpha, \beta)\in \mathcal{P}_{\bS}$. 
Take a   boundary interval $\F$ with marked points $p, p'$ at the ends. The decorations   at the points $p, p'$  can be viewed as decorations  $(v_p, \omega_p), (v_{p'}, \omega_{p'})$ at the fiber  of ${\cal L}$ at  $\F$. We define   the  {\it pinning point} 
$p_{\F}$  at the   fiber of  the local system   ${\cal L}_{\P^1}$ at $\F$ as the projectivisation of the one dimensional subspace spanned by the 
following vector:
\be \la{pli}
p_{\F}:= \left\langle \omega_p(v_p, v_{p'})v_p +  \omega_p(v_{p'}, v_{p})v_{p'} \right \rangle.
\ee
So a point $(\cal L, \alpha, \beta)\in \mathcal{P}_{\bS}$ provides the following  distinguished points in the fibers of 
    ${\cal L}_{\P^1}$:

\begin{enumerate}

\item A decoration point  at the fiber over the marked point $p$.

\item A pinning point  at the fiber over the boundary edge $\F$.

\item A framing point  at the fiber over a  point near a puncture $o$, invariant under the local monodromy. 

\end{enumerate}

Pick an ideal triangulation ${\cal T}$ of  $\bS$, that is a triangulation with the vertices at the marked points and punctures. 
It includes the collection of {\it external edges} given by the boundary segments of $\bS$. The rest of the edges are  the {\it internal edges}. Each edge $\rm A$ of the triangulation gives rise to a rational function
$$
\X_{\rm A}: \mathcal{P}_{\bS} \lra \A^1.
$$
It is defined as follows. 
If $\rm E$ is an internal edge, there is a unique quadrilateral $Q_{\rm E}$ containing $\rm E$ as a diagonal. Let us denote its vertices by $z_1, ..., z_4$, so that $z_1$ is a vertex of the edge $\rm E$, and the order of the points follows the orientation of the quadrilateral induced by the orientation of $\bS$. Then we set 
$$
\X_{\rm E}:= r(x_1,x_2,x_3, x_4):= \frac{\omega(\widetilde x_1, \widetilde x_2)\omega(\widetilde x_3, \widetilde x_4)}{\omega(\widetilde x_2, \widetilde x_3)\omega(\widetilde x_1, \widetilde  x_4)}.
$$
where $x_i$ is the distinguished point in the fiber of the local system ${\cal L}_{\P^1}$ at  the point $z_i$. To define the cross-ratio $r(x_1,x_2,x_3,x_4)$,  we parallel transform the   points $x_i$  to a center of the quadrilateral, and pick 
 arbitrary  non-zero vectors $\widetilde x_i$  in the fiber projecting to  the points $x_i$, and use any area form $\omega$ in the fiber. 
  Then $\X_{\rm E}$  evidently does not depend on the choices. 
  
  If the edge $\rm A$ is an external edge $\F$, there is a unique triangle $t_\F$ of the triangulation with the base $\F$. Then there is a quadruple of distinguished points 
  $(x_1, x_2, x_3, x_4)$,  where $(x_1, x_3)$ are the decoration points at the vertices of $\F$, counted counterclockwise,  $x_2$ is the pinning point at the edge $\F$, and $x_4$ is the framing/decoration point over the vertex of the triangle $t$ which does not belong to $\F$. We define $\X_\F$ as the cross-ratio 
  $$
  \X_\F:=r(x_1, x_2, x_3, x_4).
  $$
 
 \bd The boundary Poisson coordinates $\B_{\F}$ are assigned to the  boundary edges $\F$  and given by 
 $$
 \B_{\F}:= \X_\F. 
 $$
 \ed
  \bt \la{CPS}
The collection of rational functions $\{\X_{\rm E}, \X_\F\}$ assigned to the edges of an ideal triangulation of $\bS$ provides a cluster Poisson structure on the space 
$\mathcal{P}_{\bS}$. The Poisson bracket is given by 
\be \la{FEPS}
\{\X_{\rm A}, \X_{\rm B}\}= \varepsilon_{\rm A B}\X_{\rm A} \X_{\rm B}, \qquad \varepsilon_{{\rm A}\rm B} = -\varepsilon_{\rm B {\rm A}}\in \Z.
\ee
\et

The Poisson tensor   $\varepsilon_{{\rm A} \B}$ is defined by 
$
\varepsilon_{{\rm A} \B} := \sum_v \delta_{v}( {\rm A}, \B).
$ 
The sum is over common vertices $v$ of the  edges ${\rm A}, \B$. We set $\delta_v({\rm A} ,\B)=0$ unless the edges ${\rm A}, \B$ are adjacent. In the latter case,   $ \delta_v({\rm A}, \B) =1$ 
 if $\B$ is after ${\rm A}$ at the vertex $v$ following the (clockwise on the pictures) orientation of $\bS$, and $-1$ otherwise. Theorem \ref{CPS} is proved in Subsection 2.6 below. \\
 
 Besides  cluster Poisson coordinates, there are  canonical functions ${\rm K}_\F$ at the boundary edges $\F$ of $\bS$. 

 \bd \la{OMD}
 Let  $(\omega_i, v_i)$ and $(\omega_{i+1}, v_{i+1})$ be the decorations at the ends of a boundary edge $\F$.  Then the  function  ${\rm K}_\F$  assigned to the  boundary edges $\F$ of $\bS$ is given by
 $$
 {\rm K}_\F :=( \omega_i(v_i, v_{i+1})\omega_{i+1}(v_i, v_{i+1}))^{-1}.
 $$
  \ed
 The function ${\rm K}_\F$ does not depend on the choice of the orientation of the edge $\rm F$. 
 
 Thanks to Definitions \ref{DEFF3.1} and \ref{OMD},  we get   a canonical regular non-zero function
 $$
{\rm K}_\F: \mathcal{P}_{\bS} \lra \A^1.
$$
More generally, given any edge $\rm E$  connecting two cusps on  $\bS$, there is a rational function $\rm {K}_{\rm E}$ on the space 
${\cal P}_{{\rm PGL}_2, \bS}$. Namely, let $(\omega_1, v_1)$ and $(\omega_2, v_2)$ be the decorations at the vertices $v_1, v_2$ of the edge $\rE$. Then 
$$
{\rm {K}}_{\rm E}:= (\omega_1(v_1, v_2)  \omega_2(v_1, v_2))^{-1}.
 $$

\subsubsection{The potential function.} For each marked point $p$ on $\bS$, there is a function $W_p$ on  $\mathcal{P}_{\bS}$, called the {\it potential at $p$}. 
It is defined as follows. Pick a non-zero volume form $\omega$ at the fiber at the marked point $p$. So the decorations at  $p$ 
 are given by pairs $(v_p,\omega)$.  Let us parallel transport the form $\omega$ to nearby points on the boundary of $\bS$ to the left and to the right of $p$. Denote by 
 $(v_-,\omega)$ and $(v_+,\omega)$  the decorations at the fibers of ${\cal L}$ 
at  the marked points $p_-$ and $p_+$ to the left and to the right of  $p$.  Then we define  
$$
W_p({\cal L}, \alpha, \beta):= \frac{\omega(v_-, v_+)}{\omega(v_p, v_-) \omega(v_p, v_+)}.
$$
 
 The potential  does not change if we multiply any of the two vectors $v_-, v_+$ by a non-zero scalar. 
 
 The potential function  also does not change under the equivalence $(v_p, \omega) \sim (\lambda v_p, \lambda^{-2}\omega)$. 
 
 Therefore the potential at $p$  is a well defined  non-vanishing regular function
 $$
 W_p: {\cal P}_{ \bS}\lra \A^1.
 $$
 Indeed,    vectors $v_p, v_-, v_+$ are non-zero, and  pairs of vectors 
 $(v_p, v_-)$ and $(v_p, v_+)$ are not  collinear. \vskip 2mm
 
 \bd The potential function $W$ is the sum of the potentials at all marked points $p$ of $\bS$:
 \be
 W:= \sum_{p}W_p.
  \ee
 \ed

\subsubsection{The cluster volume form and its sign.}  The cluster volume form $\Omega_\bS$ is a volume form with logarithmic singularities  on the moduli space $\mathcal{P}_{\bS}$. 
It is
 given by the 
   product   of the logarithmic 1-forms over all edges of a given ideal triangulation ${\cal T}$ of $\bS$,  defined up to a sign:\footnote{The extra factor $2^{\pi_0(\bS)}$ spares us from extra factor $2$ in Lemma \ref{L2.16} appearing  when the edge $\rE$ there is separating. However it results in the extra factor $2$ in front of the Bessel function in (\ref{89}).  
Note that the volume form with/without the extra factor is multiplicative: $ \Omega_{\bS_1\cup \bS_2} =  \Omega_{\bS_1}\wedge  \Omega_{\bS_2}$. }
\be \la{CVF}
 \Omega_\bS:= \pm 2^{\pi_0(\bS)} d\log \X_{\rm E_1} \wedge \ldots    \wedge d\log \X_{{\rm E}_{k}}.
\ee 
The sign of $\Omega_\bS$   depends on the choice of an order of the edges. \vskip 2mm

We take care of the sign issue of (\ref{CVF}) as follows. Denote by ${\cal E}_{\cal T}$ the set of the edges of an ideal triangulation ${\cal T}$ of $\bS$. Let 
 ${\rm Or}_{\cal T}$ be its orientation torsor. It is a $\Z/2\Z-$torsor, that is a 2-element set equipped with the non-trivial $\Z/2\Z-$action. Its elements are orderings 
 of the set ${\cal E}_{\cal T}$ modulo   even permutations. 
We denote the element corresponding to an ordering ${\rm E}_{i_1},  ...,{\rm E}_{i_{k}}$ by 
$$
\varepsilon_{\cal T}(i_1, ..., i_{k}) \in {\rm Or}_{\cal T}.
$$
Given a flip of an ideal triangulation $\varphi_{\rm E}: {\cal T} \lra {\cal T}'$ at an edge $\rE$,  there is a canonical identification of the edges of 
${\cal T}$ and ${\cal T}'$. It gives rise to the following  isomorphism of the orientation torsors
\be \la{53}
\begin{split}
&\varphi_{\rm E}:  {\rm Or}_{\cal T} \lra {\rm Or}_{\cal T'}, \\
&\varepsilon_{\cal T}(i_1, ..., i_{k})\lms -\varepsilon_{\cal T'}(i_1, ..., i_{k}).\\
\end{split}
\ee
Here $-$ amounts to the action of the element $-1 \in \Z/2\Z$. 
To state the properties of this construction,  recall the modular groupoid ${\cal G}_\bS$ of a decorated surface $\bS$.

The modular groupoid  ${\cal G}_\bS$ is a groupoid whose objects are ideal triangulations ${\cal T}$ of $\bS$, the morphisms are generated by the flips, and the relations   between the flips are generated by 

(i)  {\it Pentagon relations}: the composition of the five flips of  diagonals of a pentagon is the identity map.

(ii) {\it Square relations}: the flips at disjoint diagonals commute. 

 The mapping class group of $\bS$ is the fundamental group of the modular groupoid. The proof is deduced from Strebel's theory of quadratic differentials \cite{Str}, \cite{Ha}, see  \cite[Section 3]{FG03a} for further discussion. 
 
 \bl There is a functor from the groupoid ${\cal G}_\bS$ to the groupoid of 
$\Z/2\Z-$torsors, which  assigns to an ideal triangulation ${\cal T}$ of $\bS$ the $\Z/2\Z-$torsor ${\rm Or}_{\cal T}$, and to a flip  at an edge $\rE$ the isomorphism (\ref{53}). 
 \el 

\begin{proof} We have to prove that 
the compositions of flips for the square and pentagon relations 
induce the identity maps on the orientation torsor. 
 The composition of the five flips related to a pentagon returns back the original triangulation of the pentagon, but the order of the two internal diagonals is switched, which agrees with $(-1)^5=-1$. 
The diagonals outside of the pentagon remain intact. So the composition   acts as the  identity map on the orientation torsor. The square relations are evident. 
\end{proof}

Since all orientation torsors ${\rm Or}_{\cal T}$ are canonically isomorphic, and the isomorphisms between them are compatible with the  relations, we arrive at the canonical orientation torsor ${\rm Or}_{\bS}$ assigned to a decorated surface $\bS$.  
It allows us to introduce the cluster volume form  with values in the tensor product of the volume forms with logarithmic singularities on ${\cal P}_\bS$ 
by the orientation torsor: 
\be \la{CVF+}
\Omega_\bS\in   {\rm Or}_{\bS} \otimes_{\Z/2\Z} \Omega_{\rm log}({\cal P}_\bS).
\ee
  Here   the group $\Z/2\Z$ acts on  volume forms so that the generator acts by changing the sign of the  form.  
  
\bd   The cluster volume form $\Omega_\bS$ in (\ref{CVF+}) is given by setting
\be \la{CVF***}
\begin{split}
  &\Omega_\bS:= 2^{\pi_0(\bS)} \cdot
 \varepsilon_{\cal T}(i_1, ..., i_{k})  \otimes_{\Z/2\Z}  d\log \X_{\rm E_1} \wedge \ldots \wedge d\log \X_{{\rm E}_{k}}.\\
\end{split}
\ee 
\ed
 
 The crucial fact that the cluster volume form does not depend on the choice of an ideal triangulation ${\cal T}$ 
 is easy to check directly. It can also be deduced from the general properties of the cluster volume form on cluster varieties, which we address in the next subsection, see Lemma \ref{L2.10}.

 \subsubsection{The canonical cluster volume form on a cluster  variety.} We address the reader to \cite{FG03b} for general properties of cluster varieties.  
 The above construction, suitably modified, provides the cluster volume form on any cluster 
 variety. 
 Namely, given a seed ${\bf c}$, we introduce the orientation $\Z/2\Z-$torsor ${\rm Or}_{\bf c}$.  Its elements  $ \varepsilon_{\bf c}(i_1, ..., i_{k})$ correspond to the ordered  cluster coordinates $(\X_{i_1}, ..., \X_{i_k})$  in the seed ${\bf c}$. Interchanging two  coordinates amounts to changing the sign. 
 A mutation $\mu_k: {\bf c} \lra {\bf c}'$ gives rise to a canonical  isomorphism of $\Z/2\Z-$torsors 
 \be \la{CICS}
\begin{split}
&\mu_{k}:  {\rm Or}_{\bf c} \lra {\rm Or}_{\bf c'}, \\
&\varepsilon_{\bf c}(i_1, ..., i_{k})\lms -\varepsilon_{\bf c'}(i_1, ..., i_{k}).\\
\end{split}
\ee
 These isomorphisms satisfy the {\it standard $(h+2)-$gon relations}, discussed in 
  \cite[Proposition 1.8]{FG03b}, which are used to introduce the cluster modular groupoid. If  the quiver underlying the seed ${\bf c}$ is simply-laced, then $h=2,3$. Moreover,   for the cluster variety ${\cal P}_\bS$ 
  we recover  the square and pentagon relations. So we arrive at the $\Z/2\Z-$torsor  ${\rm Or}_{\cal X}$. 

 The group  $\Z/2\Z$ acts on  volume forms so that the generator acts by changing the sign of the form. 
 Then the cluster  volume form on a cluster Poisson variety ${\cal X}$ 
 is defined by 
 \be \la{CVF*}
\begin{split}
&  \Omega_{\cal X}\in   {\rm Or}_{\cal X} \otimes_{\Z/2\Z}  \Omega_{\rm log}({\cal X}); \\
  &\Omega_{\cal X}:=  
 \varepsilon_{\bf c}(i_1, ..., i_{k})  \otimes_{\Z/2\Z}   d\log \X_{i_1} \wedge \ldots \wedge d\log \X_{i_k}. \\
\end{split}
\ee 
 It evidently does not depend on the choice of the order of the cluster coordinates. 
 
 Similarly, the cluster  volume form on a cluster  ${\cal A}-$variety ${\cal A}$  
 is given  by
  \be \la{CVF*A}
\begin{split}
&  \Omega_{\cal A}\in   {\rm Or}_{\cal A} \otimes_{\Z/2\Z}  \Omega_{\rm log}({\cal A}); \\
  &\Omega_{\cal A}:=  
 \varepsilon_{\bf c}(i_1, ..., i_{k})  \otimes_{\Z/2\Z}   d\log \A_{i_1} \wedge \ldots \wedge d\log \A_{i_k}. \\
\end{split}
\ee 

\begin{remark} The cluster volume form $\Omega_\bS$ on the space ${\cal P}_\bS$ introduced in  (\ref{CVF***}) is equal to $2^{\pi_0(\bS)}$ times the cluster volume form 
$\Omega_{\cal X}$ on the related cluster variety. Due to the extra factor $2^{\pi_0(\bS)}$ we get  simpler constants in unfolding formulas, including the constant $1$ in  the cutting formulas (\ref{ECab})  and (\ref{EC}). 
\end{remark}

 \bl \la{L2.10} The cluster volume forms $\Omega_{\cal X}$ and ${\Omega}_{\cal A}$ are invariant under cluster mutations.
 \el
 
 \bpr i) The cluster Poisson mutation $\mu_{k}$ acts by $\X_k \lms \X_k^{-1}$ and $\X_j \lms \X_j P_j(\X_k)$, $j \not = k$, where $P_j(\X_k)$ is a Laurent polynomial in $\X_k$, see \cite[formula (13)]{FG03b}. 
 Therefore it changes the sign of the  form $d\log \X_{i_1} \wedge \ldots \wedge d\log \X_{i_n}$. 
 This sign change is compensated by the sign change in  isomorphism (\ref{CICS}).  
 
 ii) The proof in the ${\cal A}-$case is similar. The only cluster coordinate which changes under the mutation $\mu_k$ is the one $\A_k$, and the exchange relation tells 
 that $\A_k\A_k'$ is a Laurent polynomial in $\A_j$ where $j \not = k$. Therefore the form $d\log \A_{i_1} \wedge \ldots \wedge d\log \A_{i_n}$ changes sign under mutations. 
 \epr
 
  \bp \la{CLVF} The cluster volume form $\Omega_{\cal X}$ induces a canonical positive measure $\Omega_{\cal X}$ on the space of positive real points 
  ${\cal X}(\R_{>0})$. The same is true for the cluster ${\cal A}-$varieties. 
  \ep
 
 \bpr  Given a seed ${\bf c}$, the cluster coordinates provide a canonical isomorphism
 $$
 i_{\bf c}: {\cal X}(\R_{>0}) \lra \R^k.  
 $$
 Therefore an order of the cluster coordinates provides an orientation of the space ${\cal X}(\R_{>0})$. 
 Altering the order we alter the orientation by the sign of the permutation. The same sign shows up in the  cluster volume form. 
 So the integral of the cluster volume form over a compact domain is  positive. \epr

\subsubsection{Calculating the potential and  $\K-$functions.}   Take a vertex $v$  of an ideal triangulation  ${\cal T}$ of $\bS$. The orientation of $\bS$ provides the counterclockwise 
orientation of the edges  of ${\cal T}$ sharing the vertex $v$.

  \bt \la{L11}   Denote by  ${\F}, {\rm E_1},  ...,  {\rm E_k}, {\F^+}$ the edges sharing a vertex $v$ of the triangulation ${\cal T}$, ordered counterclockwise. So $\F, \F^+$ are the two external edges sharing the vertex $v$, which may coincide. 
  
   Then  the function $ {\rm K}_\F$ and the potential $W_v$ are  given by:
\be \la{kf}
{ \rm K}_{\F^+}= \B_{\F}\X_{\rm E_1} \ldots \X_{\rm E_k} \B_{\F^+}.
\ee
 \be \la{fw}
 W_v = \B_{\F} + \B_\F \X_{{\rm E}_1} + \ldots + \B_\F \X_{{\rm E}_1}  \ldots  \X_{{\rm E}_{k}}.
 \ee
  \et

  \bpr This is done by an explicit calculation. We start  from the two simplest examples of the calculation. \vskip 2mm

\subsubsection{\it The   triangle.} Let $\bS=\tau$ be a triangle. Denote by $(v_i,\omega)$ a decoration at the vertex $i$, where $i=\Z/3\Z$, and the vertices are numbered counterclockwise. Here $v_i$ are non-zero vectors in a two dimensional vector space $V_2$. There is a 
unique vector $v_{i+1}'$ proportional to $v_{i+1}$ such that $\omega(v_i, v'_{i+1})=1$, see Figure \ref{sz3}:
$$
v_{i+1}'= \frac{v_{i+1}}{\omega(v_i, v_{i+1})}.
$$
   \begin{figure}[ht]
\centerline{\epsfbox{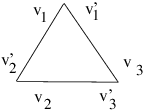}}
\caption{Defining the pinnings for a triangle.}
\label{sz3}
\end{figure} 
So the pinning vector  on the side $(i, i+1)$ of the triangle is given by 
 $$
 p_{i, i+1}= v_i+v'_{i+1}.
 $$ Therefore the $\B-$coordinate assigned to the side $(i, i+1)$ is 
 $$
 \B_{i, i+1}:= r^+(v_i, p_{i, i+1}, v_{i+1}, v_{i+2})  = 
 \frac{\omega(v_{i+1}, v_{i+2})}{\omega(v_{i}, v_{i+1})\omega(v_{i}, v_{i+2})}.
 $$ 
  This coincides with   the potential $W_{i}$ at the vertex $i$, see \cite[Example 3.16]{GS19}, confirming (\ref{fw}):
\be \la{FPo}
 W_{i} = \B_{i, i+1}.
\ee
Finally, we have the following, confirming formula (\ref{kf}):
$$
\B_{12}\B_{31}=  
 \frac{\omega(v_{2}, v_{3})}{\omega(v_{1}, v_{2})\omega(v_{1}, v_{3})} \cdot  \frac{\omega(v_{1}, v_{2})}{\omega(v_{3}, v_{1})\omega(v_{3}, v_{2})} = \omega(v_1, v_3)^{-2}.
 $$

\paragraph{\it The  rectangle.} Let $\bS$ be a rectangle  with decorations $(\omega, v_i)$ at the vertex $i$, where $i=\Z/4\Z$. 
Let ${\rm E}$ be the internal diagonal connecting  vertices $1$ and $3$.    
\begin{figure}[ht]
\centerline{\epsfbox{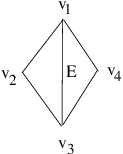}}
\caption{Relating the coordinates and local potentials for a rectangle.}
\label{sz4}
\end{figure} 
Denote the potential at the vertex $1$ of the oriented angle $314$ by
\be
\la{PEVE}
W_{1}((1,3), (1,4)):=W_{314}:=\frac{\omega(v_3, v_4)}{\omega(v_1, v_3) \omega(v_1, v_4)},
\ee
etc. Note that $W_{413}=-W_{314}$. Then 
$$
\B_{23} = W_2, \ \ \ \B_{41} = W_4, \ \ \ \B_{12}= W_{213}, \ \ \ \B_{34}= W_{431}.
$$
Calculation shows that  
$$
\X_{\rm E} = \frac{W_{314}}{\B_{12}}.
$$
Therefore we have, using the additivity of the potential, and confirming formula (\ref{fw}):
$$
W_1 =  W_{213} + W_{314} = \B_{12}+ \B_{12}\X_{\rm E}.
$$

Finally, we have the following,   confirming formula (\ref{kf}):
$$
\K_{41} = \B_{12}\X_{\rm E} \B_{41}=  \frac{\omega(v_{2}, v_{3})}{\omega(v_{1}, v_{2})\omega(v_{1}, v_{3})} \cdot \frac{\omega(v_{1}, v_{2})\omega(v_{3}, v_{4})}{\omega(v_{2}, v_{3})\omega(v_{1}, v_{4})}\cdot   \frac{\omega(v_{1}, v_{3})}{\omega(v_{4}, v_{1})\omega(v_{4}, v_{3})}= \omega(v_1,v_4)^{-2}.
$$
The case of an $n-$gon is very similar, and the general case reduces to this. \epr 

 We conclude that  the moduli space $\mathcal{P}_{ \bS}$ carries   functions of three flavors: 
 
 \begin{enumerate}
 
 \item The regular functions $ {\rm K}_{\F}$ assigned to the external edges $\F$ on $\bS$.
 
 \item The rational functions $ {\B}_{\F}$ assigned to the external edges $\F$  on $\bS$. 
 
 \item  The rational functions $ {\X}_{\rm E}$ assigned to the   internal edges $\rm E$ of a given ideal triangulation of  $\bS$.   
 
 \item The regular potential functions $W_p$ assigned to the marked boundary points $p$. 
 
 \end{enumerate}
 
 These functions are related by the  relations (\ref{kf}) and (\ref{fw}).\\

 \bp \la{PropP=X} There is a canonical isomorphism
\be \la{P=Xa}
 \mathcal{P}_{\bS} = \mathcal{X}_{\bS'}
\ee
 where the decorated surface $\bS'$ is obtained from $\bS$ by adding a marked point inside each boundary interval. 
 \ep
 
 \bpr Any $({\cal L}, \alpha, \beta)\in  \mathcal{P}_{\bS}$ gives rise to a pinning 
 point $p_\F$ in the fiber of the local system ${\cal L}_{\P^1}$ over each boundary interval $\F$, see (\ref{pli}). So we get a point  
 of $\mathcal{X}_{\bS'}$ given by the local system  ${\cal L}$ together with the pinning points $p_\F$ and the framing points. 
  
 Conversely, take the three framing lines $\L_p, \L_\F, \L_q\subset V_2$  at the ends $p, q$ of a  boundary interval  $\F$ and at $\F$.   Then  $\L_\F$ is the graph of an   isomorphism 
 $i_\F: \L_p\lra \L_q$. 
 There exist decorating pairs  $(v_p,\omega_p)$ and $(v_q,\omega_q)$ at the vertices $p,q$ such that 
 \be \la{Co}
 \omega_p(v_p, v_q) \omega_q(v_p, v_q)=1, \qquad v_q = i_\F(v_p). 
 \ee
 The pair $(v_p,\omega_p)$, and hence $(v_q,\omega_q)$, are defined uniquely up to rescaling. Indeed, if $v'_p:=  \lambda v_p$, then  $v_q'=\lambda v_q$. So there is a unique up to a sign  $\lambda\in \Bbb C^\times$ such that equation (\ref{Co}) holds. Note that $(-v_p,\omega_p)$ is obtained by 
 rescaling  of $(v_p,\omega_p)$ by $-1$.   We declare these $(v_p,\omega_p)$ and $(v_q,\omega_q)$ the decorations at the marked point $p$ and $q$. 
 So the framing lines $\{\L_p\}$  at the marked points $p$ and the pinning lines $\{\L_\F\}$  at the boundary edges $\F$ determine uniquely the decorations 
 $(v_p, \omega_p)$ at the marked points $p$.
 \epr

 \subsubsection{Proof of Theorem \ref{CPS}.} The mapping class group equivariant     
 cluster Poisson structure on $\mathcal{X}_{G,\bS'}$   was  defined in  \cite{FG03a}.  It is the one constructed in Theorem \ref{CPS}.

\subsubsection{The enhanced Teichm\"uller space ${\cal T}^{\rm en}_\bS$ is the set of real positive points of ${\cal P}_{ \bS}$.}  
 
The {\it enhanced Teichm\"uller space} ${\cal T}_\bS^e$ is  the $2^{m-r}:1$ ramified cover of the Teichm\"uller space 
 ${\cal T}_\bS$,  parametrising ideal hyperbolic structures on $\bS$ with a choice of a sign of the length of   geodesic  around each puncture.

  \bt \la{P=X} The enhanced Teichm\"uller space ${\cal T}^{\rm en}_\bS$ of a decorated surface $\bS$  is canonically isomorphic to the space of real positive points of the cluster Poisson variety 
${\cal P}_{\bS}$:
\be
{\cal T}^{\rm en}_\bS = {\cal P}_{  \bS}(\R_{>0}).
\ee
\et
 
 \bpr Proposition \ref{PropP=X} reduces the claim to the description of positive real locus of the space  $\mathcal{X}_{\bS'}$. 
 According to \cite{FG03a}, we have 
 $
 \mathcal{X}_{\bS'}(\R_{>0}) = {\cal T}^{\rm en}_{\bS'}. $
 On the other hand, we have 
 $
 \mathcal{X}_{\bS'}(\R_{>0}) =   \mathcal{P}_{\bS}(\R_{>0}). 
 $
  \epr

The  cluster Poisson coordinates  of an ideal triangulation  of $\bS$ 
 provide an isomorphism 
$
\mathcal{T}^{en}_\bS \stackrel{\sim}{\lra} \R_{>0}^N.
$ 
 
\bl The cluster volume form $\Omega_\bS$ in (\ref{CVF+}) gives rise to a positive measure on the enhanced Teichm\"uller space ${\cal T}^{\rm en}_{\bS}$.
\el

\bpr
Ordering the edges of an ideal triangulation ${\cal T}$ of $\bS$ we get the sign of the  form $\Omega_\bS$ 
 as well as an orientation of the  enhanced Teichm\"uller space  ${\cal T}^{\rm en}_{\bS}$.  Changing the order of the edges we change simulteneously the sign of the volume form and the orientation of the  enhanced Teichm\"uller space. 
 So the restriction of the  form $\Omega$ to the enhanced Teichm\"uller space provides canonical measure  on  it.  
 \epr
 
 Therefore for any continuous function $f$ on the Teichm\"uller space, and any compact domain ${\cal D}$, the integral $\int_{\cal D} f\Omega_\bS$ makes sense. 
 If $f$ is a positive function, the integral is positive.

 \subsection{The cutting and gluing maps for the moduli spaces ${\cal P}_\bS$}
 \la{subsectioncg}
 
 \subsubsection{The gluing map.} Take a decorated surface $\bS'$, possibly disconnected.  Let us glue   edges $\rm {E}'$ and ${\rm E}''$ on $\bS'$ into an edge $\rm E$ on $\bS$, getting a new decorated surface $\bS$. 
 Denote by ${\cal P}_{\bS'}^\delta$ the subspace of ${\cal P}_{\bS'}$ defined by the equation $\rm {K}_{{\rm E}'} = \rm {K}_{{\rm E}''}$. 
Then there is a gluing map \cite[Section 7.1]{GS19}:
\be \la{GM}
\gamma_{{\rm E}' {\rm E}''}: {\cal P}^\delta_{ \bS'}  \lra  {\cal P}_{   \bS}.
\ee
Namely, take a point $({\cal L}, \alpha, \beta)\in {\cal P}^\delta_{  \bS'}$. Then  we can glue uniquely 
the ${\rm PGL}_2-$local system 
${\cal L}$ on  $\bS'$  to a local system onto $\bS$ so that the decorations at the vertices $x'_i, x''_i$ on $\bS'$ matching the vertex $x_i$ on $\bS$ produce decorations at the vertices $x_i$ on $\bS$, for $i=1,2$. Indeed, take the decoration vectors $v'_1, v'_2$ on $\rE'$   and $v_1'', v_2''$ on $\rE''$. Since $\rm {K}_{{\rm E}'} = \rm {K}_{{\rm E}''}\not = 0$,  
we have  $\omega(v_1', v_2') = \pm \omega(v_1'', v_2'') \not = 0$. So  there is a unique up to multiplication by $-1$ isomorphism of  vector spaces 
$$
\varphi: {\cal L}'_{| \rm E'}\lra {\cal L}'_{|\rm E''}
$$
such  that $\varphi(v_i')= v_i''$.  Note  that decorations $(v,\omega)$ and $(-v, \omega)$ are the same by the definition. 
 \vskip 2mm

The gluing map acts on cluster Poisson coordinates as follows.  Take an  ideal triangulation of $\bS'$.  It provides   
 an ideal  triangulation  of $\bS$.  Let $\B', \B'', \X$ be the cluster Poisson coordinates assigned to the edges 
$\rm {E}', \rm {E}'', {\rm E}$.   By \cite[Lemma 7.1]{GS19}, we have  
\be \la{BB}
\gamma_{{\rm E}' {\rm E}''}^*(\X) = \B'\B''. 
\ee
The  map $\gamma_{{\rm E}' {\rm E}''}^*$ does not change  
  the other coordinates.

\subsubsection{The cutting map.} Take a decorated surface $\bS$ with an internal edge $\rm E$ connecting  {marked} points. 
Denote by $\bS'$ the decorated surface obtained by cutting $\bS$ along $\rm E$. It 
has two new boundary edges ${\rm E}'$, ${\rm E}''$. 
We define a  rational map
$
{\cal P}_{\bS} \lra {\cal P}_{\bS'} 
$  
  by restricting a local system  on $\bS$ to $\bS'$, and inducing the decorations and framing. It is well defined  if and only if  
 the decoration lines at the ends of the edge $\rE$, parallel transported   to a  point of ${\rm E}$ along the edge, 
 are different. 
  Its image   lies in the subspace ${\cal P}^\delta_{\bS'}$, defined by the equations $\K_{\rm E'}=\K_{{\rm E}''}\not = 0$. So we arrive at  the rational  {\it cutting  map}   
\be \la{cut2}
{\rm cut}_{\rm E}: {\cal P}_{\bS} \lra {\cal P}^\delta_{\bS'}.
\ee

\bl The cutting and gluing maps   ${\rm cut}_{\rm E}$ and $\gamma_{{\rm E}' {\rm E}''}$ are mutually inverse birational isomorphisms. 
\el
\bpr Follows from the very definitions. \epr

\subsubsection{Relative volume forms.} Consider a   map  assigning to a point of ${\cal P}_\bS$  the collection of its $\K-$coordinates: 
\be \la{bs}
\begin{gathered}
    \xymatrix{
b_\bS:   {\cal P}_{\bS}       \ar[r]^{}& {\cal K}_{\bS}} := {\Bbb G}_m^{\{\mbox{boundary intervals of $\bS$}\}}.
\end{gathered}
 \ee
 There is a similar  map for the surface $\bS'$, where the projection to the first factor is given by  ${\rm K}_{\rm E'}=\K_{\rm E''}$: 
 \be \la{bs75}
\begin{gathered}
    \xymatrix{
b^\delta_{\bS'}:   {\cal P}^\delta_{\bS'}       \ar[r]^{}& {\cal K}^\delta_{\bS'}} = {\Bbb G}_m\times { {\Bbb G}_m}^{\{\mbox{boundary intervals of $\bS$}\}}.
\end{gathered}
 \ee

  The  volume  form $\Omega_\bS$ together with the product of the  canonical invariant volume forms  $d\log \K$ on ${\Bbb G}_m$ 
  induce a volume form on the fibers of projection (\ref{bs}), denoted by 
  $\Omega_\bS(\K)=\Omega_\bS(\K_1,...,\K_b)$, and defined up to a sign from the equation:\footnote{The $\K$ in   $\Omega_\bS(\K)$  indicates that the    $\K-$coordinates on the decorated surface $\bS$ are frozen.}  
\be \la{RVF}
d\log \K_1 \wedge \ldots \wedge d\log \K_b \wedge \Omega_\bS(\K) = \Omega_\bS.
\ee 
  Here $b$ is the number of  boundary edges on $\bS$. 
    The sign is determined by an order of the set of boundary edges, modulo 
   even permutations. 
   Similarly, we get a  form  $\Omega^\delta_{\bS'}(\K)$ on the fibers of projection (\ref{bs75}).  
More generally, we define $\Omega_\bS(\K,\L)$ such that
\be \la{RVF2}
d\log \K_1 \wedge \ldots \wedge d\log \K_b \wedge d l_1\wedge \ldots\wedge d l_a \wedge \Omega_\bS(\K,\L) = \Omega_\bS.
\ee 
where $\L=(l_1,...,l_a)$ are the boundary geodesic circle lengths.

 \vskip 2mm

Before we proceed with the proof, let us make some preparations. 

The gluing map is described nicely in the cluster Poisson coordinates $(\B, X)$ in (\ref{BB}). The relative volume form 
$\Omega_\bS(\K)$ is defined in (\ref{RVF}) by dividing the cluster volume form $\Omega_\bS$, expressed nicely in the cluster Poisson coordinates $(\B, X)$, by the volume form 
on the frozen torus, defined via  $\K-$coordinates.

The surface $\bS'$ has 
two new boundary edges ${\rm E}'$ and ${\rm E}''$. 
We have the gluing condition: 
\be \la{KEEK}
\K_{\rm E'}=\K_{\rm E''} =\K_{\rm E}.
\ee

  Pick an ideal triangulation of $\bS$ containing the edge ${\rm E}$. It induces a triangulation of $\bS'$.  
There is a unique vertex $v_1'$ of ${\rm E}'$ such that ${\rm E}'$ is the last among the edges sharing $v_1'$  in the counterclockwise order. There is a  vertex $v_2'$ of ${\rm E}''$ with the same property.  
Denote by ${\rm F}', {\rm E}'_{1}, ... {\rm E}'_{a}, {\rm E}'$ (respectively ${\rm F}'',   {\rm E}''_{1}, ...,{\rm E}''_b,  {\rm E}''$) the edges sharing the vertex $v_1'$ (respectively $v_2'$) on  $\bS'$, counted counterclockwise. Then we have the following monomial relations:
\be \la{72}
\begin{split}
&\K_{\rm E'} \stackrel{ (\ref{kf})}{=} \B_{\rm F'} \X_{ {\rm E}'_{1}} \ldots \X_{ {\rm E}'_{a}}\B_{\rm E'},\\
&\K_{\rm E''} \stackrel{ (\ref{kf})}{=} \B_{\rm F''}\X_{ {\rm E}''_{1}} \ldots \X_{ {\rm E}''_{b}}  \B_{\rm E''} ,\\
&\X_{\rm E} \stackrel{(\ref{BB})}{=}\B_{\rm E'} \B_{\rm E''}. \\
\end{split}
\ee

Now we return to the proof. Let us determine how the volume forms behave under cutting of an edge. 
We set, using the logarithmic coordinates $dk_*:= d\log \K_*$:\footnote{Here $\delta(f)  \Omega$, where $\Omega$ is a volume form  on an $n-$dimensional manifold, is the volume form on the hypersurface $f=0$ obtained as follows: we find an $(n-1)-$form $\omega$ such that $df \wedge \omega = \Omega$, and  restrict $\omega$ to the hypersurface $f=0$. The result is independent on the choice of $\omega$, see \cite[Chapter III, section 1.2]{GF1}.}
$$
\Omega^\delta_{\bS'}:=\delta(k_{\rE'}- k_{\rE''}) \Omega_{\bS'}.
$$

 \bl \la{L2.16} Cutting a decorated surface $\bS$ along an edge $\rE$, we get
   \be \la{ECab}
\begin{split}
& \Omega_\bS =  {\rm cut}^*_{\rm E} \ \Omega^\delta_{\bS'}
\end{split}
\ee
   \el

   \begin{figure}[ht]
	\centering
	\includegraphics[scale=0.3]{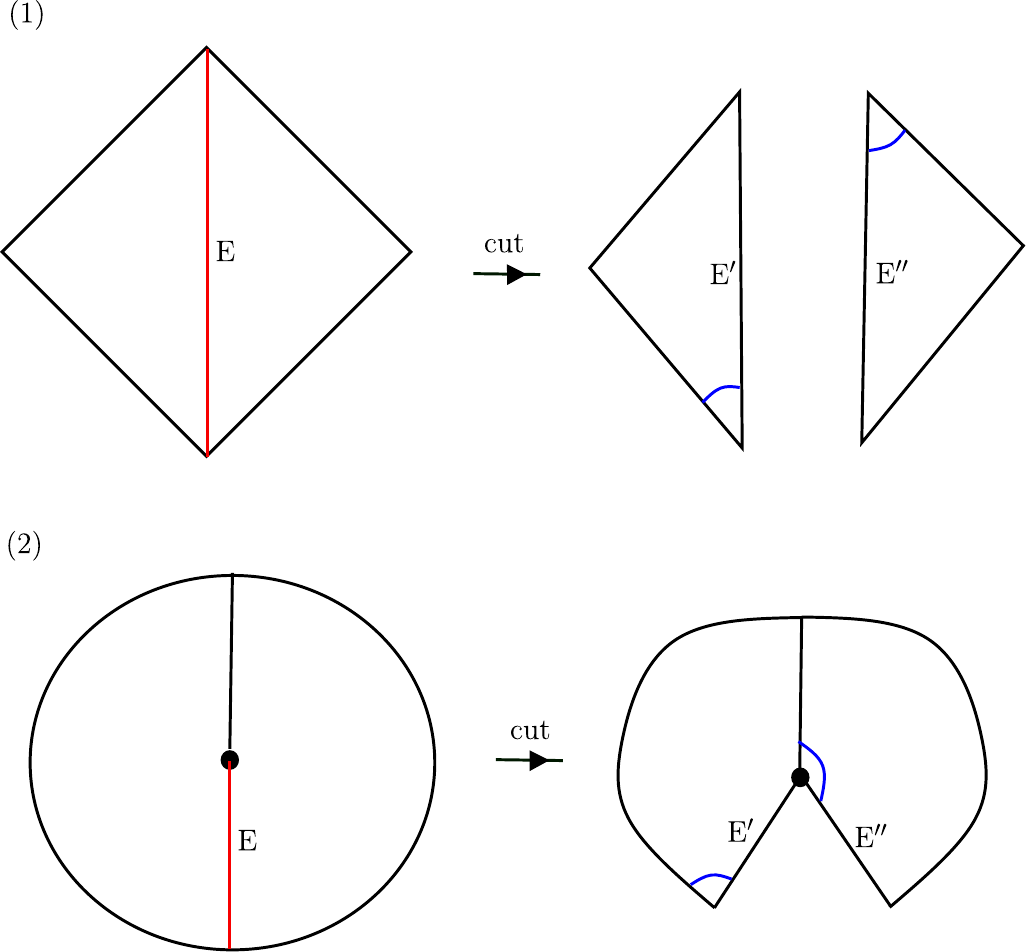}
	\small
	\caption{(1) The edge $\rE$ is separating. (2) The edge $\rE$ is non-separating.}
	\label{figure:cut}
\end{figure}

   \bpr   
When $\rE$ is separating as in Figure \ref{figure:cut}(1), let  $k', k'', b', b''$ be the logarithmic $\K$ and $\B$ coordinates for the edges $\rE'$ and $\rE''$ on $\bS'$, 
   and  $x_\rE$ the logarithmic $x-$coordinate at the edge $\rE$.  Then, see (\ref{72}):
   $$
   x_\rE = b'+b'', \ \ k'= b'+\ldots ,  \ \ \ 
   k''= b'' + \ldots .
   $$
   Here $\ldots$ stands for a sum of certain coordinates $b_i$ and $x_j$. The gluing condition is 
   $$
   0= k'-k'' = b'-b'' + \ldots . 
   $$
   Since $x_\rE= b' + b''$, we have 
   \be 
   \begin{split}
  & \Omega_\bS = d x_\rE \wedge i_{\partial/\partial x_\rE} \Omega_\bS=2d b' \wedge i_{\partial/\partial x_\rE} \Omega_\bS, \\
 & \frac{1}{2} \Omega_{\bS'} = db' db'' \wedge   i_{\partial /\partial x_\rE} \Omega_\bS,\\
 & \Omega_{\bS'} =2 d(k'-k'')\wedge db' \wedge i_{\partial /\partial x_\rE} \Omega_\bS=d(k'-k'')\wedge \Omega_\bS.\\  \end{split}
   \ee
  The lemma follows from this. 

When $\rE$ is non-separating as in Figure \ref{figure:cut}(2),  we have
   $$
   x_\rE = b'+b'', \ \ k'= b'+\ldots ,  \ \ \  
   k''= b'+b'' + \ldots .
   $$
Then the gluing condition is
\[0=k'-k''=-b''+...\]
Thus we get
   \be 
   \begin{split}
  & \Omega_\bS = d x_\rE \wedge i_{\partial/\partial x_\rE} \Omega_\bS=d b' \wedge i_{\partial/\partial x_\rE} \Omega_\bS, \\
 &  \Omega_{\bS'} = db' db'' \wedge   i_{\partial /\partial x_\rE} \Omega_\bS,\\
 & \Omega_{\bS'} =d(k'-k'')\wedge db' \wedge i_{\partial /\partial x_\rE} \Omega_\bS=d(k'-k'')\wedge \Omega_\bS.\\  \end{split}
   \ee
 \epr

We introduce the relative form $\Omega^\delta_{\bS'}(\K_\rE) $, 
so that\footnote{Strictly speaking, the relative form is the form $\Omega^\delta_{\bS'}(\K_\rE)$,  restricted to the fiber $\K_\rE=C$.} 
$$
\Omega^\delta_{\bS'}=\Omega^\delta_{\bS'}(\K_\rE) \wedge d\log \K_\rE.
$$ 
Formula (\ref{ECab}) implies the key relation:
    \be \la{ECa}
\begin{split}
& \Omega_\bS =  {\rm cut}^*_{\rm E} \ \Omega^\delta_{\bS'}(\K_\rE) \wedge d\log \K_\rE.
\end{split}
\ee

  \bp \la{PROP2.9} 
Let $\K=(\K_a)_a$ for all the boundary intervals. One has  
\be \la{EC}
\begin{split}
& \Omega_\bS(\K) =  {\rm cut}^*_{\rm E} \ \Omega^\delta_{\bS'}(\K_\rE,\K) \wedge d\log \K_\rE.
\end{split}
\ee
\ep

\bpr To calculate the  form $\Omega_\bS(\K)$, we take the form  $\Omega_\bS$, then divide it by the product of $d\log \K_a$ for all boundary intervals on $\bS$, and restrict to the fiber where 
$\K_a$'s are constants. To calculate $\Omega^\delta_{\bS'}(\K)$ we have to do exactly the same. 
\epr

   \begin{figure}[ht]
	\centering
	\includegraphics[scale=0.4]{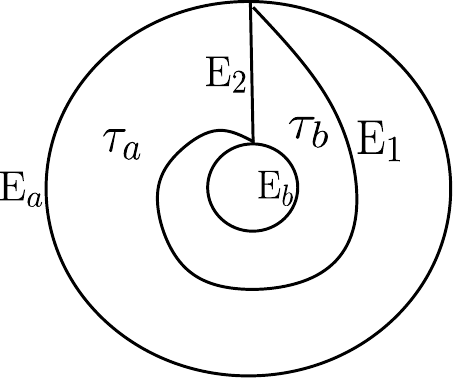}
	\small
	\caption{An example $A_{1,1}$.}
	\label{figure:a11}
\end{figure}
 \paragraph{An example: $\bS= \A_{1,1}$.} The annulus $\A_{1,1}$ is glued from two triangles $\tau_a$ and $\tau_b$ with the sides ${\rm E}_1^{(a)}, {\rm E}_2^{(a)},  {\rm E}_{a}$ and 
 ${\rm E}_1^{(b)}, {\rm E}_2^{(b)},  {\rm E}_{b}$ respectively as in Figure \ref{figure:a11}. The  logarithmic coordinates on the triangle $\tau_a$ are ${b}_1^{(a)}, {b}_2^{(a)},  {b}_{a}$ and  
 ${k}_1^{(a)}, {k}_2^{(a)},  {k}_{a}$, and similarly for the triangle $\tau_b$. The relations are 
\be 
\begin{split}
& {b}_1^{(a)} + {b}_2^{(a)} = {k}_1^{(a)}, \ \  {b}_2^{(a)} +{b}_a = {k}_2^{(a)}, \ \ \  {b}_a +{b}_1^{(a)} = {k}_a.\\
 &{b}_1^{(b)} + {b}_2^{(b)} = {k}_1^{(b)}, \ \ \  {b}_2^{(b)}+ {b}_b = {k}_2^{(b)}, \ \ \  \ {b}_b +{b}_1^{(b)} = {k}_b.\\
\end{split}
\ee
   From this we calculate the relative forms on the triangles, which are just numbers:
  $$
  \Omega_{\tau_a}(\K_a, {\K}_1^{(a)}, {\K}_2^{(a)})= 1, \ \ \   \Omega_{\tau_b}(\K_b, {\K}_1^{(b)}, {\K}_2^{(b)}) = 1.
     $$

Let us also set the coordinates  satisfying the gluing conditions, for the fiber with frozen $k_a, k_b$:
 $$
k_1 = {k}_1^{(a)}= {k}_1^{(b)}, \ \ \ \ k_2 = {k}_2^{(a)}= {k}_2^{(b)}.
$$

\bl 
\la{lA11}
 The form $\Omega_{\A_{1,1}}(\K_a, \K_b)$ is defined and calculated as follows:
 $$
 \Omega_{\A_{1,1}}(\K_a, \K_b) = \frac{2 d x_1\wedge dx_2 \wedge db_a \wedge db_b}{dk_a \wedge dk_b} = \frac{1}{2}   dx_1\wedge dx_2=    dk_1\wedge dk_2.
  $$
  \el
 
  One gets the last equality of the  Lemma by Proposition \ref{PROP2.9}. Let us prove in another way as follows.
  \bpr
Indeed, we have
  $$
 2 b_a + x_1+ x_2=k_a, \ \ \ \   2b_b + x_1 + x_2=k_b.   
  $$
  Therefore 
 $$
 2dx_1\wedge dx_2 \wedge db_a \wedge db_b =  \frac{1}{2}   dx_1\wedge dx_2 \wedge dk_a \wedge dk_b.
  $$

      The ${\cal X}-$coordinates are 
 $
 x_1=  {b}_1^{(a)} + {b}_1^{(b)}, \ \ \  x_2=  {b}_2^{(a)}  + {b}_2^{(b)}.  
 $
 Therefore we have 
 $$
 x_1+ x_2 = {k}_1^{(a)} + {k}_1^{(b)};\ \ \ x_1-x_2   = k_a+ k_b- {k}_2^{(a)} - {k}_2^{(b)}.  
 $$
 This implies that on the fiber with frozen $k_a, k_b$ we have 
 $$
 dx_1 \wedge dx_2 = 2 dk_1\wedge dk_2.
 $$
 \epr

By applying directly Proposition \ref{PROP2.9} and observing that $\Omega_{\Delta}(\K)=1$ for the triangle $\Delta$, we get
\bc
\la{cortri}
Let $\K=(\K_a)_a$ be  the $\K$-parameters for all boundary intervals of the $n-$gon $\P_n$, 
 $n\geq 3$. Given an ideal triangulation of $\P_n$, let $(\K_1,\ldots,\K_{n-3})$ be  the $\K$-parameters for the internal edges. Then  
\be \la{Pn}
\Omega_{\P_n}(\K)= d \log \K_1 \wedge \ldots \wedge d \log \K_{n-3}.
\ee
\ec 
\bpr
The $n-3$  internal edges cut the polygon $\P_n$ into $n-2$ triangles. Thus we get (\ref{Pn}).
\epr
By equation (\ref{RVF2}) we have the following Corollary.
\bc
\la{corcrown}
Let $\K=(\K_a)_a$ be the $\K-$parameters for all  boundary intervals of the once punctured disc $\D_n^*$ with $n$ marked boundary points for $n\geq 1$. Let $\L$ be the exponential of the unique boundary length. Given an ideal triangulation of $\D_n^*$, with the $\K$-parameters  $(\K_1,\ldots,\K_{n})$ at the internal edges, we have
\[\Omega_{\D_n^*}(\K,\L)= 2 d \log \B_1 \wedge \ldots \wedge d \log \B_{n}.\]
\ec

   \subsubsection{Additivity of  local potentials under the cutting and gluing.} 
 Let $v'$ and $v''$ be the vertices of the edges of ${\rm E}'$ and ${\rm E}''$ 
which glue into a vertex $v$ of ${\rm E}$. 
Denote by $W_{\bS', v'}$ the local potential ${\cal P}_{\bS'} $ at $v'$ etc. 

\bl \la{APP} One has 
$$
W_{\bS, v} = {\rm cut}^*_{{\rm E}}(W_{\bS', v'} + W_{\bS', v''}).
 $$
\el

  \begin{figure}[ht]
\centerline{\epsfbox{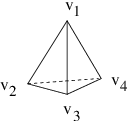}}
\caption{Additivity of the potential.}
\label{sz2}
\end{figure} 

\bpr This boils down to the followoing calculation, illustrated on Figure \ref{sz2}:
\be
\begin{split}
&\frac{\omega(v_2, v_3)}{\omega(v_1, v_2)\omega(v_1, v_3)} + \frac{\omega(v_3, v_4)}{\omega(v_1, v_3)\omega(v_1, v_4)} = \\
&\frac{ \omega(v_2,v_3)\omega(v_1,v_4)+ \omega(v_1, v_2)\omega(v_3, v_4)}{\omega(v_1, v_2)\omega(v_1, v_3)\omega(v_1, v_4)} 
\stackrel{}{=}\\
&\frac{ \omega(v_1,v_3)\omega(v_2,v_4)}{\omega(v_1, v_2)\omega(v_1, v_3)\omega(v_1, v_4)} =  \frac{ \omega(v_2,v_4)}{\omega(v_1, v_2)\omega(v_1, v_4)}.\\
\end{split}
\ee
Here the second equality is the provided by the Pl\"ucker relation.
\epr

Denote by $W^\delta_{\bS'}$   the potential   on ${\cal P}_{\bS'}^\delta$.  Lemma \ref{APP} implies that 
\be \la{-1eq}
{W_\bS} = {\rm cut}^*_{\rm E}(W^\delta_{\bS'}).
\ee
  \bp \la{PROP2.9+} 
One has 
\be 
\begin{split}
& e^{-W_\bS}\Omega_\bS(\K) =  {\rm cut}^*_{\rm E} \Bigl(e^{-W^\delta_{\bS'}}\Omega^\delta_{\bS'}(\K_\rE,\K) \Bigr)\wedge d\log \K_{\rm E} .
\end{split}
\ee
\ep

\bpr  
Follows from  (\ref{EC}) and (\ref{-1eq}). 
  \epr

   \subsubsection{Cutting along a simple closed geodesic.} 
\label{ssscs}

For the punctured disc   with $n$ marked boundary points $\D_n^*$,
the dimension of moduli space ${\cal P}_{\D_n^*}$ is $2n$. There is canonical   ideal triangulation  shown on Figure \ref{sz1}. Let $\K_1,...,\K_n$ and $\B_1,...,\B_n$ be the $\K$- and $\B$-variables at the boundary intervals, and $\X_1,...,\X_n$ the variables at the internal edges of the triangulation. 
Then   $\{\X_1,...,\X_n,\B_1,...,\B_n\}$ are the cluster Poisson coordinates on the moduli space ${\cal P}_{\D_n^*}$  \cite[Section 13]{GS19}. Let $\L_1=e^{l_1}$ be the product of the cluster Poisson coordinates at the edges sharing the  puncture. Let {\em the $\rE$-parameters} be
\[\rE:= 
 \K_1  \cdots \K_{n}.\]
By \cite[Section 20.3.4]{GS19}, the functions $\rE$ and $\L_1$ are algebraically independent and generate the center of the Poisson algebra ${\cal O}({\cal P}_{\D_n^*})$. We call them the Casimirs. 
Thus the cluster Poisson structure induces a symplectic $2$-form $\omega_{\rm cl}$ on the fiber ${\cal P}_{\D_n^*}(\rE,\L_1)$ where $\rE$ and $\L_1$ are fixed.

For a genus $g$ ideal hyperbolic surface $\bS$  with ideal crowns $\rC_1,\ldots, \rC_r$ with $(n_i)_{i=1}^r$ cusps and $m-r$ boundary geodesic circles, the dimension of moduli space ${\cal P}_{\bS}$ is $6g-6+3m+2\sum_{i=1}^r n_i$. We denote the $\rE$-parameters of $\rC_1,\ldots, \rC_r$ by $\rE=(\rE_1,\ldots,\rE_{r})$. Let $\L=(l_{r+1},\ldots, l_m)$ be the geodesic lengths of the boundary geodesic circles, and 
$(\L_{r+1}, ..., \L_{m})$ their exponents.  By \cite[Section 20.3.4]{GS19}, the functions $\rE_1,...,\rE_r$ and $\L_{r+1},...,\L_m$ are algebraically independent Casimirs, generating the center of the Poisson algebra ${\cal O}({\cal P}_{\bS})$. Thus the cluster Poisson structure on ${\cal P}_{\bS}$ induces a symplectic $2$-form $\omega_{\rm cl}$ on ${\cal P}_{\bS}(\rE,\L)$.
 \begin{figure}[ht]
	\centering
	\includegraphics[scale=0.33]{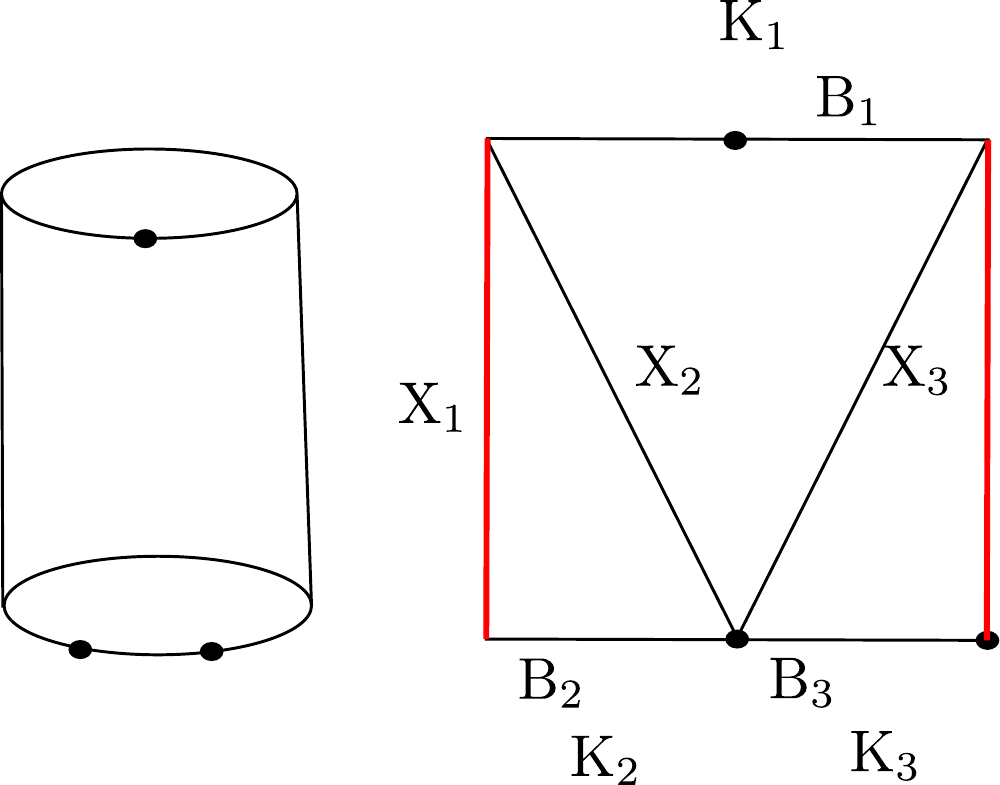}
	\small
	\caption{The cluster Poisson coordinates,  the $\K-$coordinates, and the Casimirs for the annulus $\A_{1,2}$.}
	\label{figure:A12}
\end{figure}

Here is  an example.  The cluster Poisson coordinates on moduli space $\mathcal{P}_{\A_{1,2}}$ for the annulus $\A_{1,2}$ with $2+1$ boundary points, for the triangulation on Figure \ref{figure:A12}, are $\X_1$, $\X_2$, $\X_3$, $\B_1$, $\B_2$, $\B_3$. The $\K-$coordinates are monomials in the cluster Poisson coordinates:
\[\K_1= \B_1^2 \X_1 \X_2 \X_3,\;\;\; \K_2= \B_3 \X_3 \X_2 \B_2,\;\;\; \K_3= \B_2 \X_1 \B_3.\]
The $\K_1$ and $ \K_2 \K_3$ are the Casimirs. 

 For a simple closed geodesic $\gamma$, let  $\bS'$ be the decorated surface obtained by cutting $\bS$ along $\gamma$, which has two new simple closed geodesics $\gamma'$, $\gamma''$ corresponding to the original $\gamma$. We define a rational map
${\cal P}_{\bS} \lra {\cal P}_{\bS'}$
  by restricting a local system  on $\bS$ to $\bS'$, and inducing the decorations and framing. 
Its image lies in the subspace ${\cal P}^\delta_{\bS'}$, defined by the equation $l_\gamma=l_{\gamma'}\not = 0$. We arrive at  the rational  {\it cutting  map}   
\be \la{cut3}
{\rm cut}_\gamma: {\cal P}_{\bS} \lra {\cal P}^\delta_{\bS'}.
\ee 

\begin{proposition}\la{propcc} There is a positive constant $c_{\bS,\gamma}$ depending only on $\bS$ and $\gamma$ such that 
\[\Omega_\bS(\rE,\L) =  c_{\bS,\gamma} \cdot{\rm cut}^*_{\gamma} \Bigl(\Omega^\delta_{\bS'}(\rE,\L,\L_\gamma) \Bigr)\wedge d \l_\gamma \wedge d \theta_\gamma,\]
restricts to 
\[\Omega_\bS(\K,\L) =  c_{\bS,\gamma} \cdot{\rm cut}^*_{\gamma} \Bigl(\Omega^\delta_{\bS'}(\K,\L,\L_\gamma) \Bigr)\wedge d \l_\gamma \wedge d \theta_\gamma.\]
\end{proposition}

\begin{remark}
We conjecture that $c_{\bS,\gamma}=2^{ \pi_0(\bS)-\pi_0(\bS-\gamma)}$. We checked that it holds for any surface with punctures but without marked boundary points. The $n$-holed sphere case is proved in Appendix \ref{APPB}.
\end{remark}
\bpr
To follow the notation in \cite{Ch20}, we will use the lambda lengths parameters, expressed via our $\K-$coordinates by $\lambda_x=\K_x^{-\frac{1}{2}}$. Then, as in Figure \ref{figure:GFGeq}, we have
$$
\B_a=\frac{\lambda_b}{\lambda_a \lambda_e},\;\;\; \B_b=\frac{\lambda_e}{\lambda_a \lambda_b},\;\;\; \X_e=\frac{\lambda_a \lambda_c}{\lambda_b \lambda_d},\;\;\; \X_d=\frac{\lambda_e \lambda_f}{\lambda_c \lambda_g}.
$$ 

The Goldman bracket between two geodesic/$\lambda$-lengths $\{\cdot,\cdot\}$ is given in \cite[formula 2.1]{Ch20}. When two arcs $x, y$ intersect at a boundary cusp $p$ and the arc $x$ orients clockwise towards the arc $y$
we have 
\be
\label{bclam}
\{\lambda_x,\lambda_y\}=\frac{1}{4}\lambda_x \cdot \lambda_y,
\ee
 \begin{figure}[ht]
	\centering
	\includegraphics[scale=0.33]{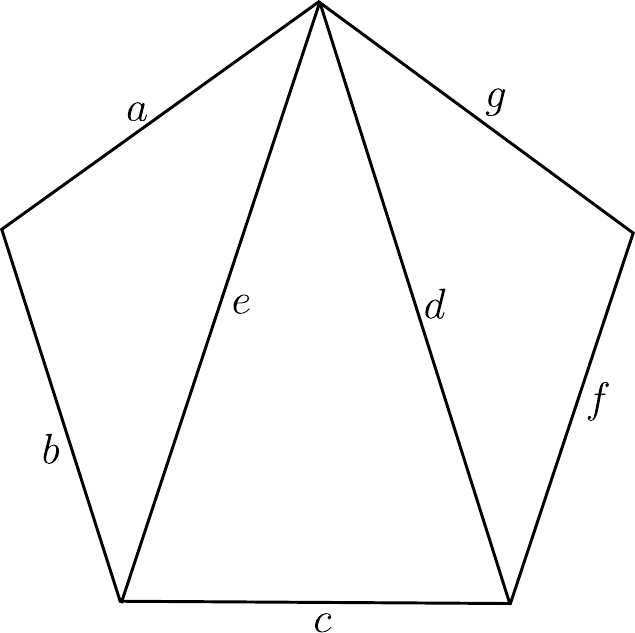}
	\small
	\caption{The $5$-gon.}
	\label{figure:GFGeq}
\end{figure}
Then by formula \eqref{bclam} we obtain
\be
\{\B_a,\B_b\}=\B_a \B_b, \;\;\;
\{\B_b,\X_e\}=\B_b \X_e,\;\;\; \{\X_d,\X_e\}=\X_d \X_e.
\ee
So cluster Poisson structure on ${\cal T}_{\bS}\subset {\cal T}^{\rm en}_{\bS}={\cal P}_{\bS}(\mathbb{R}_{>0})$ coincides with the Goldman Poisson structure defined by curves and their intersections.

 The neck geodesic circles $\gamma_1,...,\gamma_r$ cut $\bS$ into $S_{g,m}\cup \rC_1\cup ...\cup \rC_r$. Let $l_1,...,l_r, \theta_1,...,\theta_r$ be the length and twist parameters of $\gamma_1,...,\gamma_r$ and set $\L_i=e^{l_i}$. Then ${\cal T}_{\bS}$ is parametrised by the following parameters: 
\begin{itemize}
\item
$l_1,...,l_r, \theta_1,...,\theta_r$, 
\item  the parameters $l_{\alpha_1},..., l_{\alpha_{3g-3+m}}$, $\theta_{\alpha_1},..., \theta_{\alpha_{3g-3+m}}$ for ${\cal T}_{S_{g,m}}(\L_1,...,\L_r)$ with the pants decomposition $\{\alpha_{1},...,\alpha_{3g-3+m}\}$,
\item  the parameters for the spaces ${\cal T}_{\rC_1}(\L_1)$, $...$, ${\cal T}_{\rC_r}(\L_r)$.
\end{itemize}
 By \cite[Lemmas 4.2,  4.7]{Ch20}, for any $i=1,...,r$, the twist function $\theta_i$ can be written explicitly as a function of curves near $\gamma_i$. Then by \cite[Section 3 and 4]{Ch20}, 
we have 
\[\{l_i,\theta_i\}=1,\ \ \  \{u, \theta_i\}=0, \ \ \ \{l_i, u\}=0,\]
for any parameter $u$ of ${\cal T}_{S_{g,m}}(\L_1,...,\L_r)$ or any $u$ of $l_1,..., \widehat l_{i},...,l_r$, $\theta_1,...,\widehat \theta_{i},...,\theta_r$. 
When $\gamma_i$ is not the neck geodesic of the crown $\rC_j$, by adding decoration at $\gamma_i^+$, we express any cluster coordinate $u$ of ${\cal T}_{\rC_j}(\L_j)$ as a ratio of products of $\lambda-$lengths of arcs in $\rC_j$. Since the curve $\gamma_i$ does not intersect with the arcs for $u$, 
\[\{u, \theta_i\}=0, \ \ \ \{l_i, u\}=0.\]
When $\gamma_i$ is the neck geodesic of the crown $\rC_j$, the ideal arc $\ell_0$ starting and ending at a cusp $p$ cuts $\rC_j$ into $\D_{n_j+1}\cup \D_1^*$. The parameter $\K_0$ is the $\K$-parameter along $\ell_0$. Then ${\cal T}_{\rC_j}(\L_j)$ is parametrised by $\K_0,\K_1,...,\K_{n_j}$, $\X_1,...,\X_{n_j-2}$ for an ideal triangulation of $\D_{n_j+1}$. The length function $l_i$ is a Casimir  on ${\cal T}_{\rC_j}$ for the cluster Poisson structure. The term of twist function $\theta_i$ in \cite[Lemmas 4.2, 4.7]{Ch20} intersecting with $\rC_j$ is $\X_{A^k B}:=\frac{\lambda_{A^k B}}{\lambda_{0}}$, where $A^k B$ is some curve starting and ending at $p$, surrounding some loop as in \cite[Figure 9]{Ch20}, and $\lambda_{0}=\K_0^{-\frac{1}{2}}$. Both $A^k B$ and $\ell_0$ intersect with $\D_{n_j+1}$ at $p$, thus $\{\X_{A^k B}, u\}=0$
for any paramter $u$ of ${\cal T}_{\rC_j}(\L_j)$. So 
$\{\theta_i, u\}=0$
for any paramter $u$ of ${\cal T}_{\rC_j}(\L_j)$......
Hence we can write twice the cluster symplectic form $2\omega_{\rm cl}$ on ${\cal T}_{\bS}(\rE,\L)$ as:\footnote{The factor $2$ comes from the comparision between the Weil--Petersson form \cite[Appendix A]{Pen92} and the cluster Poisson structure.}
\be
\label{eq:wcut}
2\omega_{\rm cl}=\sum_{i=1}^{r}d l_i \wedge d \theta_i+\omega_{\rm gm}+\sum_{i=1}^r \omega_{i},
\ee
where $\omega_{\rm gm}=\sum_{i=1}^{3g-3+m} d l_{\alpha_i} \wedge d \theta_{\alpha_i}$ is the Weil--Petersson form on ${\cal T}_{S_{g,m}}(\L_1,...,\L_r,\L)$ for the pants curves $\alpha_i$, and $\omega_{i}$ is twice the cluster symplectic form on ${\cal T}_{\rC_i}(\rE_i,\L_i)$, which has a normalized expression in Darboux coordinates. 
The above description of the cluster/Goldman Poisson structure naturally extends from ${\cal T}_{\bS}$ to ${\cal P}_{\bS}$.
Thus we define the {\em Weil--Petersson volume form} on ${\cal P}_{\bS}(\rE,\L)$ to be 
$$
\Omega^{\rm WP}_\bS(\rE,\L):=\frac{(2\omega_{\rm cl})^d}{d!}, \ \ \ \ \ \ d:= \frac{1}{2}{\rm dim}{\cal P}_{\bS}(\rE,\L).
$$

Suppose $\gamma$ is one of the pants curves $\alpha_j$. Then by direct computation, we get
\be
\la{vel0}
\Omega^{\rm WP}_\bS(\rE,\L)=\Omega^{\rm WP}_{\bS'}(\rE,\L') \wedge d l_\gamma \wedge d \theta_\gamma,\ee
where $\L'=(\L,l_{\gamma'},l_{\gamma''})$.
Now let us relate the Weil--Petersson volume form $\Omega^{\rm WP}_\bS(\rE,\L)$ to the cluster volume form $\Omega_\bS(\rE,\L)$. Given an ideal triangulation $\mathcal{T}$ of the decorated surface $\bS$, we have the cluster Poisson structure on ${\cal P}_{\bS}$ induced by the cluster quiver $\epsilon_{ij}$ on $\bS$ with respect to the orientation of $\bS$. The cluster volume form is $\Omega_\bS=2^{\pi_0(\bS)} \bigwedge_i d x_i$ where $x_i$ is the log of the cluster variable $\X_i$. Then we choose $2d$ variables $\{y_i\}_{i=1}^{2d}$ among them such that the determinant $\det_\bS$ of sub matrix of the skew-symmetric matrix for the quiver is non-zero. Thus we have
\be
\la{vel1}
\frac{\bigwedge_{i=1}^{2d} d y_i}{\sqrt{\det_\bS}}= \frac{\omega_{\rm cl}^d}{d!}.
\ee
On the other hand, by
\[
\Omega_\bS(\rE,\L) \wedge \bigwedge_{i=1}^{r} d \log \rE_i  \wedge \bigwedge_{i=r+1}^{m} d l_i=2^{\pi_0(\bS)}\bigwedge_i d x_i.
\]
we get 
\be
\la{vel2}
\bigwedge_{i=1}^{2d} d y_i=d_\bS \Omega_\bS(\rE,\L),
\ee
for some positive rational constant $d_\bS$.
 Combining with equations (\ref{vel1})(\ref{vel2}), we obtain
\be
\la{vel3}
\frac{2^d d_\bS}{\sqrt{\det_\bS}}\Omega_\bS(\rE,\L)= \frac{\omega_{\rm WP}^d}{d!}.
\ee
Since $\frac{\omega_{\rm WP}^d}{d!}$ and $\Omega_\bS(\rE,\L)$ are invariant under the cluster transformations, the number $\frac{2^d d_\bS}{\sqrt{\det_\bS}}$ is also invariant under the cluster transformations, hence a constant depending only on $\bS$.
Plugging equation (\ref{vel3}) into equation (\ref{vel0}), we get 
\[\Omega_\bS(\rE,\L) =  c_{\bS,\gamma} \cdot{\rm cut}^*_{\gamma} \Bigl(\Omega^\delta_{\bS'}(\rE,\L,\L_\gamma) \Bigr)\wedge d \l_\gamma \wedge d \theta_\gamma,\]
where $c_{\bS,\gamma}=\frac{2 d_\bS \sqrt{\det_{\bS'}}}{d_{\bS'} \sqrt{\det_{\bS}}}$ depending only on $\bS$ and $\gamma$.

\epr
\begin{proof}[Proof of Theorem \ref{NRF2}]
By applying Proposition \ref{propcc} to the neck geodesics, combining with \cite[formula 7.5]{Mir07a}, we obtain Theorem \ref{NRF2}.
\end{proof}

 \section{The exponential volumes of moduli spaces and ${\cal B}-$functions}  \la{SecEVII} 
  \subsection{${\cal B}-$functions of decorated surfaces}

Recall the parameters $\K$  in (\ref{KC}), and 
 the  ${\cal B}-$function from Section \ref{SECTT1.4}:
 \be \la{EIa}
\begin{split}
{\mathcal{B}}_\bS({\K}, s_1, ..., s_m; \hbar) :=&
\int_{{\cal M}^{\circ \circ}_\bS}e^{-W/\hbar} \ 
e^{-(l_1 s_1 + ... +l_m s_m)/2} \ \Omega_\bS({\K}).
\\
\end{split}
\ee

\paragraph{Example 1.} Let $\bS= \D^*_n$ be the  punctured disc  with   $n$ marked points $x_1, ..., x_n$, see Figure \ref{sz1}. 
It has an ideal triangulation given by the internal edges ${\rm E}_1, ..., {\rm E}_n$ connecting the puncture  with the points $x_1, ..., x_n$, and 
the boundary edges $\F_i = x_{i-1}x_{i}$. 
 The pure mapping class group  is trivial. So the Teichm\"uller and the moduli spaces are the same. 
 They parametrise ideal hyperbolic structures on the crown with a choice of the eigenvalue of the monodormy around the boundary loop. 
  Denote by $\X_i$ and $\B_i$ the  cluster Poisson coordinates  assigned to the edges ${\rm E}_i$ and  $\F_i$. Then the cluster volume form is 
$$
\Omega_{{\cal P}_{\D_n^*}}= 2 d\log \X_1 \wedge \ldots \wedge d\log \X_n \wedge  d\log \B_1\wedge \ldots \wedge d\log \B_n.
$$
The $\K-$variable along $x_{i-1} x_{i}$, see Figure \ref{sz1}, is calculated by: 
\be \la{KXBB}
\K_{i+1} = \B_{i}\X_{i}\B_{i+1}, \quad i \in \Z/n\Z.
\ee
  \begin{figure}[ht]
\centerline{\epsfbox{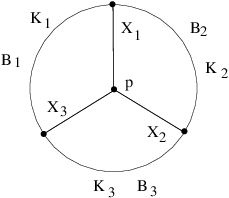}}
\caption{Cluster Poisson coordinates $\B_i, \X_i$ for the punctured disc with $3$ marked points. The variables $\K_{i+1}:= \B_i \X_i\B_{i+1}$  are parameters. The partial potentials are $W_i:= \B_i+\B_i\X_i$. }
\label{sz1}
\end{figure} 
Then 
\be
\X_i = \K_{i+1}/(\B_i\B_{i+1}).
\ee
The potential at the cusp $x_i$ is given by 
\be
W_i:= \B_i + \B_i\X_i = \B_i + \frac{\K_{i+1}}{\B_{i+1}}.
\ee
Therefore  the total potential is given  in the coordinates $\{\B_i, \K_i\}$ by
\be
W:= \sum_{i\in \Z/n\Z}W_i= \sum_{i\in \Z/n\Z} \left(\B_i + \frac{\K_i}{\B_{i}}\right).
\ee
We set 
\be \la{MNOaa}
\begin{split}
&{\bf B}:= \B_1 \cdot \cdot \cdot \B_n, \qquad  {\bf K}:=  \K_1 \cdot\cdot\cdot \K_n.  \\
\end{split}
\ee
The square $\L$ of the largest eigenvalue of the monodromy around the boundary loop is  calculated from the following formulas: 
\be \la{LX}
\L := e^l = \X_1\cdots \X_n={\bf K}/{\bf B}^2.
\ee
So the geodesic length $l$ of the neck geodesic $\ell$  is $l = \log (\X_1 
\cdots \X_n)$. \vskip 2mm

We consider the  variables  $\K = (\K_1, \ldots \K_n)$ as fixed parameters. Then the fiberwise volume form is 
$$
\Omega_\bS(\K) :=2d\log \B_1 \wedge \ldots \wedge d\log \B_{n}.
$$
Indeed, it satisfies the following relation,   where according to (\ref{KXBB}), we have $\K_{i +1}= \X_i\B_i\B_{i+1}$:
$$
d\log \K_1 \wedge \ldots \wedge d\log \K_n \wedge \Omega_\bS(\K)  =  2 d\log \X_1 \wedge \ldots \wedge d\log \X_n \wedge  d\log \B_1\wedge \ldots \wedge d\log \B_n.$$

Recall the Bessel function: 
 \be \la{f22*}
\begin{split}
J_s(z):=&\int_{0}^\infty {\rm exp}\Bigl({-\sqrt{z}(\lambda+\lambda^{-1}} )\Bigr)\lambda^{s} d\log {\lambda} \ = 
\ z^{-s/2}\cdot \int_{0}^\infty {\rm exp}\Bigl({-t-\frac{z}{t}} \Bigr)t^{s} d\log {t}.\\
\end{split}
\ee

\paragraph{1.}  For the  decorated surface $\D_1^*$ the cluster Poisson coordinates are $\B, \X$. So   
$
\L = e^l = \X=\K/\B^2, 
$ and\footnote{The volume form is obtained by dividing $d\log \B \wedge d\log \X$ by $ d\log \K = d\log (\X/\B^2)$. So it is $d\log \B$.}
\be \la{89}
\begin{split}
  {\cal B}_{{\rm D}^*_1}(\K, s) =\ &\int_{0}^\infty e^{-(\B+\frac{\K}{\B})} \L^{-s/2} \cdot 2d \log \B= 2 \int_{0}^{\infty} e^{-(\B+\frac{\K}{\B})} \left(\frac{\K}{\B^2}\right)^{-s/2}  d\log \B 
=\  2 J_{s}(\K). \\
\end{split}
\ee
This recovers formula (\ref{FBFa}).

\paragraph{2.}  Let  $\bS={\rm D}_n^*$.  
We set  ${\Bbb K}:=\{\K_1, \ldots \K_n\}$, and introduce the  notation
\be \la{MNOa}
\begin{split}
& J_{s}({{\Bbb K}}):=    J_{s}(\K_1) \cdot \cdot \cdot  J_{s}(\K_n). \\
\end{split}
\ee
Then we have, using (\ref{LX}):
\be \la{56}
\begin{split}
  {\cal B}_{\D_{n}^*}({\K}, s) = \ &\int e^{-W} \L^{-s/2}\Omega_{{\rm D}_n^*} \\
= \ & \int_{\B_i>0}  e^{-(\B_1+\frac{K_1}{\B_1} +\ldots + \B_n + \frac{K_n}{\B_n}) }  \left( \frac{{\bf K}}{{\bf B}^2}\right)^{-s/2} \cdot 2 d\log \B_1 \wedge \ldots \wedge d\log \B_{n} \\
= \ & 2 J_{s}({\K_1})\cdots J_{s}({\K_n}) \stackrel{(\ref{MNOa})}{=}:\ 2 J_{s}({\Bbb K}).\\
\end{split}
\ee

 Recall that $r$ is the number of crown ends of a decorated surface $\bS$, so that $m-r$ is the number of punctures on $\bS$. 
 We denote by $J_{s_i}({\Bbb  K}_i)$ the function (\ref{MNOa}) for the $i-$th crown, where $1 \leq i \leq r$. 
Recall $\rC_{\bS, \ell}:=2^{-\mu_\ell} c_{\bS, \ell}$ where $\mu_\ell$ is the number of one-holed tori cutting off by $\ell$ and $c_{\bS, \ell}$ is a positive constant depending only on $\bS$ and $\ell$ defined in Proposition \ref{propcc}.

\bt \la{MTHOE} The   integral ${\cal B}_{\bS}({\K}, s_1, ... , s_m)$ is finite if ${\rm Re}(s_i>0)$, and can be analytically continued and calculated 
as  follows, where $|k|=k_1+...+k_m$:
\be \la{MFOEa}
\begin{split}
 &{\cal B}_{\bS}({\K}, s_1, ... , s_m) \ = \rC_\bS  \sum_{k_1+ ... + k_m\leq 3g-3+m} {\cal V}_{g, k_1, ..., k_m} 2^{m} \ 
\prod_{i=1}^r-\Bigl(2\frac{d}{ds_i}\Bigr)^{2k_i+1}   J_{s_i}({\Bbb  K}_i)\cdot  \prod_{j=r+1}^m \frac{(2k_j)!}{s_j^{2k_j+1}}.     \\
\end{split}
 \ee
 \et

\bpr

By Mirzakhani's theorem (\ref{V1}) the volume of the moduli space $\mathcal{M}_S(\L)$ of the genus $g$ surfaces $S$ with boundary  geodesics of the  
lengths $l_1, ..., l_m$ is given by: 
\be \la{V1*}
{\rm Vol}(\mathcal{M}_S(\L)) = \sum_{k_1, ..., k_m \leq 3g-3+m} {\cal V}_{g, k_1, ..., k_m}l_1^{2k_1} \cdots  l_m^{2k_m}.
\ee
Therefore the enhanced variant of the  neck recursion formula \eqref{NRF2*} implies the following  formula:
  \be \la{MFOE}
\begin{split}
  {\cal B}_{\bS}({\K}, s_1, ..., s_m) =  \rC_\bS \sum_{k_1+ ... + k_m\leq 3g-3+m} {\cal V}_{g, k_1, ..., k_m} \cdot &\\
 \prod_{i=1}^r \int_{-\infty}^{\infty} { \rm Vol}_{\mathcal{E}}(\D_{n_i}^*)({\K}_i,l_i) \ e^{-l_is_i/2} \ l_i^{2k_i+1} d l_i\cdot & \prod_{j=r+1}^m \int_0^\infty e^{-l_js_j/2} \ l_j^{2k_j}dl_j. \\
\end{split}
 \ee
Using formula (\ref{56}), the  integral for the punctured disc ${\rm D}_n^*$  appearing in   (\ref{MFOE})  is calculated as follows:
 \be
\begin{split}
& 
 \int e^{-W_{}} e^{-ls/2} \ l^{2k+1} \Omega_{{\rm D}_n^*}
= 2 \Bigl(-2\frac{d}{ds}\Bigr)^{2k+1}  
    J_{s}({\Bbb K}). \\
        \end{split}
\ee   
 Theorem \ref{MTHOE} is proved. \epr

A similar approach to the function ${\cal L}_\bS(\K; s)$ does not lead to a formula expressing it via Bessel functions.
 
  \paragraph{Example 2.} For a once crowned torus $\bS$ with $n$ cusps on  Figure \ref{figure:oct}, we have
\be
\begin{split}
   {\cal B}_{\bS}({\K}, s)\ = \  & \frac{1}{2} \left(\frac{\pi^2}{6}\int_{-\infty}^{+\infty}   { \rm{Vol}}_{\mathcal{E}}({\cal M}_{\D_{n}^*})({\K},l)\ e^{-ls/2} \ l \ d l+\frac{1}{24}\int_{-\infty}^{+\infty}  {\rm{Vol}}_{\mathcal{E}}({\cal M}_{\D_{n}^*})({\K},l)\ e^{-ls/2} \ l^3 \ d l\right)\\
=\ &  -\frac{1}{3}\Bigl(\pi^2 \frac{d}{ds}+ \frac{d^3}{ds^3}\Bigr)J_{s}({\K}). \\
\end{split}
\ee  
\begin{figure}[ht]
	\centering
	\includegraphics[scale=0.3]{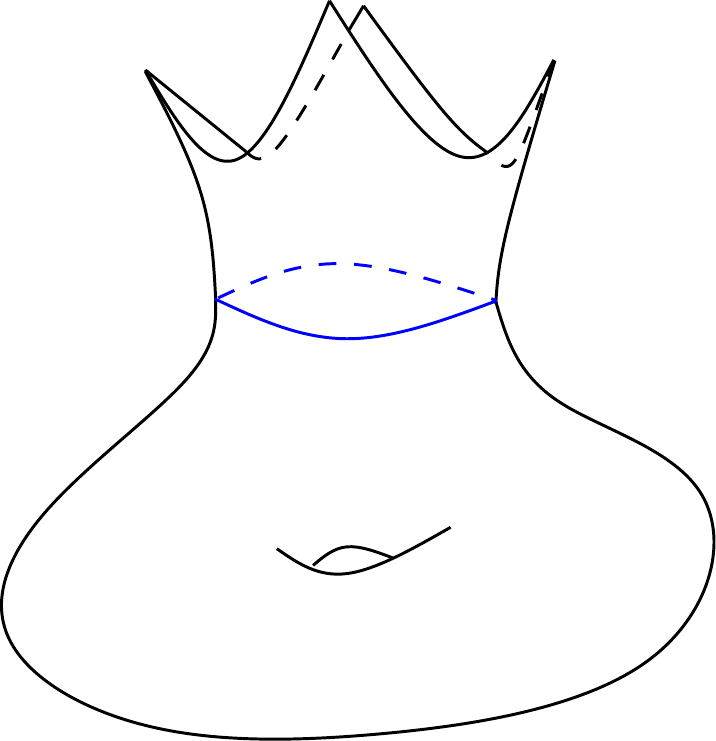}
	\small
	\caption{Once crowned torus. }
	\label{figure:oct}
\end{figure}
Indeed,
 the Weil--Petersson volume of the moduli space ${\cal M}_{1,1}(l) $ of hyperbolic structures on a torus with a single hole of the geodesic length $l$  is given by   \cite[Page 180 Table 1]{Mir07a} by the formula
 $$
{\rm Vol}({\cal M}_{1,1}(l) )= \frac{\pi^2}{6} + \frac{l^2}{24}.
$$

\subsection{The exponential volume of the crown} \la{Sec3.1}

Recall the Bessel function
\be
\begin{split}
J_0(z)= \int_{0}^\infty {\rm exp}\Bigl({-t-\frac{z}{t}} \Bigr) d\log {t}.\\
\end{split}
\ee
We use below both notations $\D_n^*({\K}; \L) = \D_n^*({\K}; l)$, where $\L$ and $l$ are related by $\L=e^l$.
 \vskip 1mm
 
  To get the volume ${\rm Vol}_{\mathcal{E}}(\D_n^*)({\K}; \L)$, we integrate the  form  over the positive locus  $B_i \geq 0$ the form
  $$
  \Omega_\B':= 2 d\log \B_1 \wedge \ldots \wedge d\log \B_{n-1}.
  $$
  \bl For $n=2$, we get
\be \la{FLe4.1}
{\rm{Vol}}_{\mathcal{E}}(\D_2^*)({\rm K}_1, \K_2; \L) = J_0(\K_1+\K_2+\sqrt{\K_1 \K_2}(\L^{1/2}+ \L^{-1/2})).
\ee 
\el

\bpr 
Since $\X_1=\K_2/(\B_1 \B_2)$ and $\X_2=\K_1/(\B_1 \B_2)$, we obtain
$\B_1 \B_2=e^{-l/2}\sqrt{\K_1 \K_2}$. 
So the integral is  
\be
\begin{split}
	&\int_{0}^\infty{\rm exp}\Bigl(-\B_1-\K_1/\B_1-e^{-l/2}\sqrt{\K_1 \K_2}/\B_1- \B_1e^{l/2} \K_2/\sqrt{\K_1 \K_2}\Bigr) d \log \B_1 \\
	&= J_0(\K_1+\K_2+\sqrt{\K_1 \K_2 }(e^{l/2}+e^{-l/2})). 
\end{split}
\ee
\epr

\bl \la{V2} The integral $\rm{Vol}_{\mathcal{E}}(\D_n^*)({\K}; \L)$ is finite. 
\el

\bpr 
Since $\B_n + \frac{\K_{n}}{\B_{n}}> 0$, we obtain
\be
\begin{split}
	&{\rm Vol}_{\mathcal{E}}(\D_n^*)({\K}; \L)=\int_{\B_1 > 0}\cdots\int_{\B_{n-1} > 0}{\rm exp}\Bigl(-\sum_{i=1}^n \Bigl(\B_i + \frac{\K_{i}}{\B_{i}}\Bigr)\Bigr) \Omega'_\B  \\
	&< 2 \prod_{i=1}^{n-1}\int_0^\infty {\rm exp}\Bigl(-\B_i - \frac{\K_{i}}{\B_{i}}\Bigr)d\log \B_i  = 2\prod_{i=1}^{n-1}J_0(\K_i). 
\end{split}
\ee
 \epr

It seems hard to obtain an explicit formula for $n\geq 3$.

\bp
\label{proposition:f}
For any integer $k\geq 0$, the integral $\int_{0}^{+\infty} {\rm Vol}_{\mathcal{E}}(\D_n^*)({\K},l)l^k d l$ is finite. 
\ep

\bpr
Since $\B_1 \cdots \B_n=e^{-l/2}\sqrt{\K_1\cdots \K_n}$, we get
\be \la{120}
\begin{split}
	&\int_{0}^{+\infty} {\rm Vol}_{\mathcal{E}}(\D_n^*)({\K},l)l^k d l =\\
     &\int_{0}^{+\infty} \int_{\B_1 > 0}\cdots\int_{\B_{n-1} > 0}{\rm exp}\Bigl(-\sum_{i=1}^n \Bigl(\B_i + \frac{\K_{i}}{\B_{i}}\Bigr)\Bigr) \Omega'_\B \cdot l^k d l   =\\ 
     &\int_{\B_1 > 0}\cdots\int_{\B_{n-1} > 0}{\rm exp}\Bigl(-\sum_{i=1}^{n-1} \Bigl(\B_i + \frac{\K_{i}}{\B_{i}}\Bigr)\Bigr) \int_{0}^{+\infty} 
     {\rm exp}\Bigl(-\frac{\K_1^{1/2}\cdots \K_n^{1/2}}{\B_1 \cdots \B_{n-1}}e^{-l/2}-\frac{\B_1 \cdots \B_{n-1}\K_n}{\K_1^{1/2}\cdots \K_n^{1/2}}e^{l/2}\Bigr)  l^k d l \cdot  \Omega'_\B.  
\end{split}
\ee
 Note that we have
\be
\begin{split}
	&\int_{0}^{+\infty} {\rm exp}\Bigl(-\frac{\K_1^{1/2}\cdots \K_n^{1/2}}{\B_1 \cdots \B_{n-1}}e^{-l/2}-\frac{\B_1 \cdots \B_{n-1}\K_n}{\K_1^{1/2}\cdots \K_n^{1/2}}e^{l/2}\Bigr)  l^k d l \\
     &\leq \frac{1}{k+1}\int_{0}^{+\infty} {\rm exp}\Bigl(-\frac{\B_1 \cdots \B_{n-1}\K_n}{\K_1^{1/2}\cdots \K_n^{1/2}}\frac{l^{k+1}}{2^{k+1}(k+1)!}\Bigr)  d l^{k+1}\leq {}
\frac{\rm C}{\B_1 \cdots \B_{n-1}}.  
\end{split}
\ee
 Here ${\rm C}$ is a constant. Therefore
\be
\begin{split}
	&\int_{0}^{+\infty}{ \rm Vol}_{\mathcal{E}}(\D_n^*)({\K},l)l^k d l 
     \leq 
     2 {\rm C}\cdot \prod_{i=1}^{n-1} \int_{\B_i > 0} {\rm exp}\Bigl(- \Bigl(\B_i + \frac{\K_{i}}{\B_{i}}\Bigr)\Bigr) \B_i^{-2}  d \B_i  
\end{split}
\ee
which converges.
\epr

\subsection{The general exponential volumes}  \la{Sec3}

Consider the  genus $g$ decorated surface $\bS$ with $m$ boundary components. 
Denote by $n_1, ..., n_r$ the numbers of cusps on the crowns.  
Recall notation (\ref{KC}) for ${\K} = \{{\K}_j\}$. 
Recall $\rC_{\bS, \ell}:=2^{-\mu_\ell} c_{\bS, \ell}$ where $\mu_\ell$ is the number of one-holed tori cutting off by $\ell$ and $c_{\bS, \ell}$ is the positive constant depending only on $\bS$ and $\ell$ defined in Proposition \ref{propcc}.
\bt \la{MTHO} The exponential volume of the moduli space ${\cal M}_{\bS}({\K}, \L)$ is finite.  
One has 
\be \la{MFO}
{\rm Vol}_{\cal E}({\cal M}_{\bS}({\K}, \L)) = \rC_\bS  \sum_{k_1+ ... + k_m\leq 3g-3+m} {\cal V}_{g, k_1, ..., k_m} \prod_{j=1}^r \int^\infty_{0}  {\rm{Vol}}_{\mathcal{E}}(\D_{n_j}^*)({\K}_j,l_j)\  l_j^{2k_j+1} d l_j. 
 \ee
\et
\bpr
The convergence follows from Proposition \ref{proposition:f}. It remains to use the neck recursion formula \eqref{NRF1}.
\epr

The  $e^{-W}$ factor is crucial. 
Without it, the volume is finite  only if there are no cusps. 

\paragraph{Examples.}
1. When $\bS$ is a once crowned pair of pants with $n$ cusps, see Figure \ref{figure:cpp}, the exponential volume is reduced to the following integral:
\be
{ \rm{Vol}}_{\mathcal{E}}({\cal M}_{\bS}({\K}; \L))= \frac{1}{2} \int_0^{+\infty}  {\rm{Vol}}_{\mathcal{E}}(\D_{n}^*)({\K},l)\ l  d l.
 \ee
\begin{figure}[ht]
	\centering
	\includegraphics[scale=0.3]{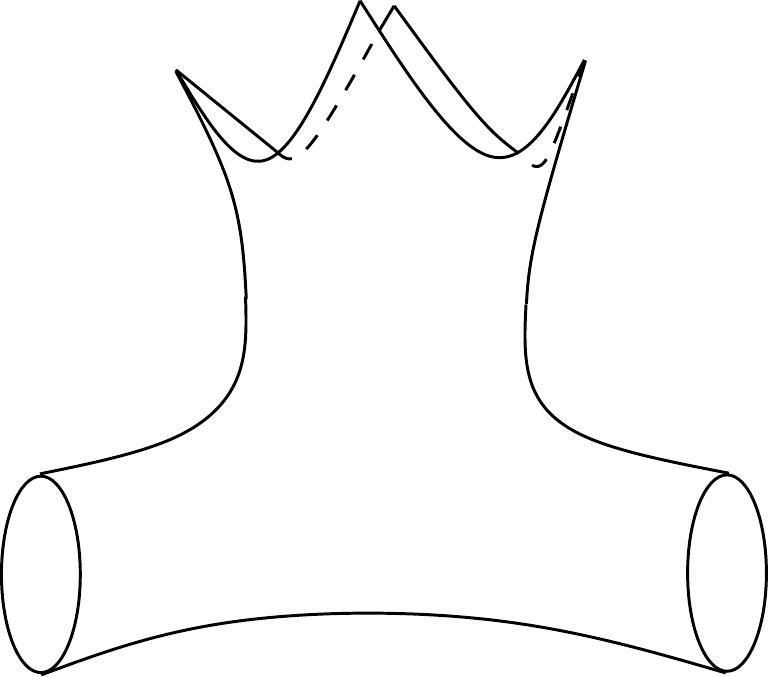}
	\small
	\caption{Once crowned pair of pants with 4 cusps. }
	\label{figure:cpp}
\end{figure}
Indeed, the  hyperbolic structure on a pair of pants is  determined by the lengths of   boundary geodesics.

i) For the next example, it is useful to note that, using the substitution $u:= 1/t$, we get
\be \la{48}
\begin{split}
&\int_{0}^{\infty} {\rm exp}\Bigl({-\sqrt{z}(t+\frac{1}{t})} \Bigr) (\log t)^{2n}d\log {t} = 2\cdot  \int_{0}^{1} {\rm exp}\Bigl({-\sqrt{z}(t+\frac{1}{t})} \Bigr) (\log t)^{2n}d\log {t}.\\
&\int_{0}^{\infty} {\rm exp}\Bigl({-\sqrt{z}(t+\frac{1}{t})} \Bigr) (\log t)^{2n+1}d\log {t} = 0.\\
\end{split}
\ee

For the  crown with one cusp,  the coordinates are $\B,\K$.   The length $l=\log (\K/\B^2)$ of the  neck  geodesic is  positive. So  $\K \geq \B^2$. We get the following integral:
\be
\begin{split}
&\int_{0}^\infty e^{-(\B+\frac{\K}{\B})}  l  d l = \int_{0}^{\K^{1/2}} e^{-(\B+\frac{\K}{\B})} \log \frac{\K}{\B^2}  d\log \frac{\K}{\B^2}\\
&=\ - 2 \log \K \cdot \int_{0}^{\K^{1/2}} e^{-(\B+\frac{\K}{\B})}   d\log \B  + 4 \cdot \int_{0}^{\K^{1/2}} e^{-(\B+\frac{\K}{\B})} \log \B \  d\log \B\\
&\stackrel{(\ref{48})}{ = }\ -  \log \K \cdot J_0(\K^{1/2}) +  4 \cdot \int_{0}^{\K^{1/2}} e^{-(\B+\frac{\K}{\B})} \log \B \  d\log \B. \\
\end{split}
\ee

\vskip 2mm

ii)  For the crown with   $n$ cusps,       
the  coordinates are $\B_1, \ldots, \B_n, \K_1, \ldots \K_n$. 
Since the length of the neck geodesic is 
 $l = \log ({\bf K}/{\bf B}^2)$ is positive, see  (\ref{MNOaa})-(\ref{LX}),  the integral is over the domain 
$$
{\cal B}_{n}:= \{0 \leq \B_1, \ldots , \B_n \ | \ {\bf B}\leq {\bf K}^{1/2}\}. 
$$
We get the integral
\be
\begin{split}
& \int_{{\cal B}_{n}}  e^{-(\B_1+\frac{\K_1}{\B_1} +\ldots + \B_n + \frac{\K_n}{\B_n}) }   l \cdot  d\log \B_1 \wedge \ldots \wedge d\log \B_{n-1} \wedge dl  = \\
&-2\int_{{\cal B}_{n}}  e^{-(\B_1+\frac{\K_1}{\B_1} +\ldots + \B_n + \frac{\K_n}{\B_n}) }  \log \frac{{\bf K}}{{\bf B}^2} \cdot  d\log \B_1 \wedge \ldots \wedge d\log \B_{n}= \\
& -2 \log ({\bf K}) \cdot \int_{{\cal B}_n} e^{-(\B_1+\frac{\K_1}{\B_1} +\ldots + \B_n + \frac{\K_n}{\B_n}) }    d\log \B_1\wedge  \ldots \wedge d \log \B_n + \\
&  4\cdot   \int_{{\cal B}_n}  e^{-(\B_1+\frac{\K_1}{\B_1} +\ldots + \B_n + \frac{\K_n}{\B_n}) }   \log ({\bf B}) d\log \B_1\wedge  \ldots \wedge d \log \B_n. \\
\end{split}
\ee
 \vskip 2mm

 2. 
The  integral we need   to apply formula (\ref{MFO}) for the exponential volume   
 of any decorated surface    is  
 \be \la{UsI}
\begin{split}
&\int_{{\cal B}_{n}} {\rm Vol}_{\mathcal{E}}(\D_{n}^*)({\K},l)\cdot l^{2k+1} \cdot d l = \\
&- 2\int_{{\cal B}_{n}}  e^{-(\B_1+\frac{\K_1}{\B_1} +\ldots + \B_m + \frac{\K_n}{\B_n}) }  \log ^{2k+1}\frac{{\bf K}}{{\bf B}^2} \log \B_1\wedge \ldots \wedge d\log \B_n. \\
        \end{split}
\ee 
Here $\K=(\K_1, ..., \K_n)$.

3. If $\bS$ is a once crowned torus with $n$ cusps as on Figure \ref{figure:oct}, the exponential volume is 
\be
\begin{split}
{\rm Vol}_{\cal E}({\cal M}_{\bS}(\K)) = &  \frac{1}{12} \int_0^{+\infty}  {\rm Vol}_{\mathcal{E}}(\D_{n}^*)({\K},l)\cdot (\pi^2 l + \frac{l^3}{4})d l.\\
\end{split}
\ee

\section{McShane-type identities for  ideal hyperbolic surfaces revisited}\la{Sec4.1}
In \cite{McS91,McS98}, McShane found a remarkable identity for the punctured hyperbolic surfaces by splitting the horocycle (or the cusp region area bounded by that horocycle) \cite[Theorem 4]{McS98}. Following the same strategy, these identities were generalized by Mirzakhani \cite[Theorem 4.2]{Mir07a} to the hyperbolic surfaces with geodesic boundary circles, and by Huang \cite[Theorem 4.5]{Hu14} to the cusps on the crowns and punctures on ideal hyperbolic surfaces. Below we revisit  the McShane identities from our perspective, and  calculate their terms 
in  cluster Poisson coordinates, which is important  for the unfolding. \

Given a puncture or a boundary cusp $p$ at the crown of the ideal hyperbolic surface, let us consider all  geodesics emitting from $p$. There are four different types of geodesics:
\begin{enumerate}
  \setcounter{enumi}{-1}
\item Bi-infinite geodesics without self-intersection;
\item Self-intersecting geodesics;
\item Simple geodesics which hit the boundary circle;
\item Simple geodesics which hit the boundary arc of some crown end.
\end{enumerate}
The first two limiting behavior of geodesics was observed in \cite{McS98}, the third one was observed in \cite[Proof of Theorem 4.2]{Mir07a} while the last one was observed in \cite[Proof of Theorem 4.5]{Hu14}.
The Birman--Series theorem \cite{BS85} tells us that the union of all bi-infinite geodesics without self-intersection is sparse in the hyperbolic surface. The Birman--Series theorem can be extended to  ideal hyperbolic surfaces, as shown in the following Theorem, proved in Appendix A as Theorem \ref{thm:bs}. 
\begin{theorem}
For an ideal hyperbolic surface, let $\mathcal{G}$ be the union of all bi-infinite geodesics without self-intersection. The area of $\mathcal{G}$ with respect to the Lebesgue measure on the surface is equal to zero.

Therefore, given a cusp $p$ of a crown, or a puncture $p$, the area of the union $\mathcal{G}_p$ of all  bi-infinite   geodesics  without self-intersection emitting from $p$ is equal to zero. For the horoarc/horocycle $h_p$ around $p$, the intersection $h_p\cap \mathcal{G}_p$ has Lebesgue measure zero.
\end{theorem}

Thus generically, a geodesic emitting from $p$ will be one of the three types (1)-(3).

 Let us denote the length of the horoarc $h_p$ by $H_p$. 
It is straightforward to see that 
the length $H_p$ is the same as the value of the local potential $W_p$ at $p$:
\be \la{EF1}
H_p = W_p.
\ee
Let us split the horoarc $h_p$ into intervals reflecting the changes of the type of geodesic emitting from $p$.

\begin{figure}[ht]
	\centering
	\includegraphics[scale=0.3]{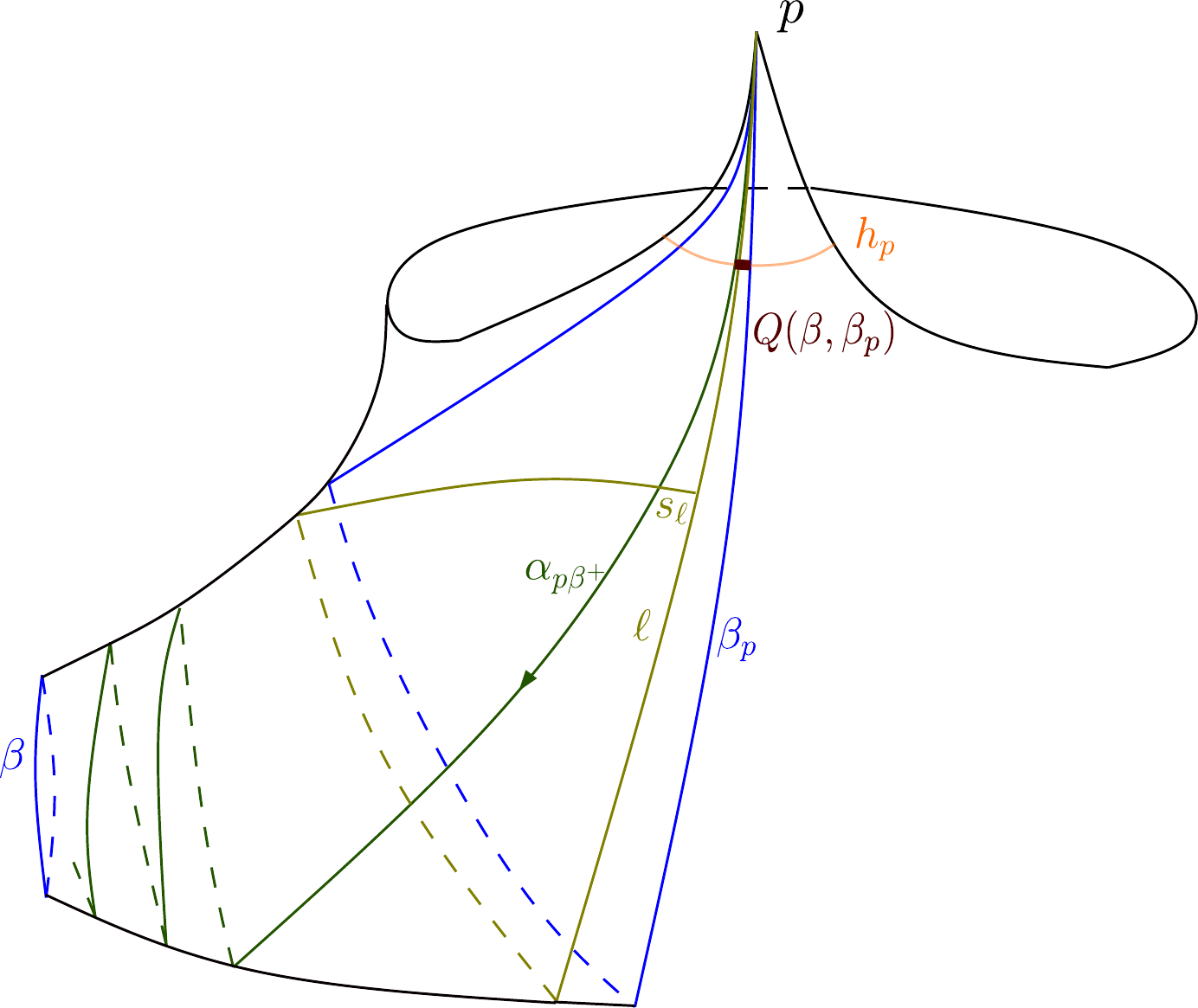}
	\small
	\caption{The geodesic $\ell$ emitting from $p$ and hitting itself at a point $s_\ell$ bounds a simple closed loop $\alpha_\ell$   homotopic to $\beta$. When  $\ell$ moves towards $\beta$, the geodesic  $ps_\ell$ converges to a bi-infinite geodesic $\alpha_{p\beta^+}$ spiraling around $\beta$. When $\ell$ moves in  the opposite direction, the geodesic $ps_\ell$ converges to a bi-infinite geodesic $\beta_p$. These geodesics $ps_\ell$    between $\alpha_{p\beta^+}$ and $\beta_p$ fill the  {\it trouser leg}  $\T(\beta,\beta_p)$. We define the {\it gap  $Q(\beta,\beta_p)$}   as the length of the interval on the horoarc $h_p$ bounded by the intersection points $h_p\cap \beta_p$ and $h_p\cap \alpha_{p\beta^+}$. The interval is shown by a solid arc. The gap  $Q(\beta,\beta_p)$ coincides with the potential 
$W_p(\alpha_{p\beta^+}, \beta_p)$.}
	\label{figure:hp}
\end{figure}

1. 
Suppose that the geodesic $\ell$ is self-intersecting.  We describe the limiting behavior of geodesics following the argument around \cite[Figure 1, 2]{McS98}. Denote by $s_\ell$ the first self-intersection point on $\ell$. Consider the loop $\alpha_\ell$ on   $\ell$ going from $s_\ell$ to itself, see  Figure \ref{figure:hp}. It is homotopic to a simple closed geodesic $\beta$.  When $\ell$ moves towards  the $\beta$,  the geodesic $p s_\ell$, obtained by moving from $p$ along $\ell$ to the first self-intersection at the point $s_\ell$,     converges to a bi-infinite geodesic  $\alpha_{p\beta^+}$ spiraling around $\beta$ infinitely many times. It provides an orientation of $\beta$. When $\ell$ moves in the opposite direction, the loop $\alpha_\ell$  does not change its homotopy class 
till $\ell$ converges to some bi-infinite geodesic $\beta_p$. The geodesic $\beta_p$ is the limit of both loop $\alpha_\ell$ and the path $p s_\ell$. 
We denote by $Q(\beta,\beta_p) $ the {\it gap}, defined as the 
 length of the horoarc between the intersection points of the  two bi-infinite geodesics 
$\alpha_{p\beta^+}$ and $\beta_p$ with the horocycle  $h_p$. 
On the other hand, the ideal triangle bounded by $(\alpha_{p\beta^+}, \beta_p)$ lifts to some $(\tilde{p},\beta^+, \beta \tilde{p})$ in the universal cover. Suppose the decoration $v_p$ at $\tilde{p}$ is the same as the decoration at $p$, the framing at $\beta^+$ and $\beta \tilde{p}$ can be obtained by parallel transporting $\left<v_p\right>$ along $\alpha_{p\beta^+}$ and $\beta_p$ respectively. Then we define the potential $W_p(\alpha_{p\beta^+}, \beta_p)$ by equation \eqref{PEVE}. 
By (\ref{EF1}), the gap is equal to the value of the potential $W_p(\alpha_{p\beta^+}, \beta_p)$. The geodesics $\beta$ and $\beta_p$ bounds an embedded {\em trouser leg},  see Figure \ref{figure:hp1}, which is denoted by 
$$
\T(\beta,\beta_p). 
$$
The decorated surface describing the trouser leg is a punctured disc with a marked point, see Figure \ref{sz6}.  
The pair $(\beta,\beta_p)$ uniquely determines a trouser leg unless we consider the once punctured torus case.\footnote{We change the terminology {\em pair of half-pants} in \cite[Definition 5.6]{HS19}  to trouser leg.}
 The isomorphism class of the ideal hyperbolic surface $\T(\beta, \beta_{p})$ is determined by  the length $l_\beta$ of the neck geodesic and 
the horocycle length ${\rm K}_{\beta_p}$ of the boundary arc. 
Let us set  
$$
\L_\beta:=e^{\l_\beta}.
$$
 \begin{figure}[ht]
	\centering
	\includegraphics[scale=0.3]{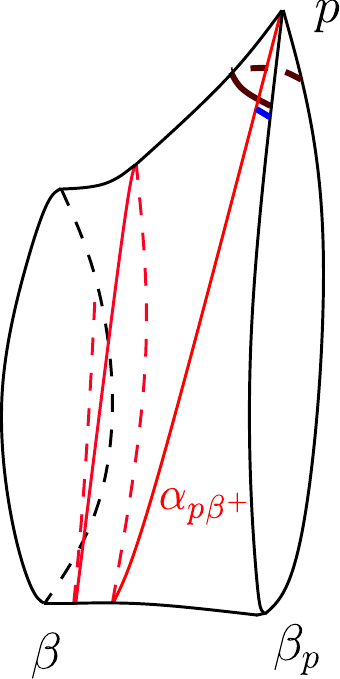}
	\small
	\caption{The trouser leg $\T(\beta, \beta_p)$.}
	\label{figure:hp1}
\end{figure}

Observe that the gap  $Q(\beta,\beta_p)$ is equal to the potential $W_p(\alpha_{p\beta^+}, \beta_p)$ 
\be \la{D=W}
Q(\beta,\beta_p) = W_p(\alpha_{p\beta^+}, \beta_p).
\ee

 \bl \la{L5.31} The potential $W_p(\alpha_{p\beta^+}, \beta_p)$  for the trouser leg $\T(\beta, \beta_p)$ is given by 
\be \la{551}
\begin{split}
W_p(\alpha_{p\beta^+}, \beta_p)=\K^{1/2}_{\beta_p} \L_\beta^{-1/2}.\\
\end{split}
\ee
Therefore the gap $Q(\beta,\beta_p)$ is given by 
\be \la{gapfunct}
\begin{split}
Q(\beta,\beta_p) 
&=\K^{1/2}_{\beta_p}\L_\beta^{-1/2}.\\
\end{split}\ee
\el

  \begin{figure}[ht]
\centerline{\epsfbox{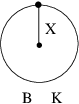}}
\caption{The decorated surface for the trouser leg is a punctured disc with a marked point.}
\label{sz6}
\end{figure} 

\bpr 
 Denote by ${\B},  {\X}$ the cluster Poisson coordinates at the two edges of the decorated surface underlying a trouser leg   on Figure \ref{sz6}. Then  
$$
{\rm K}_{\beta_p} = \B\X\B, \quad \L_\beta = \X. 
$$
The potential is given by $W_p(\alpha_{p\beta^+}, \beta_p)=\B_{\beta_p}$. 
So  we get:  
\be \la{EXVT}
\begin{split}
W_p(\alpha_{p\beta^+}, \beta_p)&=   \K_{\beta_p}^{1/2}\L_\beta^{-1/2}.\\
\end{split}
\ee
The second claim follows immediately from this and (\ref{D=W}).\epr

{\bf Remark}. 
By \cite[Proposition 1]{McS98}, the bi-infinite geodesic $\beta_p$ uniquely determines the trouser leg $\T(\beta,\beta_p)$, but $p$ and $\beta$ do not determine $\beta_p$. 
For example, let $\D_\delta \gamma_p$ be the Dehn-twist of the bi-infinite geodesic $\gamma_p$ on Figure \ref{figure:gammapdt} around the loop $\delta$. Then $\gamma_p \not = \D_\delta \gamma_p$, but $\gamma = \D_\delta \gamma$. 
The related embedded trouser leg is $\T(\gamma,\D_\delta \gamma_p)$.
 
 We could also Dehn-twist $\gamma_p$ around $\eta$. The Dehn twists around  $\delta$ and $\eta$ do not commute since they intersect. 
 So the orbit of a trouser leg by the action of the group ${\rm Mod}(\bS)$ can be  complicated. 
 \begin{figure}[ht]
	\centering
	\includegraphics[scale=0.3]{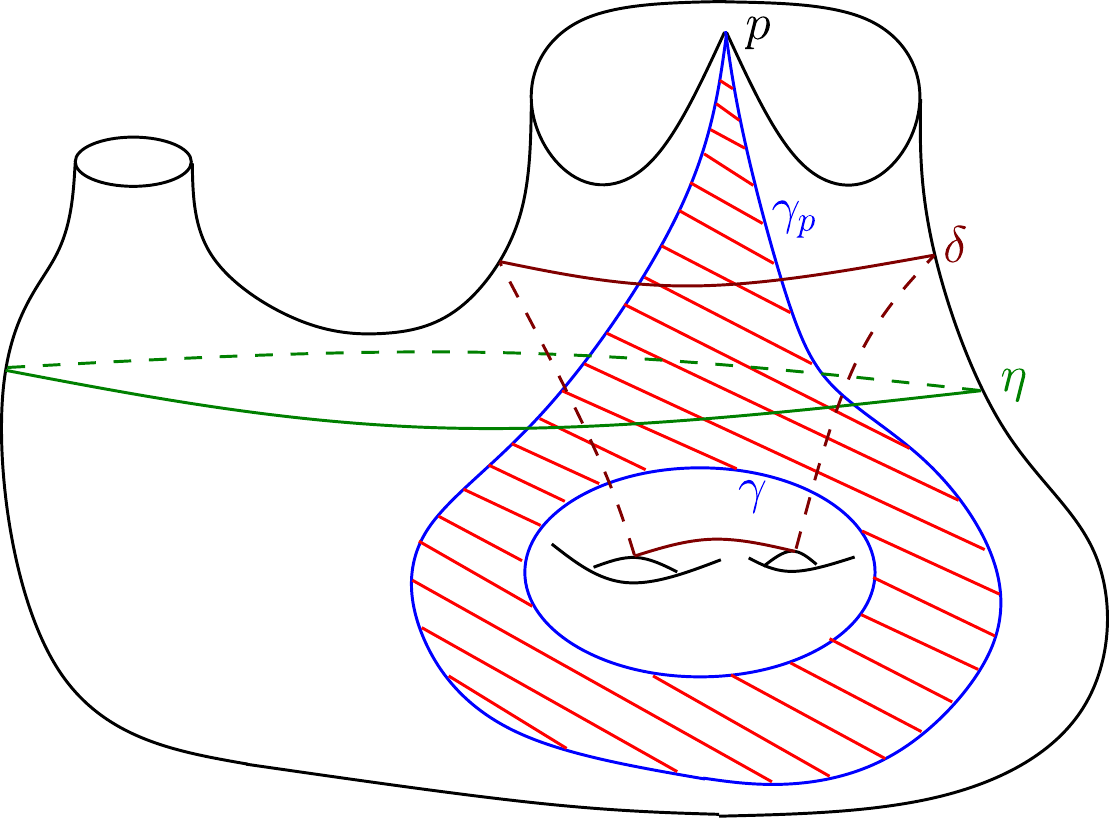}
	\small
	\caption{Trouser leg $\T(\gamma,\gamma_p)$. Dehn twists of $\gamma_p$ by $\delta$ and $\eta$ are different.}
	\label{figure:gammadt}
\end{figure}
\vskip 2mm 

2. 
Suppose that the simple geodesic $\ell$ hits the boundary circle $\beta$ as in Figure \ref{figure:rterm}. If we move $\ell$ towards the left, it will converge to some simple bi-infinite geodesic $\alpha_{p\beta^-}$ spiraling around $\beta^{-}$ - that is $\beta$ with the opposite orientation -  without changing its homotopy class. If we move $\ell$ towards the right, it will converge to a simple bi-infinite geodesic $\alpha_{p\beta^+}$ spiraling around $\beta$ without changing its homotopy class. So there is a unique embedded trouser leg   $\T(\beta,\beta_p)$ containing $\alpha_{p\beta^-}$ and $\alpha_{p\beta^+}$.

\begin{figure}[ht]
	\centering
	\includegraphics[scale=0.3]{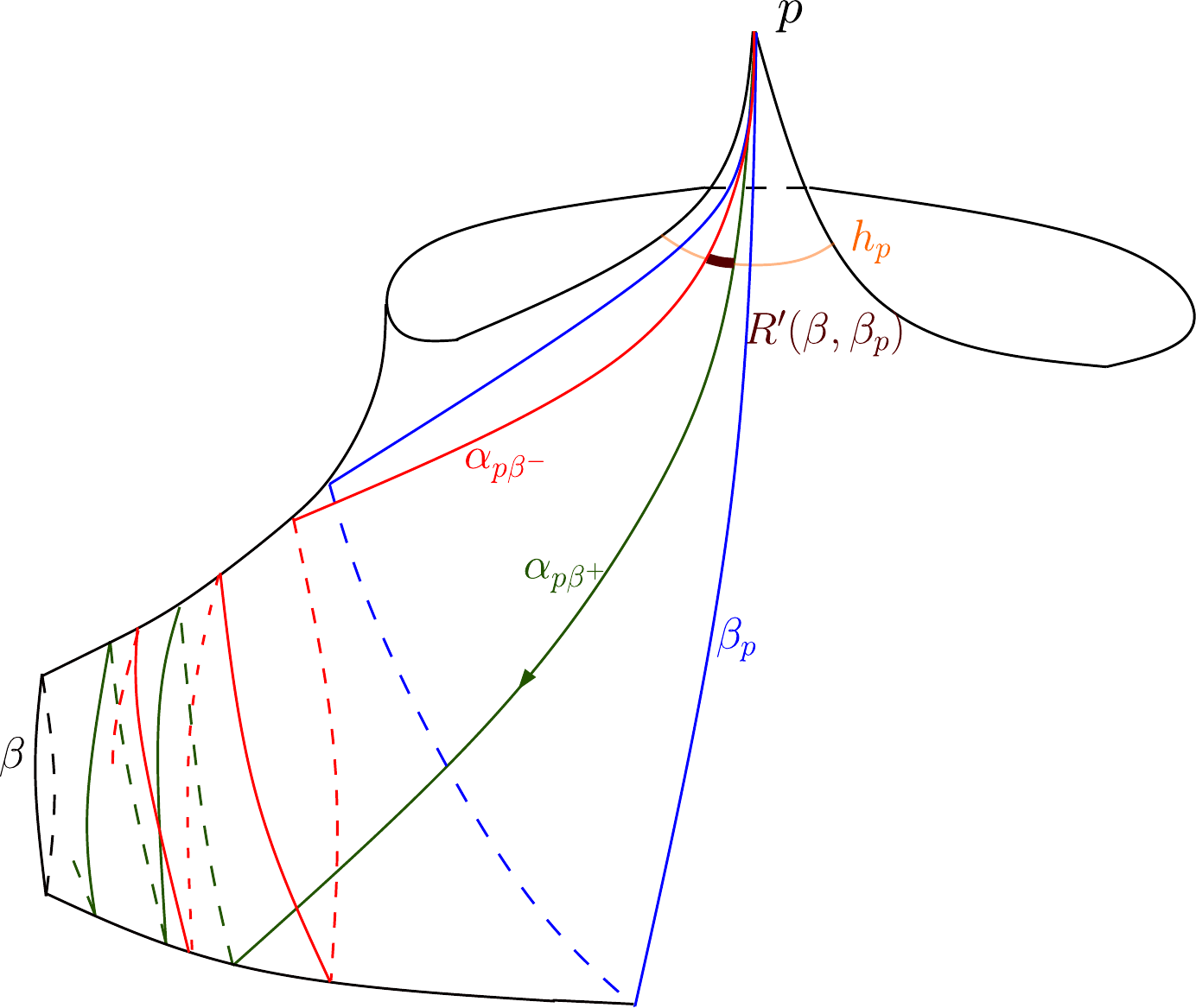}
	\small
	\caption{The geodesic $\ell$ emitting from $p$ hits the boundary circle. Moving $\ell$ between $\alpha_{p\beta^-}$ and $\alpha_{p\beta^+}$ will not change the homotopy class of $\ell$. The trouser leg containing $\alpha_{p\beta^-}$ and $\alpha_{p\beta^+}$ is denoted by $\T(\beta,\beta_p)$. The gap for $\T(\beta,\beta_p)$ is denoted by ${\cal R}'(\beta,\beta_p)$.}
	\label{figure:rterm}
\end{figure}

 Let us denote by $R'(\beta,\beta_p) $ the gap--length of the horoarc between the two bi-infinite geodesics $\alpha_{p\beta^-}$ and $\alpha_{p\beta^+}$ intersecting $h_p$.  
Then the length of the horoarc on $h_p$ with the two ends at $h_p\cap \beta_p$ is
$$
R(\beta,\beta_p):=R'(\beta,\beta_p)+2Q(\beta,\beta_p).
$$
Note that this is just the value of the potential $W_p(\beta, \beta_p)$ at $p$ for the trouser leg $\T(\beta, \beta_p)$:
\be \la{R=W}
R(\beta,\beta_p)=W_p(\beta, \beta_p).
\ee

\bl \la{L5.3} The potential $W_p(\beta, \beta_p)$  for the trouser leg $\T(\beta, \beta_p)$ is given by 
\be \la{EXVT}
\begin{split}
W_p(\beta, \beta_p) =\K^{1/2}_{\beta_p}\cdot (\L_\beta^{-1/2}+\L_\beta^{1/2}).\\
\end{split}
\ee
Therefore we have 
\be \la{5511}
\begin{split}
R(\beta,\beta_p)=\K^{1/2}_{\beta_p}\cdot (\L_\beta^{-1/2}+\L_\beta^{1/2}).\\
\end{split}
\ee\el

\bpr We use the same notation as in the proof of Lemma \ref{L5.31}. Then 
 the potential is given by $W_p(\beta, \beta_p) = \B_{\beta_p}  + \B_{\beta_p} \X$. So it is calculated in terms of the parameters $\K_{\beta_p}$ and $l_\beta$ 
 just by (\ref{EXVT}).  
The second claim follows from this and (\ref{R=W}).
\epr

We observe that this implies the following formula
:
$$
R'(\beta,\beta_p) =  \K_{\beta_p}^{1/2}(\L_\beta^{1/2}-\L_\beta^{-1/2}).
 $$
 \vskip 2mm

3. 
Suppose that the simple geodesic $\ell$ hits the boundary arc $ab$ of some crown end as in Figure \ref{figure:tri}. If we move $\ell$ towards the left, then it will converge to some simple bi-infinite geodesic $\alpha_{pa}$ without changing its homotopy class. Moving $\ell$ towards the right, it will converge to some simple bi-infinite geodesic $\alpha_{pb}$ without changing its homotopy class. Thus we get an embedded ideal triangle   $\tau(\alpha_{pa}, \alpha_{pb})$, with the geodesic sides $\alpha_{pa}$, $\alpha_{pb}$ and $ab$. Denote the gap--length of the horoarc between the two bi-infinite geodesics intersecting $h_p$ by $S(\alpha_{pa},\alpha_{pb}) $. 
It is equal to the  potential $W_{p}(\alpha_{pa},\alpha_{pb})$. 

Let $v_p, v_a, v_b$ be the decoration vectors at the vertices $p, a, b$ of the ideal triangle, see Figure \ref{sz5}. Then  
\be \la{Sfunct}
 S(\alpha_{pa},\alpha_{pb}) =  W_{p}(\alpha_{pa},\alpha_{pb})  = \frac{\omega(v_a, v_b)}{\omega(v_p, v_a) \omega(v_p, v_b)} = \Bigl(\frac{\K_p}{\K_a\K_b}\Bigr)^{-1/2}.
 \ee

\begin{figure}[ht]
	\centering
	\includegraphics[scale=0.25]{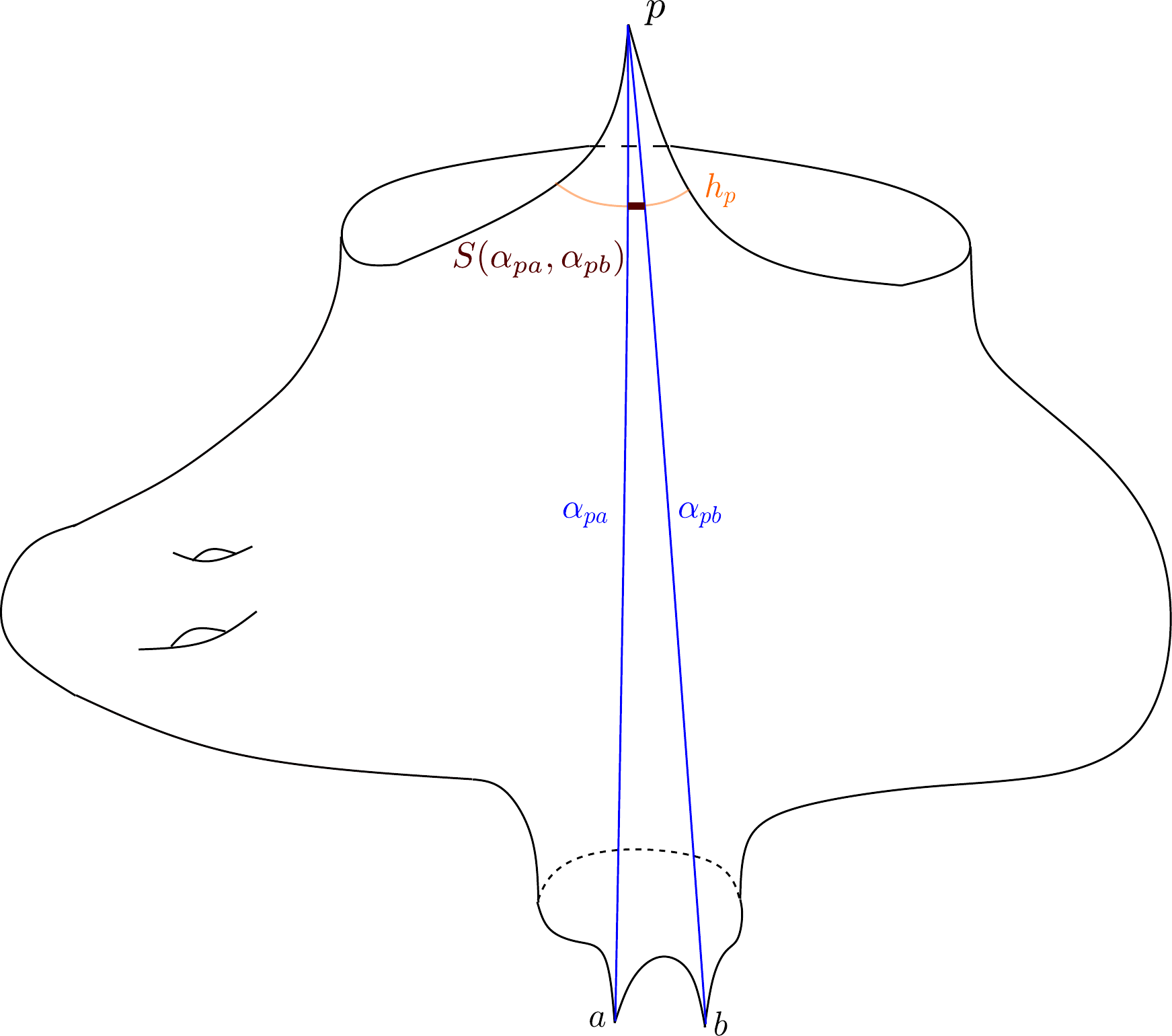}
	\small
	\caption{Embedded ideal triangle $\tau(\alpha_{pa}, \alpha_{pb})$. The gap for $\tau(\alpha_{pa}, \alpha_{pb})$ is denoted by $S(\alpha_{pa}, \alpha_{pb})$.}
	\label{figure:tri}
\end{figure}

  \begin{figure}[ht]
\centerline{\epsfbox{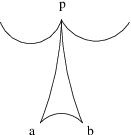}}
\caption{ $S(\alpha_{pa},\alpha_{pb}) = W_{p}(\alpha_{pa},\alpha_{pb})$ is the partial potential at the cusp $p$ for the ideal triangle $apb$.}
\label{sz5}
\end{figure} 

\bd
\label{defn:hp}
An embedded ideal triangle $\tau(\alpha_{pa},\alpha_{pb})$ is {\em $p$-narrowest} if any cusp region at $p$ of any other embedded ideal triangle can not be embedded into the cusp region at $p$ of $\tau(\alpha_{pa},\alpha_{pb})$.   
    \ed

An ideal geodesic triangle $\tau$  id $p$-narrowest if and only if its  opposite  side $ab$  
is  a bi-infinite boundary geodesic. So  it connects two adjacent cusps on the same crown $\rm C$. 
The vertices $a$ and $b$ coincide if and only if the crown $\rm C$ has a single cusp, as shown on the right of Figure \ref{figure:tce}.    
The opposite side $ab$ can be at the same crown as the cusp $p$. The three vertices of the triangle $\tau$ can coincide, as shown  on Figure \ref{figure:crpp}. 
\vskip 2mm

There are the following  sets of elementary decorated surfaces containing the chosen cusp $p$:

 \begin{itemize} 
 
 \item The set ${\cal H}_{{\rm T}, p}$  of  isotopy classes of  trouser legs $\T$ containing  $p$.  
It is   a union of two disjoint subsets:
\be
 {\cal H}_{{\T}, p} ={\cal H}^{\partial}_{{\T}, p}\cup ({\cal H}_{{\T}, p}-{\cal H}^{\partial}_{{\T}, p}).
 \ee
Here ${\cal H}^{\partial}_{{\T}, p}$  is the subset where $\ell_\T\subset \partial \bS$: the geodesic boundary loop $\ell_\T$ lies on the boundary of $\bS$.

 \item The set ${\cal H}_{{\tau}, p}$ of  isotopy classes of  ideal   {\it $p$-narrowest triangles}  triangles $\tau$ containing the cusp $p$. 
It   is a union of two disjoint subsets, depending how many internal sides has the triangle $\tau$:  one or two:
 \be
 {\cal H}_{{\tau}, p} ={\cal H}^{(1)}_{{\tau}, p}\cup {\cal H}^{(2)}_{{\tau}, p}.
 \ee

\end{itemize}
\vskip 2mm

Let us state the McShane identity for a cusp/puncture $p$ on an ideal hyperbolic surface. 

Recall the local potential $W_p$ at  cusp $p$. 

\begin{theorem} \la{MSRF1}
Given an ideal hyperbolic surface, given a cusp $p$ at a crown or a puncture $p$, we get
\be \la{MSRF2}
\sum_{\rm{T}(\beta,\beta_p) \in {\mathcal{H}}_{T, p} -  {\mathcal{H}}_{T, p}^{\partial}} 2 Q(\beta,\beta_p)+\sum_{\T(\beta,\beta_p)\in  {\mathcal{H}}_{T, p}^\partial}R(\beta,\beta_p)+
\sum_{\tau(\alpha_{pa},\alpha_{pb})\in {\cal H}_{\tau, p} }S(\alpha_{pa},\alpha_{pb})=W_p.
\ee
 \end{theorem}

\bpr It follows immediately from the above construction and Birman--Series \cite{BS85} theorem, adapted to the surfaces with crowns in Theorem \ref{thm:bs}. Indeed, the sum of the horoarc lengths of all intervals on the horoarc $h_p$ involved in the sum is the horocycle length $\H_p$ of the horoarc $h_p$. Note that $\H_p=W_p$. 
Observe that for non-boundary geodesics $\beta$, the term ${\cal D}(\beta,\beta_p)$ appears twice: for the trouser leg which orients the geodesic $\beta$  one way, and for another trouser leg where $\beta$ is oriented  the other way.  
\epr
\paragraph{Remarks.}

1. On Figure \ref{figure:tce}(1) for the ideal triangle $\tau(\alpha_{pa},\alpha_{pb})$, the potential $W_{p}(\alpha_{pa},\alpha_{pb})$ is the sum of potentials $W_{p}(\alpha_{pa},\ell)$ and $W_{p}(\ell,\alpha_{pb})$ for the two smaller ideal triangles. Counting   potentials $W_{p}(\alpha_{pa},\alpha_{pb})$, $W_{p}(\alpha_{pa},\ell)$ and $W_{p}(\ell,\alpha_{pb})$ doubles the horoarc length count for potential $W_{p}(\alpha_{pa},\alpha_{pb})$.  

2. On Figure \ref{figure:tce}(2) the ideal triangle $\tau(\ell, \alpha_{pb})$ lies inside  of the cylinder bounded by $\ell_1$ and $C_b$. Counting  potentials $W_{p}(\ell,\alpha_{pb})$ and $W_{p}(\alpha_{pa},\alpha_{pb})$ doubles the  horoarc length count for the $W_{p}(\ell,\alpha_{pb})$. 
\begin{figure}[ht]
	\centering
	\includegraphics[scale=0.25]{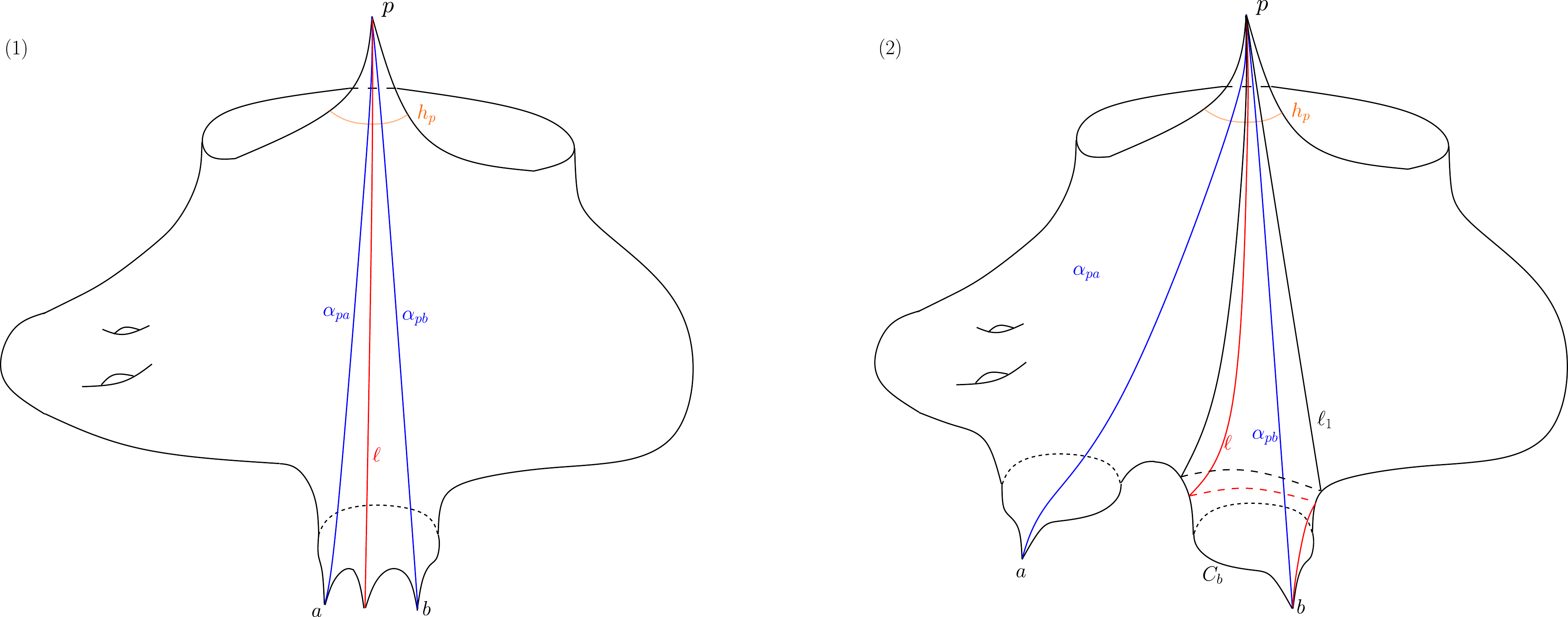}
	\small
	\caption{The opposite side of a $p$-narrowest triangle is a bi-infinite boundary geodesic. For example: 
	1) The two triangles $\tau(\alpha_{pa},\ell)$ and $\tau(\ell,\alpha_{pb})$ are two $p$-narrowest ideal triangles, while $\tau(\alpha_{pa},\alpha_{pb})$ is not. 2) The triangle $\tau(\ell, \alpha_{pb})$ is a $p$-narrowest ideal triangle, while $\tau(\alpha_{pa},\alpha_{pb})$ is not. }
	\label{figure:tce}
\end{figure}

\section{The unfolding formula}\la{Sec4}
 
\subsection{Geometry of  unfolding at the cusp $p$.} \la{cutting}

Let us equip $\bS$ with an ideal hyperbolic structure. Then we refer to  a trouser leg or an ideal triangle  on $\bS$ as a {\it geodesic trouser leg} and {\it ideal geodesic triangle}, respectively.   
  \vskip 1mm

1. Take an ideal geodesic triangle $\tau$, and  cut   $\bS$ along the  $\tau$:
$$
\bS = ({\bS-\tau}) \cup \tau.
$$
 If $\tau$ has two internal sides, then the  $\K-$coordinates $(\K_a, \K_b)$ at the sides of   $\bS-\tau $ and  $\tau$ corresponding to them  provide  two projections:
$$
 {\cal T}_{{\bS-\tau}} \lra (\R_{+}^\ast)^2, \qquad {\cal T}_{{\tau}} \lra (\R_{+}^\ast )^2. 
$$ 
Consider their fibered product
$
{\cal T}_{{\bS-\tau}} \times_{(\R_{+}^\ast)^2} {\cal T}_{{\tau}}. 
$
 It comes with the canonical projection
$$
(\K_a, \K_b):{\cal T}_{{\bS-\tau}} \times_{(\R_{+}^\ast)^2} {\cal T}_{{\tau}} \lra (\R_{+}^\ast)^2. 
$$

If the  triangle $\tau\subset \bS$ has a single internal side, there is a similar fibered product with the projection
$$
\K_a:{\cal T}_{{\bS-\tau}} \times_{\R_{+}^\ast } {\cal T}_{{\tau}} \lra  \R_{+}^\ast. 
$$
Let  ${\rm Stab}_\tau\subset {\rm Mod}(\bS)$ be the subgroup  stabilizing the triangle $\tau$.  One has 
\be \la{gr1}
{\rm Stab}_{\tau} = {\rm Mod}({\bS- \tau}).
\ee

\bp \la{Ptau} Cutting  $\bS$ along the $k$ sides of  an ideal triangle $\tau$ provides an isomorphism 
\be \la{itau}
i_\tau:   {\cal T}_{\bS} \stackrel{\sim}{\lra}    {\cal T}_{\bS-\tau}\times_{(\R^\ast_{+})^k} {\cal T}_{\tau} .
\ee
It is equivariant under the action of the group ${\rm Stab}_\tau$.
One has 
\be \la{118}
\begin{split}
&i_\tau^*\Omega_\bS(\K, \L) =   \Omega_{\bS-\tau}(\K', \L) \wedge d\log \K_a\wedge d\log \K_b; \ \ \mbox{if $k=2$}\\
&i_\tau^*\Omega_\bS(\K, \L) =  \Omega_{\bS-\tau}(\K', \L) \wedge d\log \K_a; \ \ \  \ \ \  \ \ \ \ \ \ \ \ \ \mbox{if $k=1$}.\\
&i_\tau^*W_\bS = W_{\bS-\tau} + W_\tau.\\
\end{split}
\ee
\ep
\bpr  
First two equalities follows from Proposition \ref{PROP2.9} 
The last  follows from Lemma  \ref{APP}.
\epr

  Passing to the quotient by the action of the group (\ref{gr1}) we get  an isomorphism
\be \la{isotau2}
{\cal M}_{\bS-\tau}\times_{(\R^\ast_{+})^k} {\cal M}_{\tau} \stackrel{\sim}{\lra} {\cal M}_{\bS, \tau}:= {\cal T}_{\bS}/{\rm Stab}_\tau.
\ee
There is the {\it unfolding map}: 
$$
{\cal M}_{\bS, \tau} = {\cal T}_{\bS}/{\rm Stab}_\tau\  \lra \ {\cal M}_{\bS} = {\cal T}_{\bS}/{\rm Mod}({\bS}).
$$

2. Similarly,  take a trouser leg $\T\subset \bS$ with  the sides $(\F, \ell_\T)$, and  cut the surface $\bS$ along the    $\T$:
$$
\bS = ({\bS-\T}) \cup \T.
$$
Recall the Fenchel--Nielsen length-twist coordinates  $(l,\theta)$ related to the loop $\ell_\T$. 
The $\K-$coordinate   and the 
length $l$ of the matching boundary components of  ${\bS-\T} $ and  $\T$   provide  projections 
$$
(\K, l):  {\cal T}_{{\bS-\T}} \lra \R_{+}^\ast\times \R_{+}, \qquad (\K, l): {\cal T}_{{\T}} \lra \R_{+}^\ast \times \R_{+}. 
$$ 
Let us consider their fibered product:
$$
{\cal T}_{{\bS-\T}} \times_{(\R_{+}^\ast\times \R_{+})} {\cal T}_{{\T}}. 
$$
The subgroup  ${\rm Stab}_{\T}$   stabilizing  $\T$ contains the Dehn twist 
$\D_\ell$ along the loop $\ell$, acting by $(l, \theta)\lms (l, \theta+l)$, and the subgroup ${\rm Mod}({\bS-\T})$, commuting with the Dehn twist. They generate the  group ${\rm Stab}_{\T}$:
\be
{\rm Stab}_{\T} = \langle \D_\ell\rangle \times {\rm Mod}({\bS-\T}).
\ee

\bl \la{PT} Cutting  $\bS$ along a trouser leg $\T=(a,\ell_T)\subset \bS$ provides a  ${\rm Mod}({\bS- \T})$-equivariant map
\be \la{SPT}
\begin{split}
&i_\T: {\cal T}_{\bS} \stackrel{\R}{\lra} {\cal T}_{\bS-\T}\times_{(\R^\ast_{+}\times\R_{+})} {\cal T}_{\T}, \ \ \ \ \mbox{if $\ell_\T\not \subset \partial \bS$};\\
&i_\T: {\cal T}_{\bS} \stackrel{\sim}{\lra} {\cal T}_{\bS-\T}\times_{\R^\ast_{+}} {\cal T}_{\T} \ \ \ \ \ \  \ \ \ \ \ \ \mbox{if $\ell_\T \subset \partial \bS$}.\\
\end{split}
\ee
 It is a principal  $\R-$fibration  parametrised by the twist parameter $\theta$ in the first case, and an isomorphism in the second. Recall $c_{\bS,\ell_T}$ in Proposition \ref{propcc}
One has 
\be \la{SPT1}
\begin{split}
&i_\T^*  \Omega_\bS(\K, \L)=c_{\bS,\ell_T}\cdot \Omega_{\bS-\T}(\K',\L') \wedge dl \wedge d\theta \wedge d\log \K_a\ \ \ \ \mbox{if $\ell_\T\not \subset \partial \bS$};\\
&i_\T^*  \Omega_\bS(\K, \L)= \Omega_{\bS-\T}(\K',\L) \wedge  d\log \K_a \ \ \ \ \  \ \  \ \ \ \ \ \  \ \ \mbox{if $\ell_\T \subset \partial \bS$}.\\
&i_\T^*W_\bS = W_{\bS-\T} + W_\T.\\
\end{split}
\ee
\el

\bpr The claims (\ref{SPT}) are the standard properties of the Teichmuller spaces. 
Proposition \ref{PROP2.9} and Proposition \ref{propcc} imply the first line in equation (\ref{SPT1}). 
The last two lines in (\ref{SPT1}) follow from Proposition \ref{PROP2.9} and 
Lemma \ref{APP}.
\epr

After the quotient by the action of the group ${\rm Stab}_\T$, the map $i_{\rm T}$ provides  an $S^1-$fibration
\be \la{opp}
{\cal M}_{\bS, \T}:= {\cal T}_{\bS}/{\rm Stab}_\T \stackrel{S^1}{\lra} {\cal M}_{\bS-\T}\times_{(\R^\ast_{+}\times \R_+)} {\cal M}_{\T}.
\ee

Passing  to the quotient by ${\rm Mod}(\bS)$  we get   the {unfolding map} 
${\cal M}_{\bS, \T}\lra {\cal M}_{\bS}$.
\vskip 1mm

Recall the exponential volume form   
$$
{\Bbb E}_\bS= 
\ e^{-W_\bS}\ \Omega_\bS,
$$
and  (\ref{eqelevol}).
Therefore  (\ref{118}) and Proposition \ref{PROP2.9+}   $\&$ (\ref{SPT1}) imply the   {\it factorization property} of exponential volume forms   under the cutting of $\bS$: 
\be \la{FEV}
\begin{split}
&i_\tau^*{\Bbb E}_\bS =
 {\cal E}_{\tau} \cdot {\Bbb E}_{\bS-\tau} \wedge d\log \K_a \wedge d\log \K_b \ \ \mbox{if $k=2$};\\
&i_\tau^*{\Bbb E}_\bS=
 {\cal E}_{\tau} \cdot {\Bbb E}_{\bS-\tau}  \wedge d\log \K_a \  \qquad \qquad \quad \mbox{if $k=1$};\\
&i_\T^*{\Bbb E}_\bS = c_{\bS,\gamma}\cdot {\cal E}_{\T}\cdot {\Bbb E}_{\bS-\T}  \wedge dl \wedge d\theta \wedge d\log \K_a;\\
&i_\T^*{\Bbb E}_\bS = 
 {\cal E}_{\T}\cdot {\Bbb E}_{\bS-\T}   \wedge d\log \K_a.\\
\end{split}
\ee

 Summarizing, we have the following spaces, highlighting the base coordinates:
\be \la{FP31}
\begin{split}
&{\cal M}_{\bS, \T}(\K_a, l, \theta)\stackrel{S^1}{\lra}{\cal M}_{\bS-\T}\times_{(\R^\ast_{+}\times \R_+)} {\cal M}_{\T}, \qquad (\K_a,l, \theta) \in \R^\ast_{+}\times \R_+\times S^1,\ell_\T \not \subset \partial\bS;\\
&{\cal M}_{\bS, \T}(\K_a)\stackrel{\sim}{\lra}{\cal M}_{\bS-\T}\times_{\R^\ast_{+}} {\cal M}_{\T}, \qquad  \  \ \ \qquad \ \ \ \ \K_a \in \R^\ast_{+}, \quad \   \ \ \ \ \ \ \ \ell_\T \subset \partial\bS;\\
&{\cal M}_{\bS,   {\tau}}(\K_a, \K_b)\stackrel{\sim}{\lra} 
{\cal M}_{\bS-\tau}\times_{(\R^\ast_{+})^2} {\cal M}_{\tau}, \qquad\ \  \  \quad \K_a,\K_b \in (\R^\ast_{+})^2, \quad k=2;\\
&{\cal M}_{\bS,  {\tau}}(\K_a)\stackrel{\sim}{\lra} 
{\cal M}_{\bS, \tau}\times_{\R^\ast_{+}} {\cal M}_{\tau}, \qquad\ \ \   \ \quad \qquad \  \ \K_a \in \R^\ast_{+}, \qquad \qquad k=1. \\
\end{split}
\ee
Let us  now elaborate  Theorem \ref{T5.5I}.  Let $d_\T$ be $\frac{1}{2}$ if $\T$ is cutting off a one-holed torus, or $1$ otherwise.

\begin{theorem} \la{T5.5}
For any  decorated surface $\bS$ and a crown with  a single cusp $p$, and any smooth function $f$ on the moduli space ${\cal M}_\bS({\K}, \L)$,  we have the unfolding formula
\begin{equation} \la{recur}
\begin{aligned}
&\int _{{\cal M}_\bS({\K}, \L)}f \ W_p\ {\Bbb E}_\bS=\\
&\sum_{{\T}} 
\int_{  {\cal M}_{\bS,  {\T}}(\K,\L,\K_a)}     f\cdot   (e^{-W_{\rm T}} W_{\T})(\K_a, l, \theta)\cdot 
 \  {\Bbb E}_{\bS-{\rm T}}    \wedge d\log \K
\\
&+\sum_{{\T} }  
 c_{\bS,\gamma} d_\T \cdot \int_{ {\cal M}_{\bS, \T}(\K, \L, l,\K_a) }    f\cdot   (e^{-W_\T}Q_\T)(\K_a, l, \theta)\cdot 
\ {\Bbb E}_{\bS-{\T}}  \wedge   dl \wedge  d\theta\wedge  d\log \K_a\\
&+\sum_{{\tau_1} }    
 \int_{ {\cal M}_{\bS, \tau_1}(\K_a)}  
   f\cdot  (e^{-W_\tau } W_{\tau_1, p})(\K_a, \K_b, \K_p)   \cdot 
  \  {\Bbb E}_{\bS-\tau_1} \wedge d\log \K_a\\
  &+\sum_{{\tau_2} }   
  \int_{ {\cal M}_{\bS, \tau_2}(\K_a, \K_b)}  
f\cdot(e^{-W_\tau } W_{\tau_2, p})(\K_a, \K_b, \K_p)   \cdot 
  \  {\Bbb E}_{\bS-\tau_2} \wedge d\log \K_a\wedge  d\log \K_b.\\
  \end{aligned}
\end{equation}

The first (respectively the second) sum is over topological types of  embedded trouser legs ${\T}$ containing  $p$, where
$\ell_\T\subset \partial \bS$  (respectively $\ell_\T \not \subset \partial \bS$). 

The third (respectively the fourth) sum is over  topological types of the embedded $p$-narrowest ideal triangles $\tau$ containing $p$ with a single internal side (respectively with two internal sides). 
\end{theorem}

\begin{remark}
 \begin{figure}[ht]
	\centering
	\includegraphics[scale=0.5]{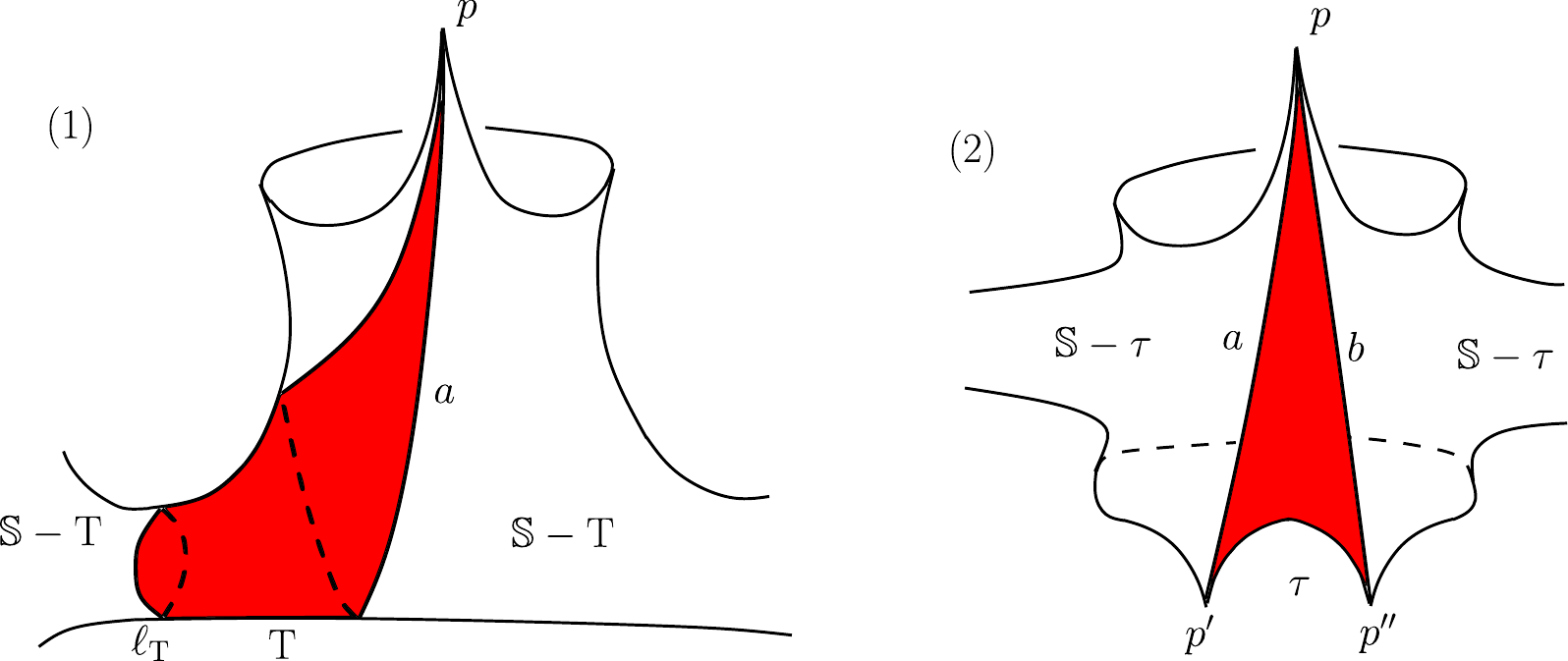}
	\small
	\caption{Figure (1) for cutting the trouser leg. Figure (2) for cutting the $p$-narrowest ideal triangle.}
	\label{figure:cuttriangle}
\end{figure}
Let us denote the decorated surface $\bS$ by $S_{g,m;n_1,...,n_r}$ where the number of boundary components is $m$ and the crowns $\rC_i$ is of type $n_i$.
There are several cases of the surface after cutting off a trouser leg or a $p$-narrowest ideal triangle.

When we cut off a trouser leg $\T$ from the surface $\bS$ as in Figure (1),
\begin{enumerate}
\item if the boundary circle $\ell_\T$ of $\T$ is also contained in the boundary of $\bS$, in this case the ideal edge $a$ is connected to the crown of type $n_1$ through the cusp $p$. Thus $\bS-\T=S_{g,m-1;n_1+1,n_2,...,n_r}$ and it corresponds to the second line of formula (\ref{recur});
\item if $\bS-\T$ is connected and the boundary circle $\ell_\T$ is not contained in the boundary of $\bS$, the left side and the right side of Figure \ref{figure:cuttriangle}(1) is connected somewhere. Cutting off $\T$ will decrease the genus $g$ by $1$, increse the number of boundary components $m$ by $1$, and change the type $n_1$ crown containing $p$ into type $n_1+1$ crown as the last case. Thus $\bS-\T=S_{g-1,m+1;n_1+1,n_2,...,n_r}$ and it corresponds to the part of the third line of formula (\ref{recur});
\item if $\bS-\T$ is not connected, the trouser leg $\T$ cuts $\bS$ into two connected components. Suppose the left side connected component as in Figure \ref{figure:cuttriangle}(1) is $S_{g_1,m_1+1;n_{k+1},...,n_r}$, then $\bS-\T=S_{g_1,m_1+1;n_{k+1},...,n_r} \cup S_{g-g_1,m-m_1;n_{1}+1,n_2,...,n_k}$ and it corresponds to the other part of the third line of formula (\ref{recur}).
\end{enumerate}
As in Figure (2), the ideal points $p'$ and $p''$ could coincide. Moreover $p, p', p''$ could also coincide, then some of $n_k$ has to be $1$, and in this case $\bS-\tau$ could be two connected components. When we cut off a $p$-narrowest ideal triangle $\tau$ from the surface $\bS$,
\begin{enumerate}
\item if one of the ideal edges $a$ and $b$ of $\tau$ is also contained in the boundary of $\bS$ even though the other ideal edge of $\tau$ is already contained in the boundary of $\bS$, then $\tau$ is contained in some crown of type $n_1$($n_1\geq 2$) containing the cusp $p$. Cutting off $\tau$ will only change this type $n_1$ crown into type $n_1-1$ crown. Thus $\bS-\tau=S_{g,m;n_1-1,n_2,...,n_r}$ and it corresponds to the fourth line of formula (\ref{recur});
\item if both of the ideal edges $a$ and $b$ are not contained in the boundary of $\bS$ and $\bS-\T$ is connected, cutting off the ideal triangle $\tau$ will combine the two crowns touching $\tau$ into a single crown. Namely, the type $n_1$ crown and the type $n_2$ crown is combined into the type $n_1+n_2+1$ crown. Thus $\bS-\tau=S_{g,m-1;n_1+n_2+1,n_3,...,n_r}$ and it corresponds to the part of the fifth line of formula (\ref{recur});
\item if $\bS-\tau$ is not connected, then the ideal triangle $\tau$ has to be the ideal triangle with a single vertex $p$ as in Figure \ref{figure:crpp} and $\tau$ cuts $\bS$ into two connected components. In this case, the crown containing $p$ is of type $n_1=1$. Suppose the left side connected component as in Figure \ref{figure:cuttriangle}(2) is $S_{g_1,m_1+1;1, n_{k+1},...,n_r}$, then $\bS-\tau=S_{g_1,m_1+1;1, n_{k+1},...,n_{r}} \cup S_{g-g_1,m-m_1;1,n_{2},...,n_{k}}$ and it corresponds to the other part of the fifth line of formula (\ref{recur}).
\end{enumerate}
\end{remark}

\subsection{Proof of Theorem \ref{T5.5}.}
 
 \bpr Let $\{\gamma\}=(\gamma_1,..., \gamma_s)$ be a collection of disjoint geodesics, where $\gamma_1$, $...$, $\gamma_{s-s'}$ are simple bi-infinite geodesics and the rest are simple closed geodesics, see Figure \ref{figure:sgamma}. 
 Consider a cover  
\be \la{cov}
\pi_{\{\gamma\}}: \mathcal{M}_{\bS}({\K},{\L}; {\{\gamma\}}) \lra \mathcal{M}_{\bS}({\K},{\L})
\ee
given by pairs $(\Sigma,  \{\eta\})$, where  $\Sigma\in \mathcal{M}_{\bS}({\K},{\L})$,  and $\{\eta\}$ is an  element of the ${\rm Mod}(\bS)-$orbit of  $\{\gamma\}$. 
 
 \begin{figure}[ht]
	\centering
	\includegraphics[scale=0.5]{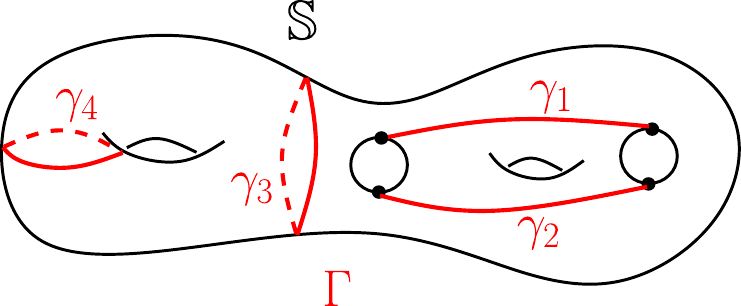}
	\small
	\caption{Bi-infinite geodesics $\gamma_1, \gamma_2$ and simple closed geodesics $\gamma_3, \gamma_4$.}
	\label{figure:sgamma}
\end{figure}

Consider the subgroup    of the pure mapping class group stabilising each  element   of the collection $\{\gamma\}$:
$$
{\rm Stab}_{\{\gamma\}} 
\subset \rm{Mod}(\bS).
$$
Then   
\be
\begin{split}
&\mathcal{M}_{\bS}({\K}, \L)=\mathcal{T}_{\bS}({\K},{\L})/{\rm Mod}(\bS),\\
&\mathcal{M}_{\bS, \{\gamma\}}({\K},{\L})=\mathcal{T}_{\bS}({\K},{\L})/{\rm Stab}_{\{\gamma\}}.\\
\end{split}
\ee
Given an integrable function $f$ on  $\mathcal{M}_{\bS, \{\gamma\}}({\K},{\L})$, let $\pi_{\{\gamma\} \ast}f$ be its push forward by the  map  (\ref{cov}). Then 
\be \la{0011}
\int_{\mathcal{M}_{\bS}({\K},{\L})}\pi_{\{\gamma\} \ast}f \  {\rm{d} Vol} = \int_{\mathcal{M}_{\bS, \{\gamma\}}({\K},{\L})} f \ \rm{d} Vol. 
\ee

The two interesting cases   of the collection  $\{\gamma\}$ are: 
\vskip 1mm

(1)  
 The $k$  internal sides of an ideal triangle $\tau$, which has    a vertex at the cusp $p$. Here $k=1$ or $2$. 

 (2) The boundary of a trouser leg $\rm T$. So it  is  the union of a 
bi-infinite geodesic and a geodesic loop. 
\vskip 1mm

We integrate  the function $fW_p$ over the moduli space $ {\cal T}_\bS({\K}, \L)/{\rm Mod}(\bS)$  over the exponential measure:
 $$
 \int _{{\cal T}_\bS({\K}, \L)/{\rm Mod}(\bS)}f \ W_p\ {\Bbb E}_\bS.
 $$
 Using McShane identity (\ref{MSRF2}), we present $W_p$ as a finite collection of sums corresponding to the topological types of the ideal triangles $\tau$ and trouser legs $\T$ containing the cusp $p$. Each of these sums is over a ${\rm Mod}(\bS)-$orbit of such a $\tau$ or $\T$. Now we have several cases to consider. \vskip 2mm

 1. Take  the sum over the orbit of an ideal triangle $\tau$. The triangle can have two or one internal edges. 
 
 ($1'$) {Let $\{\tau\}=\ell_a\cup\ell_b$ be the union of two ideal edges of $\tau$ other than the boundary edge. Recall 
 \[\left|\mathrm{Stab}_{\{\tau\}}/\left(\mathrm{Stab}_{\ell_a}\cap \mathrm{Stab}_{\ell_b}\right)\right|=1.\]
Then by \cite[Lemma 7.3]{Mir07a}
\[\sum_{\alpha\in \mathrm{Mod}(\bS)\{\tau\}} f= \pi_{\{\tau\} \ast}f. \]}

It can be written as 
 a single integral over the unfolding cover 
 $$
 \pi_{\tau}: {\cal M}_{\bS, \tau}\lra {\cal M}_{\bS}
 $$
 of the product of the function $f$ by the potential function $W_{\tau, p} = S(\alpha_{pa},\alpha_{pb})$: 
 $$
 \int _{{\cal T}_\bS({\K}, \L)/ {\rm Stab}_\tau}f \ W_{\tau,p}\ {\Bbb E}_\bS.
 $$
Using the isomorphism $i_\tau$ in (\ref{itau}), and since $k=2$ in the ($1')$ case,  we can write it as 
 \be \la{opp1}
\begin{split}
&\int _{ {\cal T}_{\bS-\tau}\times_{(\R^\ast_{+})^2} {\cal T}_{\tau} / {\rm Stab}_\tau}f \ W_{\tau,p}\ i_\tau^*({\Bbb E}_\bS) \\
 \stackrel{(\ref{isotau2})}{=}&  \int _{ {\cal M}_{\bS-\tau}\times_{(\R^\ast_{+})^2} {\cal M}_{\tau}}f \ W_{\tau,p}\ i_\tau^*({\Bbb E}_\bS)\\
 \stackrel{(\ref{FEV})}{=}&    \int _{ {\cal M}_{\bS-\tau}\times_{(\R^\ast_{+})^2} {\cal M}_{\tau}}f e^{-W_\tau}\ W_{\tau,p}\ {\Bbb E}_{\bS-\tau} \wedge d\log \K_a \wedge d\log \K_b.\\
  \end{split}
 \ee
 
The bottom line here is exactly the bottom line in formula (\ref{recur}). \vskip 2mm
 
($1''$)  The same argument for the ideal triangles $\tau$ with a single internal side gives the last line in  (\ref{recur}).\vskip 2mm

 2. Take  the sum over the orbit of a trouser leg $\T$. It can be written as 
 a single integral over the unfolding cover 
 $$
 \pi_{\T}: {\cal M}_{\bS, \T}\lra {\cal M}_{\bS}
 $$
  of the product of the function $f$ by either the function ${\cal Q}_{\T}$ or ${W}_\T$. {By \cite[formula 7.5]{Mir07a}, the factor $d_{\T}$ comes from the hyperelliptic evolution of the one-holed torus cut out from the closed geodesic of the touser leg $\T$.} The rest of the argument follows the same lines as above, using (\ref{opp}) instead of (\ref{isotau2}) in (\ref{opp1}). Theorem \ref{T5.5} is proved. \epr

\subsection{The reduced unfolding formula.} \la{Sect5.2} 
Given a crown with a single cusp $p$, there is a canonical trouser leg $\T_p^\sharp$ for which   the boundary geodesic loop  is  the neck geodesic for the crown. 
 The trouser leg $\T_p^\sharp$ is preserved by the mapping class group ${\rm Mod}(\bS)$. 
All other trouser legs  at the cusp $p$ have infinite orbits  under the mapping class group action. So the related terms in the unfolding formula 
are  integrals over  moduli spaces which are  simpler than the original one. 
However this is not the case for the two terms assigned to the  $\T_p^\sharp$, which we will discuss momentarily.  We subtract them from  the unfolding formula, getting {\it the reduced unfolding formula}. Then 
 we can express the exponential volume of the moduli space inductively  via  simpler surfaces.

 \begin{figure}[ht]
	\centering
	\includegraphics[scale=0.23]{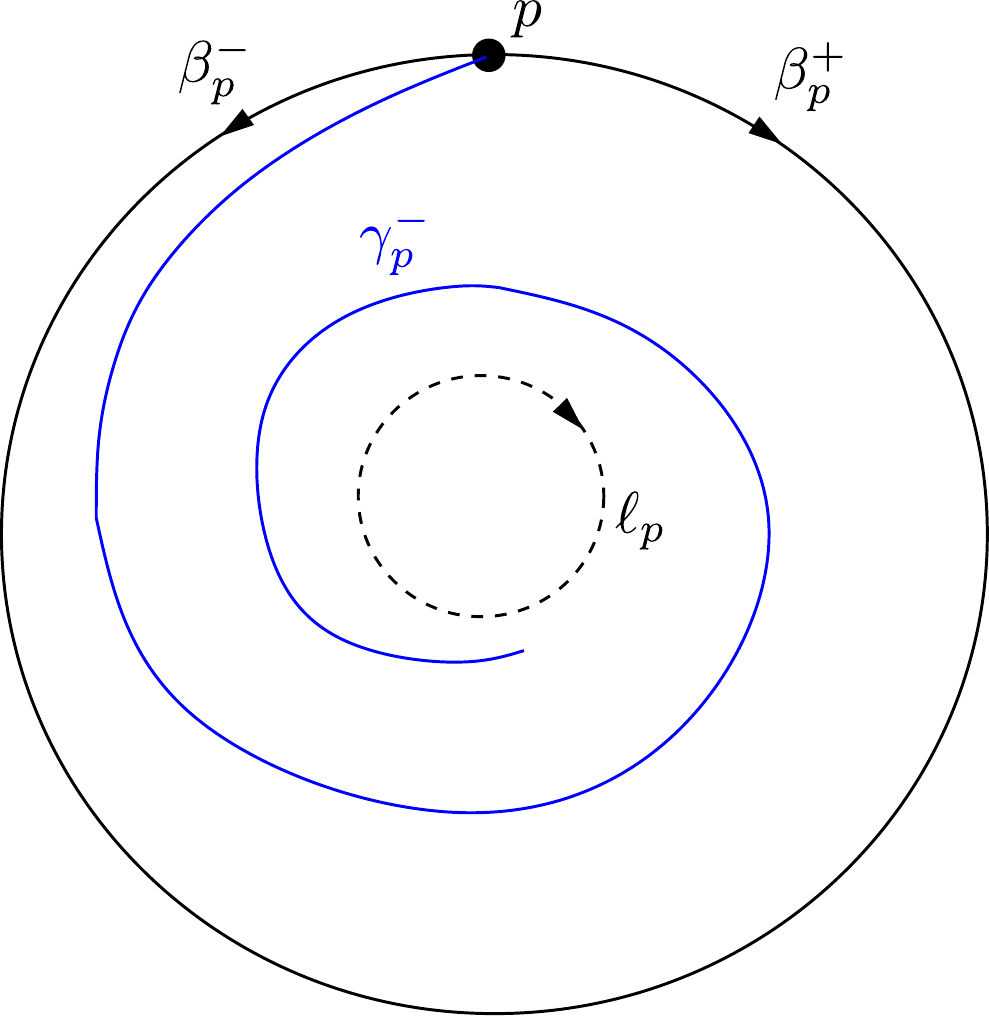}
	\small
	\caption{Boundary geodesics $\beta_p^-, \beta_p^+$ and the geodesic $\gamma_p^-$ winding around the neck geodesic $\ell_p$. The geodesic $\gamma_p^+$ is not shown.}
	\label{figure:tls}
\end{figure}

There are four   ${\rm Mod}(\bS)-$invariant {\it oriented} bi-infinite geodesics starting at  $p$, see Figure \ref{figure:tls}, see also 
 paragraph 3 of Section \ref{SECT6}  for the key example when $\bS$ is an annulus:

\begin{itemize}
\item The geodesics $\gamma^-_{p}$ and $\gamma^+_{p}$ winding in the opposite directions around the neck geodesic $\ell_p$. 

\item 
 The boundary geodesic  with two possible orientations, denoted by   $\beta^-_{p}$ and $\beta^+_{p}$.  
\end{itemize}

The orientation of  $\bS$ near $p$  induces the following order  of the oriented geodesics:
 $
\beta^-_p, \gamma^-_{p}, \gamma^+_{p}, \beta_p^+$. 
\vskip 2mm

We denote by  $
W^-_p, W^\sharp_p, W^+_p
$ 
 are the areas of  triangles cut  by the horocycle $h_p$ and  the  geodesics
$
(\beta_p^-, \gamma^-_{p}), \ (\gamma^-_{p}, \gamma_p^+),  \ (\gamma^+_{p}, \beta_p^+), 
$ respectively. 
Clearly  one has 
$$
W_p = W^-_p + W^\sharp_p + W^+_p, \quad 
W_p^-=W_p^+.
$$
 So 
  the {\it reduced potential} $W_p^\sharp$ is given by 
\be \la{15}
W^\sharp_p = W_p - 2W^+_p.
\ee

   The sum of the two terms   in   the first line of the RHS of (\ref{recur**}) corresponding to the trouser leg $\T^\sharp_p$ with   two possible orientations of the boundary bi-infinite geodesic $\beta_p$,  which we call   {\it the $\T^\sharp_p-$term}, is given by
\be \la{MUF}
\int _{{\cal M}_\bS({\K}, \L)}f \ (W^+_p + W_p^-) \  {\Bbb E}_\bS.
\ee

Subtracting (\ref{MUF}) from  the both parts of   (\ref{recur**}), and using (\ref{15}), we get the {\it reduced unfolding formula}: 
\begin{equation} \la{recur***}
\begin{aligned}
&\int _{{\cal M}_\bS({\K}, \L)}f \ W^\sharp_p\  {\Bbb E}_\bS\ = \ \mbox{the RHS of formula (\ref{recur**})} - \{\mbox{the $\T^\sharp_p-$term}\}.\\
  \end{aligned}
\end{equation}
Then each integral on the right hand side is  the integration over the moduli space for a  surface which is simpler than the original one.

   \section{Tropicalized exponential volumes of moduli spaces} \la{tropical}

    \subsection{Tropicalized moduli spaces and their geometric interpretations} \la{sect6.1}
 
  Recall that a {\it semifield $\Bbb P$} is a set with the operations of addition, multiplication and division, satisfying the usual axioms. In particular $\Bbb P$ is an abelian group 
 for the multiplication and division.  
 
 The  most important example for us is the semifield of positive real numbers $\R_{>0}$.  
 
 Other important examples  are   the {\it tropical semifields} ${\Bbb A}^t$ associated with the following  abelian groups: 
$$
{\Bbb A}= \Z, \quad {\Bbb A}=\Q, \quad {\Bbb A}=\R.
$$
The semifield structure is given by 
$
\alpha \oplus \beta:= {\rm max}(\alpha, \beta), \ \ \alpha\otimes b := \alpha+\beta, \ \ \alpha : \beta := \alpha-\beta.\ \ 
$ \vskip 2mm

 Since the space $\mathcal{P}_\bS$ has ${\rm Mod}(\bS)-$equivariant cluster Poisson structure, and since the cluster Poisson transformations    are subtraction free, we can  consider the set ${\mathcal{P}}_{\bS}(\Bbb P)$ of  points   of the space  ${\cal P}_{ \bS}$ with values in any semifield $\Bbb P$, see  \cite[Section 1.1]{FG03b}. 
 Namely, a tropical $\Bbb P-$point is defined by assigning   to each cluster coordinate system ${\bf c}$ a  collection of  elements 
 $$
 (\chi^{\bf c}_1, ..., \chi^{\bf c}_n)\in {\Bbb P}^n, \ \ \ \ \ n = {\rm dim}{\cal P}_{\bS},
 $$
 related 
by  tropicalized cluster Poisson transformations.  For example,  positive points of ${\cal P}_{ \bS}$ are the  points of ${\cal P}_{  \bS}$ whose coordinates in one, and hence any, cluster Poisson coordinate system are positive real numbers.  
The modular group ${\rm Mod}(\bS)$ acts on the set ${\cal P}_{\bS}(\Bbb P)$. So we can 
consider the {\it enhanced moduli orbifold/orbiset}
$$
{\cal M}^{\rm en}_\bS(\Bbb P):= {\cal P}_{\bS}(\Bbb P)/{\rm Mod}(\bS).
$$

 In particular we have sets of positive, and tropical real, rational and integral  points of the space ${\cal P}_\bS$:
 \be \la{es}
 {\cal P}_\bS(\R_{>0}), \quad  {\cal P}_\bS(\R^t), \quad {\cal P}_\bS(\Q^t), \quad {\cal P}_\bS(\Z^t).
\ee
 Here $ {\cal P}_\bS(\R_{>0})$ is a manifold, ${\cal P}_\bS(\R^t)$ is a piecewise linear manifold, and ${\cal P}_\bS(\Z^t)$ is  its discrete 
 subset:
 $$
 {\cal P}_\bS(\Z^t)\subset {\cal P}_\bS(\R^t).
  $$ 
Each of the sets (\ref{es}) has a geometric interpretation via the Teichm\"uller theory 
\cite[Section 12]{FG03a}, see also   \cite{FG05}, briefly discussed below. 
   \vskip 2mm

 First, the  triple $({\cal T}^{\rm en}_\bS, W, \Omega)$, where  ${\cal T}^{\rm en}_\bS$ is the enhanced Teichmuller space,  
has an algebraic geometric avatar:   
  the stack ${\cal P}_{ \bS}$ with  the   ${\rm Mod}(\bS)-$invariant   volume form $\Omega_\bS$, and 
  a  regular potential function $
  {W}= \sum_{p}{W}_p$. Indeed, by Theorem \ref{P=X},  
$$
{\cal T}^{\rm en}_\bS= {\cal P}_{ \bS}(\R_{>0}), 
 $$   
and the volume form and  local potentials    on the enhanced Teichm\"uller space are the restrictions of the form $\Omega_\bS$ and the functions $W_p$  to the  positive locus.      
\vskip 1mm

Next, the  points of the space ${\cal P}_{  \bS}$ with values in the tropical semifields $\Bbb A^t$, where $\Bbb A=\Z, \Q, \R$, are identified,  respectively, with the sets of  {\it  integral, rational, and real laminations} on $\bS$. 
 This follows from the description of   ${\cal X}-$laminations on a decorated surface, and coordinates  parametrising them, see \cite{FG05}.

\subsection{The exponential tropical volumes}  \la{sect6.2}
The potential $W$,  the functions $\K$, and the exponent of the boundary geodesic  length  $\L:=e^l$ are  Laurent polynomials with positive integral coefficients in any cluster Poisson coordinate system. So they  can be tropicalized, providing 
a ${\rm Mod}(\bS)-$invariant functions
\be
W^t, \ \K^t, \ \L^t: {\cal P}_{  \bS}({\Bbb A}^t)\lra \Bbb A.
\ee
So we can consider 
 the corresponding spaces of negative laminations
  \be \la{msms}
 {\cal M}^-_\bS({\Bbb A}^t)(\kappa, \alpha):= \frac{\{\ell\in  {\cal P}_\bS({\Bbb A}^t)\ | \ \K^t(\ell) = \kappa, \ \ \L^t(\ell) = \alpha, \ \ \ W^t(\ell) \leq 0\}}{\mbox{modulo the action of the group ${\rm Mod}(\bS)$}}. 
  \ee
  The cluster volume form $\Omega_\bS$ induces a volume form $\Omega_\bS^t$ on the space of real tropical points, which provides in each of the tropical cluster Poisson coordinate systems $(\chi_1 \ldots \chi_k)$  the Lebesgue measure $2\cdot d\chi_1 ... d\chi_k$.

\bd Given a decorated surface $\bS$, the {\it exponential  tropical volume}  is the volume of the real tropical moduli space ${\cal M}^-_\bS(\R^t)(\kappa, \alpha)$ for the tropical volume form $\Omega^t$:
 \be \la{def4.1}
 \begin{split}
 {\rm Vol}^t_{\cal E}({\cal M}^{\rm en}_\bS(\kappa, \alpha)):= 
 &\int_{{\cal M}^-_{\bS}(\R^t)(\kappa, \alpha)}\Omega_\bS^t.\\
\end{split}
 \ee
 \ed

\begin{remark}
The tropicalised exponential volume can be obtained by starting with the exponential volume integral, substituting $$
\X_i = e^{\chi_i t}, \ \ \B_i = e^{\beta_it}, \ \ \K_i=e^{\kappa_it}, 
$$
dividing the integrand by $t^n$, and then take the limit of the integral when $t \lra \infty$. Indeed,  the  cluster volume form is: 
$$
2\cdot d\log \X_1 \wedge \ldots \wedge d\log \X_n = 2\cdot t^n d\chi_1\wedge \ldots \wedge d\chi_n.
$$
This is why we  divide the integral by $t^n$ before taking the limit. 
\end{remark}

 \bt \la{MTRTH}
 For any decorated  surface $\bS$, the  exponential tropical volume  (\ref{def4.1})  is finite: 
   $$
   {\rm Vol}^t_{\cal E}({\cal M}^{\rm en}_\bS(\kappa, \alpha))<\infty. 
    $$
 \et
When $\bS=S$ is a surface with punctures but without the special boundary points this follows from Theorem \ref{TH2.26} below. The general case reduces to this,  see Section \ref{SECT3.3.1}.

\subsection{Tropical volumes of moduli spaces for punctured surfaces = Kontsevich volumes} \la{sect2.4}

  In Section \ref{sect2.4} let  $\bS:=S$  be a connected genus $g$ surface with $n$ punctures but without boundary points. 
  \vskip 1mm

In his proof of Witten's Conjecture, M. Konstevich \cite{K} introduced a combinatorial model ${\cal M}^{\rm comb}_{g,n}$ of the moduli space ${\cal M}_{g,n}$. 
It is an orbifold,  parametrising equivalence classes of metrised connected ribbon graphs $\Gamma$, that is ribbon graphs with  vertices of valency $\geq 3$ and positive real numbers  at the edges, so that the associated oriented surface has genus $g$ and $n$ holes.  By  \cite[Theorem 2.2]{K}, there is a homeomorphism 
$$
\zeta: {\cal M}_{g,n}\times \R_+^n \stackrel{\sim}{\lra} {\cal M}^{\rm comb}_{g,n}.
$$ 
It assignes to a surface $S$ and a collection of positive numbers $\alpha=\{\alpha_1, ..., \alpha_n\}$ the Jenkins-Strebel differential with the critical graph $\Gamma$, whose  perimeters of  boundary components are given by the set $\L$. See also \cite[Section 3]{ABCGLW20}.
\vskip 1mm

 The space ${\cal M}^{\rm comb}_{g,n}$ has a  Poisson structure 
$\{*,*\}_\K$  given by the  bivector
 \be \la{PSK}
 \beta_\K= \sum_{e} \frac{\partial}{\partial \alpha_e}\wedge  \frac{\partial}{\partial \alpha_{s(e)}}.
  \ee
  Here the sum is over all oriented edges $e$ of the graph $\Gamma$, $\alpha_e>0$ is the number assigned to the edge, and $s$ is the cyclic clockwise  for the ribbon 
  structure  of $\Gamma$ shift by one, acting on the set of oriented edges sharing the same initial vertex. 
 
 The subspace of metrised graphs with a given set of perimeters $\alpha$ is denoted by ${\cal M}^{\rm comb}_{g,n}(\alpha)$. They are the symplectic leaves for the Poisson structure.  Their volumes  are defined by
\be \la{118a}
{\rm Vol}({\cal M}^{\rm comb}_{g,n}(\alpha)):= \int_{{\cal M}^{\rm comb}_{g,n}(\alpha)}e^{8\beta_\K^{-1}}.
\ee
Here 
$\beta_\K^{-1}$ is the symplectic form induced by the Poisson bivector $\beta_\K$  restricted to the fibers. The  coefficient $8$ comes  from  \cite[Section 8]{K91}.
\vskip 1mm

  Note that since the surface $S$ has no boundary points ${\cal P}_S={\cal X}_S$. 
  By \cite[Lemma 12.3]{FG03a}, for each puncture of $S$  the group $\Z/2\Z$ acts by positive birational transformations of the moduli space $\mathcal{X}_S$.  
  The set $\mathcal{X}_S(\R^t)$ of  real tropical points  of the  space $\mathcal{X}_S$ parametrises measured ${\cal X}-$laminations on $S$   \cite[Section 12]{FG03a}. So 
   the group $\Z/2\Z$ acts  on the  space   
  $\mathcal{X}_S(\R^t)$ by the  tropicalization. So the group ${\rm Mod}(S) \times (\Z/2\Z)^n$ acts 
on the space of $\mathcal{X}_S(\R^t)$. 

Let us expand the punctures to holes. Consider the set ${\cal C}_S$ of pairs (an ideal triangulation of $S$, a choice of boundary orientations  of all holes on $S$).
The group ${\rm Mod}(S) \times (\Z/2\Z)^n$ acts  on the set ${\cal C}_S$.  The set ${\cal C}_S$ parametrises coordinate systems on $\mathcal{X}_S(\R^t)$, called {\it ideal coordinate systems}. So for each ideal triangulation ${\cal T}$ of $S$ there are $2^n$ ideal coordinate systems. 

There are two flavors of  moduli spaces related to $S$: the classical one, and its enhanced variant:\footnote{Here we abuse the notation by denoting   by ${\cal M}_S$ the space ${\cal M}_S(\R_{>0})$, and 
by ${\cal M}^{\rm en}_S$ the space ${\cal M}^{\rm en}_S(\R_{>0})$.}
\be
\begin{split}
&{\cal M}_S = {\cal M}_S(\R_{>0}) \ := {\cal T}^{\rm en}_S/({\rm Mod}(S)\times (\Z/2\Z)^n), \\
&{\cal M}^{\rm en}_S = {\cal M}^{\rm en}_S(\R_{>0}):= {\cal T}^{\rm en}_S/{\rm Mod}(S). \\
\end{split}
 \ee
The moduli space ${\cal M}^{\rm en}_S$ is a $2^n:1$ ramified cover of the classical one ${\cal M}_S$. Specifying the lengths $\L=(l_1, ..., l_n)$ of boundary geodesics, we get the fibers 
${\cal M}_S(\L)$ and ${\cal M}^{\rm en}_S(\L)$. When $\L=0$, we have 
\be
{\cal M}_S(0) = {\cal M}^{\rm en}_S(0) = {\cal M}_{g,n}.
\ee

If   $\alpha_1...\alpha_n\not = 0$,  the tropical moduli space ${\cal M}^{\rm en}_S(\R^t)(\alpha)$ has $2^n$  components. They are  permuted by the $(\Z/2\Z)^n-$action,   altering the signs of $(\alpha_1, ..., \alpha_n)$.  The classical component is the one with $\alpha_1, ..., \alpha_n\geq 0$. 

\vskip 2mm
 Let  $d=3g-3+n$, so that  $2d$  is the dimension of the  symplectic fiber. In particular, $d  = {\rm dim}{\cal M}_{g,n}$. 
  
 The volume form $\Omega^t_S(\alpha)$ induced on the symplectic fibers ${\cal M}^{\rm en}_S(\R^t)(\alpha)$ by the tropical cluster volume form $\Omega^t_S = 2\cdot d\chi_{\rE_1} \wedge ... \wedge d\chi_{\rE_n}$ is  a  multiple  of the tropical volume form on the symplectic fibers:  
 \be
 \frac{(8\beta_\K^{-1})^d}{d!}= {\sigma}_S 
\cdot \Omega^t_S(\alpha).
 \ee
By \cite[Appendix C]{K} and \cite[Section 9]{K91}, the constant $\sigma_S$ is given by\footnote{Note that $\sigma_S = \frac{1}{2}\rho_S$ where $\rho_S$ is the constant used by Kontsevich. The factor $\frac{1}{2}$ is forced by the factor $2= 2^{\pi_0(S)}$ in the definition (\ref{CVF***}) of the volume form $\Omega_S$, and hence of $\Omega^t_S$. Note that since $S$ is connected,  $\pi_0(S) =1$.} 
\be
\sigma_S:=\frac{\prod dp_i \cdot \frac{(8 \beta_\K^{-1})^{3g-3+n}}{(3g-3+n)!}}{2\cdot \prod_{\rE} |d \chi_{\rE}|}\ 
= 2^{5g-6+2n} 
= 4^d\cdot 2^{-g}.
\ee
 The parameters $p_i$ are  perimeters of the metrised ribbon graphs.  
 \vskip 1mm
 
 Adapting definition (\ref{def4.1}) for the moduli space ${\cal M}_S(\R^t)(\alpha)$, we have
 \be \la{def4.1a}
 \begin{split}
 {\rm Vol}^t({\cal M}_S(\R^t)(\alpha)):= 
 &\int_{{\cal M}_{S}(\R^t)(\alpha)}\Omega_S^t(\alpha).\\
\end{split}
 \ee
 Similar formula holds for the tropical space ${\cal M}^{\rm en}_S(\R^t)(\alpha)$.

  \bt \la{TH2.24} Let $S$ be a connected genus $g$ oriented topological surface with $n$ punctures. Then the volume of the orbispace ${\cal M}^{\rm comb}_{g,n}(\alpha)$ in (\ref{118a}) relates 
  to the   tropical volume of the space  ${\cal M}_S(\R^t)(\alpha)$ by:
  \be \la{85}
  \begin{split}
    {\rm Vol}({\cal M}^{\rm comb}_{g,n}(\alpha)) =  
  &
\sigma_S\cdot {\rm Vol}^t({\cal M}_S(\R^t)(\alpha))\\
  =& 
\sigma_S\cdot{\rm Vol}^t({\cal M}^{\rm en}_S(\R^t)(\alpha)).\\
  \end{split}
  \ee \et

\bpr The tropicalization of the  cluster Poisson coordinates $\{X_\rE\}$ on the space ${\cal X}_S$  delivers  the cluster Poisson coordinates $\{\chi_\rE\}$ on the real tropical space $\mathcal{X}_S(\R^t)$ \cite{FG03a}. It follows from
  Theorem \ref{CPS} that  the induced cluster Poisson bracket on the space $\mathcal{X}_S(\R^t)$ is given by
 \be \la{PSFG}
 \{\chi_\rE, \chi_\F\}_{\rm cl} = \varepsilon_{\rE\F}. 
 \ee

An ${\cal X}-$lamination $\ell$ on $S$ is called {\it positive}, if there exists an ideal coordinate system on $S$ such that the coordinates $\chi_\rE(\ell)$ of $\ell$, 
assigned to the edges $\rE$ of the underlying ideal triangulation,   are positive: $\chi_\rE(\ell)>0$. 
By \cite[Theorem 14.1]{FG03a} there is a dense open subset of {positive ${\cal X}-$laminations}. For any positive  ${\cal X}-$lamination such an ideal coordinate system is unique. The complement to the set of positive laminations  has measure zero. 

There is a bijection between  ideal triangulations ${\cal T}$ of $S$ and trivalent ribbon graphs $\Gamma$ of genus $g$ with $n$ holes. It assigns to an ideal triangulation ${\cal T}$ its dual  graph $\Gamma_{\cal T}$. The orbits   of ideal triangulations of $S$  under the action of the  group ${\rm Mod}(S)$ correspond to the isomorphism classes of trivalent  ribbon graphs. 

Trivalent metrised ribbon graphs form an open dense subset   of full measure of  ${\cal M}^{\rm comb}_{g,n}$. 
The  correspondence between ideal triangulations of $S$ and ribbon graphs of type $(g,n)$ extends to an isomorphism 
\be \la{IDI}
\begin{split}
&\mbox{The open dense part  ${\cal M}^{\times}_S(\R^t)\subset  \mathcal{M}_S(\R^t)$, parametrising positive ${\cal X}-$laminations} \stackrel{\sim}{\lra} \\
&\mbox{The open dense part ${\cal M}^{\times, {\rm comb}}_{g,n}\subset {\cal M}^{\rm comb}_{g,n}$, 
parametrising (metrised trivalent ribbon graphs)/${\rm Iso}$}.\\
\end{split}
\ee
For an ideal triangulation ${\cal T}$ of $S$, it assigns to a positive lamination $\ell$ with  the coordinates   $\chi_\rE(\ell)>0$  the 
 trivalent ribbon graph $\Gamma_{\cal T}$ metrised by the numbers  $\chi_\rE(\ell)$ at the  edges $\rE$.

 \bl \la{Kle} Isomorphism (\ref{IDI})  provides a Poisson isomorphism 
\be \la{PKISO}
 ({\cal M}^{\times}_S(\R^t), \{*,*\}_{\rm cl}) \lra    ({\cal M}^{\times, {\rm comb}}_{g,n}, \{*,*\}_\K).
\ee
 \el
 \bpr
Follows immediately comparing formulas (\ref{PSK}) and (\ref{PSFG}) for the Poisson brackets.   \epr
 
The isomorphism (\ref{PKISO}) identifies the   perimeters $p_i$ of metrised ribbon graphs with the  values of  tropical Casimir functions. 
This  and Lemma \ref{Kle} imply  formula (\ref{85}). \epr

By Kontsevich's theorem \cite[Section 8]{K91}, \cite{K} we have, where $d  = {\rm dim}{\cal M}_{g,n}=3g-3+n$:
\be
\begin{split}
 {\rm Vol}({\cal M}^{\rm comb}_{g,n}(\alpha)) = & \sum_{d_1+ ...+d_n =d}  \frac{\alpha_1^{2d_1}}{d_1!} \ldots \frac{\alpha_n^{2d_n}}{d_n!} \int_{\overline {\cal M}_{g,n}}\psi_1^{d_1}\cdots \psi_n^{d_n}\\
=& \frac{1}{d!} \ \int_{\overline {\cal M}_{g,n}}(\alpha_1^2\psi_1 + \ldots + \alpha_n^2\psi_n)^{d} .\\
\end{split}
  \ee
We recall  that Kontsevich \cite[Section 7]{K91} defines the volume using the  form ${\rm exp}(8\beta_\K^{-1})$, see (\ref{118a}).  

 This and Theorem \ref{TH2.24} immediately imply
 \bt \la{TH2.26} Tropical volumes of  {moduli spaces}  ${\cal M}_S(\R^t)(\alpha)$ of $\mathcal{X}-$laminations on a genus $g$ surface with $n$ punctures $S$ carry  the same information as  the intersection theory on $\overline {\cal M}_{g,n}$. Precisely, 
 \be \la{TRV}
 \begin{split}
  \int_{\overline {\cal M}_{g,n}}e^{\alpha_1^2\psi_1 + \ldots + \alpha_n^2\psi_n}  = \frac{1}{d!}\ \int_{\overline {\cal M}_{g,n}}(\alpha_1^2\psi_1 + \ldots + \alpha_n^2\psi_n)^{d} = &
\sigma_S\cdot{\rm Vol}^t({\cal M}_S(\R^t)(\alpha))\\
 =& 
 \sigma_S \cdot {\rm Vol}^t({\cal M}^{\rm en}_S(\R^t)(\alpha)).\\
\end{split} \ee  
\et
Let us now compare the  top  degree $d$ part  ${\rm Vol}_{\rm top}(\mathcal{M}_S(\L))$ of the volume polynomial   
${\rm Vol}_{\rm WP}(\mathcal{M}_S(\L))$ for the Weil-Peterssen volume form, and the tropical volume   ${\rm Vol}^t({\cal M}_S(\R^t)(\alpha))$. Mirzakhani 
calculated the volume polynomial ${\rm Vol}_{\rm WP}(\mathcal{M}_S(\L))$.  
Therefore by \cite{Mir07a} and     \cite{Mir07b}, see  (\ref{WCM})  
we get the formula for the top degree part of the Weil-Peterssen volume polynomial:
\be \la{V1z}
 \begin{split}
 {\rm Vol}_{\rm top}(\mathcal{M}_S(\L)) = &\sum_{d_1+\ldots + d_n  = d} {\cal V}_{g, d_1, ..., d_n}\cdot l_1^{2d_1} \ldots  l_n^{2d_n};\\
&   {\cal V}_{g, d_1, ..., d_n}  \stackrel{(\ref{WCM})}{=} \frac{1}{2^dd!}\cdot \int_{\overline {\cal M}_{g,n}}\psi_1^{d_1}\cdots \psi_n^{d_n}.\\\end{split}
 \ee
It is handy to introduce the following notation:
 \be \la{10.24.24}
{\cal V}^*_{g, d_1, ..., d_n} \stackrel{}{=}  \frac{2^dd!}{d_1!...d_n!}{\cal V}_{g, d_1, ..., d_n}.
 \ee
 
 \begin{remark} Mirzakhani's   $\L=(l_1, ..., l_n)$  are the lengths of boundary geodesics on hyperbolic surfaces homeomorphic to $S$. Our  $\alpha=(\alpha_1, ..., \alpha_n)$ have different nature: 
 they are  parameters on   measured laminations on $S$.  
 We  relate them  using Kontsevich's perimeter parameters. Namely, we relate the  real tropical space  of measured laminations with Kontsevich's combinatorial moduli space, identifying  $\alpha$'s with Kontsevich's perimeter parameters,  also denoted by 
 $\alpha = (\alpha_1, ..., \alpha_n)$. Then we match the two sets:  
 $$
 \L=(l_1, ..., l_n) \longleftrightarrow  \alpha = (\alpha_1, ..., \alpha_n).
 $$
 Note that in Section \ref{tropical} we used systematically the small greek letters for the tropical parameters, reserving the latin letters for the geometric one. Here we match the two, in a somewhat mysterious way. 
 \end{remark}
 
 Keeping this remark in mind, and combining (\ref{10.24.24}) with (\ref{V1z})  and   formula (\ref{TRV})       for the tropical volume polynomial, we get:  
\be \la{V1za}
 \begin{split}
 {\rm Vol}^t(\mathcal{M}_S(\R^t)(\alpha)) 
  =  & 
\sigma_S^{-1} \sum_{d_1+\ldots + d_n  = d}  \int_{\overline {\cal M}_{g,n}}\frac{(\alpha_1^2\psi_1)^{d_1}\cdots (\alpha^2_n\psi_n)^{d_n} }{d_1! \ldots d_n!}\\
\stackrel{( \ref{V1z})}{=}&
\sigma_S^{-1}  
\sum_{d_1+\ldots + d_n  = d}  {\cal V}^*_{g, d_1, ..., d_n}\cdot \alpha_1^{2d_1} \ldots  \alpha_n^{2d_n}.\\
\end{split}
 \ee

\subsection{Calculating tropicalised   exponential volumes} \la{SECT3.3.1}

  If $\bS=\D_n^*$ is a punctured disc with $n$ marked  points, the Teichmuller space coincides with the moduli space.  The tropical   moduli space ${\cal M}_{{\D_n^*}}(\R^t)( {\kappa}, \alpha)$    carries  real  functions $\chi_i, \beta_i, \kappa_i$, $i \in \Z/n\Z$.  
 We consider the  variables  $ \kappa= (\kappa_1, \ldots \kappa_n)$ as fixed parameters, so the fiberwise tropical cluster volume form is 
$$
\Omega^t_\bS(\kappa) :=2d\log \beta_1 \wedge \ldots \wedge d\log \beta_{n}.
$$

\bl \la{LEMDN} a) The tropicalization of the exponential volume of the moduli space for $\bS=\D_n^*$ is given by
$$
 {\rm{Vol}}^t_{\mathcal{E}}(\D_{n}^*)({\kappa}) = 2\prod_{i=1}^n \kappa_i.
 $$

b) For any integer $d \geq 0$, the following integral is a polynomial in $\kappa= (\kappa_1, ..., \kappa_{n-1})$.
\be
 \int_0^\infty  {\rm{Vol}}^t_{\mathcal{E}}(\D_{n}^*)({\kappa},\alpha)\  \alpha^{d} d \alpha. \ee
\el
  
\bpr
 a)
The tropicalized potential $W^t$ provides inequalities:
$$
W^t= {\rm max}\Bigl(\beta_1, \kappa_1 - \beta_1, \ldots \beta_n, \kappa_n-\beta_n\Bigr)\leq 0 \ \ \ \ \longleftrightarrow \ \  \ \ \kappa_i \leq \beta_i \leq 0, \ \ \forall i \in \Z/n\Z.
$$

b)  The tropicalization of  integral (\ref{120}) delivers system of inequalities
\be
\begin{split}
&\kappa_{i} \leq \beta_{i} \leq 0, \ \  \ \ \ i=1, ..., n-1,\\
 &\kappa_n + \sum_{i=1}^{n-1}\kappa_{i}  \leq \alpha+2\sum_{i=1}^{n-1} \beta_i  \leq  - \kappa_n + \sum_{i=1}^{n-1}\kappa_i. \\
\end{split}
\ee 
So given $\kappa$, this determines a finite polyhedron in the space $\R^n$ with the coordinates $(\beta_1, ..., \beta_{n-1}, \alpha)$.  
\epr

\bp \la{NRFT} For any decorated surface $\bS$ we have the tropical neck recursion formula:\footnote{See  (\ref{10.24.24}) for the modified coefficients $ {\cal V}^*_{g, d_1, ..., d_m}$ for the tropical volume polynomial.}
\be \la{MFOa}
{\rm Vol}^t_{\cal E}({\cal M}_{\bS}({\kappa}, \alpha)) = \rC_\bS  \sum_{d_1+ ... + d_m= 3g-3+m} 
{\cal V}^*_{g, d_1, ..., d_m} \prod_{j=1}^r \int^\infty_{0}  {\rm{Vol}}^t_{\mathcal{E}}(\D_{n_j}^*)({\kappa}_j, \alpha_j)\  \alpha_j^{2d_j+1} d \alpha_j. 
 \ee
 \ep

\bpr Let $\bS_\ell:= \bS-{\rm C}_\ell$ be the decorated surface obtained by cutting out from $\bS$ the crown ${\rm C}_\ell$ with the neck loop $\ell$. 
Denote by $r$ the number of cusps  on the crown ${\rm C}_\ell$. There is the real tropical analog of the Fenchel--Nielsen coordinates, see \cite[Section 6]{FLP} for the classical formulation, adopted to our setting as follows: there is an isomorphism:
\be
{\cal P}_\bS(\R^t) \stackrel{\sim}{=} {\cal P}_{\bS_\ell}(\R^t)\times \R^2\times {\cal P}_{{\rm D}^*_r}(\R^t).
\ee
Here the first coordinate in $\R^2$ is given by the intersection number of the measured lamination $\mu$ with the loop $\ell$, that is the total measure of the loop for the transverse measure $\mu$.
 We use it, together with the fact that ${\rm Mod}(\bS)= {\rm Mod}(\bS_\ell)\times \mathbb{Z}$, similarly to the classical case, except that we use the tropical volume form.
\epr
 \paragraph{Proof of Theorem \ref{MTRTH}.} It follows from Lemma \ref{LEMDN},  Proposition \ref{NRFT}, and Theorem \ref{TH2.26}. 
 \vskip 2mm

\section{Examples} \la{SECT6} 

 \subsection{The annulus $\A_{1,1}$ with a cusp on each boundary component}
 
Recall   the coordinates ${\rm K}_{a}, {\rm K}_{b},  {\rm K}_{p}$ at the  sides opposite to the vertices $(a,b,p)$ of the geodesic triangle 
$\tau$ on Figure \ref{figure:tri}.   The potential of the triangle $\tau$ is given by
\be \la{EXVt}
\begin{split}
{W}_{\tau}(\K_a, \K_b, \K_p)  =& \Bigl(\frac{{\rm K}_{p}}{{\rm K}_{a}{\rm K}_{b}}\Bigr)^{-\frac{1}{2}} + \Bigl(\frac{{\rm K}_{a}}{{\rm K}_{p}{\rm K}_{b}}\Bigr)^{-\frac{1}{2} }+\Bigl(\frac{{\rm K}_{b}}{{\rm K}_{a}{\rm K}_{p}}\Bigr)^{-\frac{1}{2}}.\\
\end{split}
\ee
\vskip 2mm

\subsubsection{The annulus $\A_{p,q}$ and  ${\rm Mod}(\A_{p,q})$.} Denote by  $\A_{p,q}$  the annulus with $p$ marked points on one of the components, and $q$ on the other. There is a unique neck geodesic loop $\ell$. Its length is denoted by $l$. The  pure mapping class group  ${\rm Mod}(\A_{p,q})$ is generated by the Dehn twist $\D_\ell$ around $\ell$. 
It is trivial if one of the integers $m,n$ is zero, and is isomorphic to $\Z$ otherwise:
\be \begin{split}
&{\rm Mod}(\A_{p,q})=\Z, \qquad p,q\not =0. \\
& {\rm Mod}(\A_{0,q})={\rm Mod}(\A_{p,0})= 0. \\
\end{split}
\ee
If the mapping class group is trivial, the Teichm\"uller space is the same as the moduli space. 

\begin{figure}[ht]
\centerline{\epsfbox{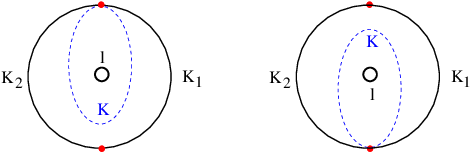}}
\caption{Calculating the exponential volume of the Teichm\"uller space for the  annulus $\A_{0,2}$, by   cutting  along the  loop starting at the top cusp.  
Starting at the other cusp gives the  same integrand. }
\label{sz9}
\end{figure} 

 Calculation of the exponential volume for the annulus $\A_{p,q}$ reduces to the case when $p,q \leq 1$. Namely, if, for example,  $p\geq 2$, we cut out an ideal $(p+1)-$gon  which has 
$p$ boundary sides, and one side given by the bi-infinite geodesic $\gamma$ at a cusp,  surrounding the other boundary component. 
This way we get the decorated surface 
$\A_{1,q}$, which we can reduce to $\A_{1,1}$ by cutting out an ideal $(q+1)-$gon. Let us elaborate first this step in the simplest case, and then concentrate in the most interesting case of the annulus $\A_{1,1}$.

\subsubsection{The annulus $\A_{0,2}$.} The cluster Poisson coordinates on the Teichm\"uller space of $\A_{0,2}$ are 
$\{\K_1, \K_2, \B, \X\}$. We have $\X = \L=e^{l}$. 
Using  (\ref{EXVt}) and (\ref{EXVT}), we get   the following integral for its volume,  see Figure \ref{sz9}:
\be \la{5555}
\begin{split}
& {\rm Vol}_{\cal E}({\cal M}_{\A_{0,2}})(\K_1, \K_2; \L) =\\
&  \int_{0}^\infty {\rm exp}\Bigl( -W_\tau(\K_1, \K_2, \K) - \K^{\frac{1}{2}}(\L^{\frac{1}{2}}+ \L^{-\frac{1}{2}})\Bigr) d\log \K.\\
\end{split}
\ee
Starting  at the other cusp we get  the same integrand. The two  are related by a sequence of two flips. \vskip 2mm

\subsubsection{The McShane identity for the annulus $\A_{1,1}$.} 

There is a collection of bi-infinite geodesics 
$\{\gamma_n\}$, $n \in \Z$, connecting the two cusps $p$ and $q$. They represent all isotopy classes of arcs connecting the two cusps, and 
form a principal homogeneous set for the action of  ${\rm Mod}(\A_{1,1})$, so that 
$\D_l(\gamma_n) = \gamma_{n+1}$. In the limit when $n \to \pm \infty$, we get  bi-infinite geodesics $\gamma_p^-$ and $\gamma_p^{+}$ starting at $p$ and winding around the geodesic  $\ell$. There is a unique bi-infinite boundary geodesic $\beta_p$  from the cusp $p$ to itself. \vskip 2mm

 Moving the end of the geodesic $\gamma_n$   along the boundary geodesic $\beta_q$ till  we get the geodesic $\gamma_{n+1}$, we 
 fill the ideal geodesic triangle $\tau_{p, n}$. Denote by $h_p$ the horoarc at the cusp $p$, and by  $h_{{\tau}_n}$ its intersection with the   triangle $\tau_{p,n}$. 
 Denote by $d_{p, \pm}$  the arcs on the horocycle $h_p$ between $\beta_p$ and $\gamma_{\pm\infty}$. 
 Then we  have
\be
\mbox{length}(h_p) = \mbox{length}(d_{p, -}) + \sum_{n \in \Z} \mbox{length}(h_{{\tau}_n}) +  \mbox{length}(d_{p, +}).
 \ee
 This is the  McShane identity for the cusp $p$. It can be rewritten via potentials as follows. 
Let $W_p$ be the potential at  $p$,  and by 
$W_{p,n}$ the potential at $p$ of the geodesic triangle $\tau_{p, n} = (\gamma_n, \gamma_{n+1}, \beta_q)$. Then we get:
\be
\begin{split}
&W_p = 2 \K_p^{\frac{1}{2}}\L^{-\frac{1}{2}} +  \sum_{n \in \Z} W_{p,n}.\\
&W_p^\sharp =  \sum_{n \in \Z} W_{p,n}.\\
\end{split}
\ee

\subsubsection{Annulus $\A_{1,1}$.}   There are two ways to calculate the exponential volume: using the neck recursion formula, and using 
unfolding formula (\ref{recur}). Let us  elaborate each of them. 

1. Denote by $\K_1, \K_2$ the frozen coordinates at the boundary circles, see Figure \ref{sz8}, and by $\X, \Y$ the cluster Poisson coordinates for the geodesics $\gamma_x, \gamma_y$ 
(which were denoted by $\gamma_0, \gamma_1$ above) of the triangulation given by these two geodesics  and boundary arcs, shown on the right  on Figure \ref{sz8}. 
\begin{figure}[ht]
\centerline{\epsfbox{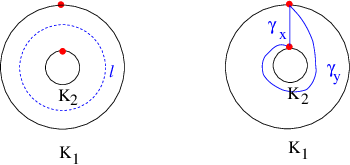}}
\caption{Cutting  the annulus $\A_{1,1}$ by the neck geodesic $\ell$. Triangulating $\A_{1,1}$  by the geodesics $\gamma_x, \gamma_y$.}
\label{sz8}
\end{figure} 

 One has  
$$
\{\X, \Y\} = -2 \X\Y.  
$$
The other Poisson brackets are 
$$
\{\B_1, \X\} = -\B_1\X, \ \ \ \{\B_1, \Y\} = \B_2\Y,  \ \ \ \{\B_2, \X\} = -\B_2\X,\ \ \ \{\B_2, \Y\} = \B_2\Y, \ \ \ \{\B_1, \B_2\}=0.
$$
The following elements  $\K_1, \K_2$ are in the center of the Poisson bracket $\{*,*\}$:
\be \la{K1K2}
\K_i = \B_i^2\X\Y, \ \ \ i =1,2.
\ee
 The Dehn twist acts by the cluster Poisson transformation given by the flip at $\gamma_y$ followed by the symmetry $(\X,\Y) \lms (\Y,\X)$. It preserves the frozen variables $\K_1, \K_2$. 
 \vskip 2mm
 
Then the  exponential volume is given by\footnote{The coefficient $\frac{1}{2}$ comes from Lemma \ref{lA11}.} 
\be \la{101a}
\begin{split}
 {\rm Vol}_{\cal E}({\cal M}_{\A_{1,1}})(\K_1, \K_2)  =  \frac{1}{2} \  &\int_{{\cal M}_{\A_{1,1}}}e^{-(\B_1+\B_2)(1+\X+\X\Y)}  d\log \X\wedge d\log \Y. \\
=  \frac{1}{2} \ & \int_{{\cal M}_{\A_{1,1}}}{\rm exp}\Bigl({-(\K_1^{\frac{1}{2}}+\K_2^{\frac{1}{2}})(\X\Y)^{-\frac{1}{2}}(1+\X+\X\Y)} \Bigr) d\log \X\wedge d\log \Y.\\
\end{split}
\ee
To calculate this integral we have to pick a fundamental domain for the Dehn twist. 
For the annulus $\A_{1,1}$ this is easy using the Fenchel--Nielsen coordinates, and leads to the neck recursion formula
\be\la{100}
 {\rm Vol}_{\cal E}({\cal M}_{\A_{1,1}})(\K_1, \K_2)  =  \frac{1}{2} \int^\infty_{1}
  {\rm exp}\Bigl(-(\K_1^{\frac{1}{2}} + \K_2^{\frac{1}{2}}  )(\L^{\frac{1}{2}}+ \L^{-\frac{1}{2}}) \Bigr)   \log \L \ d \log \L.
\ee
If the limits of the integration were $(0,  \infty)$, this will be the derivative at $s=0$ of the Bessel function. However the limits are 
$(1, \infty)$, and this integral can not be reduced to the Bessel integral. \vskip 2mm

\bl
 The variable $\L^{\frac{1}{2}}$ is related to the variables $(\X,\Y)$ by the equation
\be \la{tt1}
(\X\Y)^{-\frac{1}{2}}(1+\X+\X\Y) = \L^{\frac{1}{2}} +\L^{-\frac{1}{2}}.
\ee
\el

\bpr 
The potential $W_1$ at the cusp  at the  crown supporting the coordinate $\K_1$ can be calculated in two ways: 
by cutting $\A_{1,1}$ along the neck geodesic $\ell$ of the length $l$, as shown on the left of Figure \ref{sz8}, or by formula (\ref{fw}) applied to the triangulated annulus $\A_{1,1}$  on the right of Figure \ref{sz8}. So we get 
\be \la{103}
\begin{split}
& W_1=\K_1^{\frac{1}{2}}(\L^{\frac{1}{2}}+\L^{-\frac{1}{2}}).\\
&W_1= \B_1(1+\X+\X\Y).  \\
\end{split}
\ee
Using (\ref{K1K2}), and  comparing the  two equations (\ref{103}) we get the formula.\epr

This is, of course, equivalent to relating the integrands in (\ref{101a}) and (\ref{100}). \\

2. For the  fourth term of unfolding formula (\ref{recur}) we use the $\K-$coordinates for the triangle sides, denoted   by $\A,\B, \K_1, \K_2$.  
We  denote by $\X,\Y$ the cluster Poisson coordinates, so that 
the coordinates $\A, \X$ are assigned to the edge $\gamma_x$, and $\B, \Y$ to $\gamma_y$ on Figure \ref{sz8}.
Then we have:
\be \la{tt}
\X= \frac{(\K_1\K_2)^{\frac{1}{2}}}{\B}, \quad \Y= \frac{\A}{(\K_1\K_2)^{\frac{1}{2}}}.
\ee
Then using the unfolding formula (\ref{recur})  for $f=1/W_p^\sharp$, and  formula (\ref{EXVt}) for the potential $W_\tau$,  we get\footnote{
 The coefficient $1$ comes from Lemma \ref{lA11}.}
\be \la{101}
\begin{split}
&  {\rm Vol}_{\cal E}({\cal M}_{\A_{1,1}})(\K_1, \K_2)  \\
&=  \int_0^\infty\int_0^\infty \left(\frac{\A\B}{\K_1\K_2}\right)^{\frac{1}{2}}\ \frac{
 {\rm exp}\Bigl(- W_\tau(\A,\B,\K_1) - W_\tau(\A,\B,\K_2)\Bigr)}{ (\L^{\frac{1}{2}} - \L^{-{\frac{1}{2}} })} d \log \A\wedge d \log \B.\\
\end{split}
\ee

3. Let us  check  unfolding formula (\ref{recur}) for the function $f=1$ at the cusp $p$.  We start with 
$$
\int_{{\cal M}_{\A_{1,1}}}W^\sharp_pe^{-W}\Omega.$$
Cutting along the geodesic loop $\alpha$ and using   the neck recursion formula, and  using (\ref{EXVT})  twice,  we get 
\be \la{555}
\begin{split}
& \int_{{\cal M}_{\A_{1,1}}}W^\sharp_pe^{-W}\Omega = \int^\infty_{0}  W^\sharp_pe^{-W_p-W_q} ldl \\
&=  \int^\infty_{0}
  \K_1^{\frac{1}{2}}(e^{l/2}- e^{-l/2} ) {\rm exp}\Bigl(-(\K_1^{\frac{1}{2}} + \K_2^{\frac{1}{2}}  )(e^{l/2}+ e^{-l/2}) \Bigr)   l d l.\\
  \end{split}
 \ee
This integral is calculated  via the Bessel function as follows. 
Observe that 
$$
-\frac{d}{dl} \ {\rm exp}(-\K(e^{l/2}+ e^{-l/2})) = \frac{1}{2}\K(e^{l/2}- e^{-l/2} ) {\rm exp}(-\K(e^{l/2}+ e^{-l/2})). 
$$
So integrating by parts, and observing that   $l {\rm exp}(-\K (e^{l/2}+ e^{-l/2}))$ vanishes at $l=\infty, 0$,  we get
\be \la{UFF1}
\begin{split}
&   \frac{ 2\K_1^{\frac{1}{2}}}{\K_1^{\frac{1}{2}} + \K_2^{\frac{1}{2}} }
 \int^\infty_{0} {\rm exp}\Bigl(-(\K_1^{\frac{1}{2}} + \K_2^{\frac{1}{2}}  )(e^{\frac{l}{2}}+ e^{-\frac{l}{2}}) \Bigr)    d l \\
&=   \frac{2\K_1^{\frac{1}{2}}}{\K_1^{\frac{1}{2}} + \K_2^{\frac{1}{2}} }
 \int^\infty_{-\infty} {\rm exp}\Bigl(-(\K_1^{\frac{1}{2}} + \K_2^{\frac{1}{2}}  )(e^{l}+ e^{-l}) \Bigr)    d l. \\
\end{split}
\ee
Note  that the integrand in the first line, denoted  ${\rm I}(l)$, is an even function of $l$. So 
$\int_{0}^\infty {\rm I}(l) dl = \frac{1}{2}\int_{-\infty}^\infty {\rm I}(l) dl$. Then we change variables $l/2\to l$. 
\vskip 2mm

On the other hand,   cutting out the triangle with the sides  supporting the variables $\A,\B, \K_2$, and using formula (\ref{recur}) for the function $f=1$, we get
$$
  \K^{-{\frac{1}{2}}}_2\int (\A\B)^{{\frac{1}{2}}}
{\rm exp}\Bigl(- W_\tau(\A,\B,\K_1) - W_\tau(\A,\B,\K_2)\Bigr)\frac{d\A}{\A} \wedge \frac{d\B}{\B}.
$$
Changing the  variables $\P:= (\A/\B)^{1/2}$, ${\rm Q}:=(\A\B)^{1/2}$ we get the same result as in (\ref{UFF1}):
$$
 2 \K^{-{\frac{1}{2}}}_2\int_0^\infty {\rm exp}\Bigl(-(\K_1^{{\frac{1}{2}}} +\K_2^{{\frac{1}{2}}} )(\P +\P^{-1} )  \Bigr)d\log \P \cdot \int_0^\infty {\rm exp}\Bigl( - (\K^{-{\frac{1}{2}}} _1+\K^{-{\frac{1}{2}}} _2){\rm Q}\Bigr)\  \  d  {\rm Q} = (\ref{UFF1}).
$$
Indeed, set $\P=e^l$. This  confirms  unfolding formula (\ref{recur}).\\
}

{\it Problem.}  Check directly that  (\ref{100}) = (\ref{101}):
\be \la{99}
\begin{split}
&2 \int_0^\infty\int_0^\infty \left(\frac{\A\B}{\K_1\K_2}\right)^{\frac{1}{2}}\ \frac{
 {\rm exp}\Bigl(- W_\tau(\A,\B,\K_1) - W_\tau(\A,\B,\K_2)\Bigr)}{ (\L^{\frac{1}{2}} - \L^{-{\frac{1}{2}} })} d \log \A\wedge d \log \B\\
 &=\int^\infty_{1}
  {\rm exp}\left(-(\K_1^{\frac{1}{2}} + \K_2^{\frac{1}{2}}  )(\L^{\frac{1}{2}}+ \L^{-\frac{1}{2}}) \right)   \log \L \ d \log \L.\\
  \end{split}
\ee

\begin{figure}[ht]
	\centering
	\includegraphics[scale=0.26]{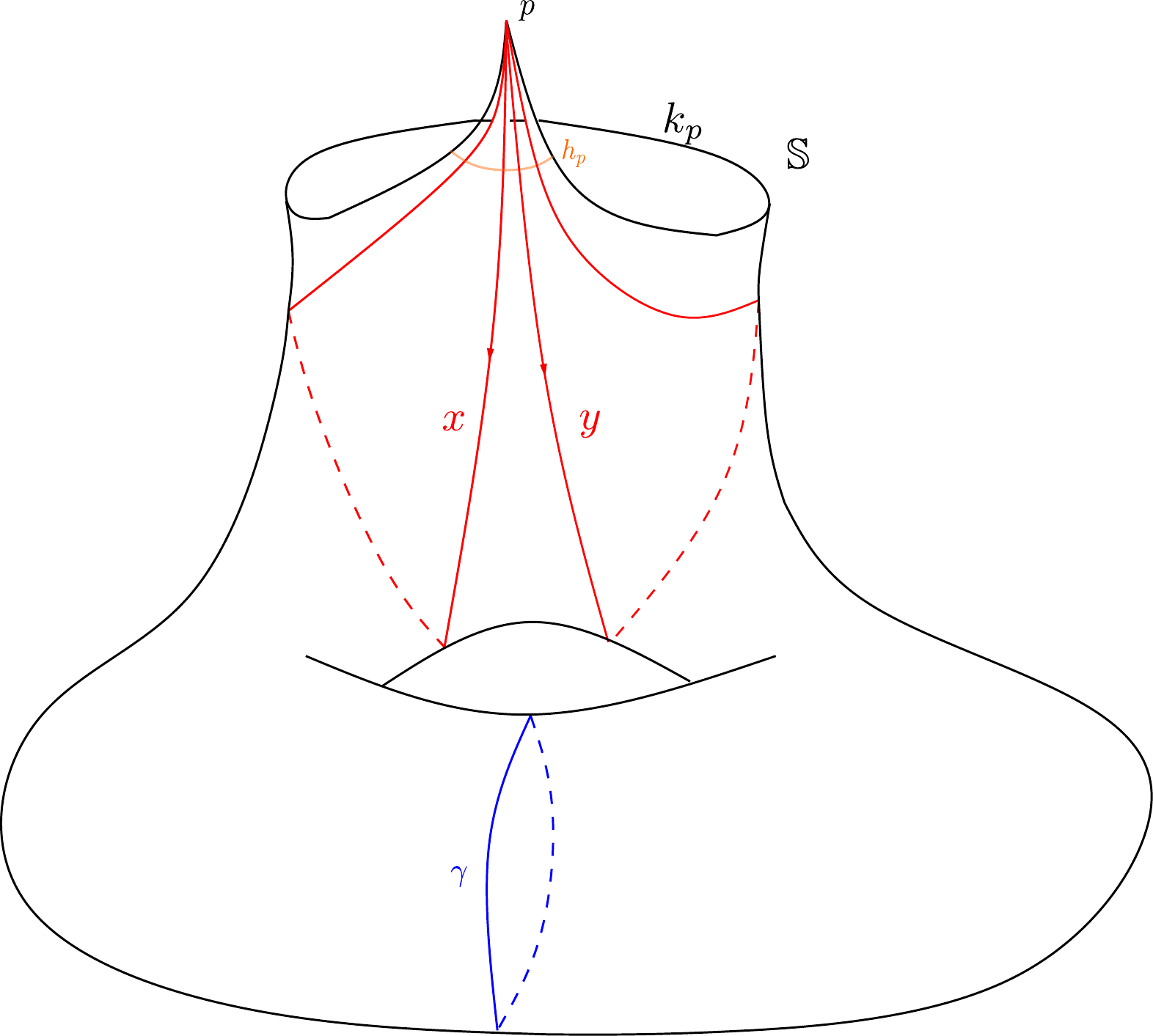}
	\small
	\caption{ $\bS$ is a once crowned torus  with  a single cusp $p$, and  $k_p$ is the geodesic of the crown end.}
\label{figure:1c1p1g}
\end{figure}

\subsection{The once crowned torus with a single cusp} Let $\bS$ be the once crowned torus with a  cusp $p$. Pick a non-peripheral simple geodesic $\gamma$. 
The subgroup of the  group ${\rm Mod}(\bS)$  stabilizing the loop $\gamma$ is isomorphic to $\Z \oplus \Z$. It is generated by the Dehn twists $\D_\gamma$ and 
$\D_{\rm C}$ along the loop $\gamma$, and   the neck loop for the crown ${\rm C}$. 
Denote by $k_p$ the  crown geodesic. 
Choose an embedded ideal triangle $xyk_p$, see Figure \ref{figure:1c1p1g}. Figure \ref{figure:crpp} helps to visualize it. 
Any trouser leg contains the  cusp $p$. There are two trouser legs ${\rm T}({\gamma, x})$ and ${\rm T}({\gamma, y})$, see   
Figure \ref{figure:1c1p1g}. The Dehn twist $\D_\gamma$ preserves them. The subgroup $\langle \D_{\rm C}\rangle $ generated by   the Dehn twist $\D_{\rm C}$ acts  freely on the set of 
 trouser legs  with two orbits $\langle \D_{\rm C}\rangle \cdot {\rm T}({\gamma, x})$ and  $\langle \D_{\rm C}\rangle \cdot {\rm T}({\gamma, y})$. 
 The moduli space $\mathcal{M}_{\bS}(\K_p)$ with the fixed coordinate $\K_p$ at the  geodesic $k_p$ has  dimension $4$.  It  is parameterized by the $\rm{K}$-coordinates $\K_x$, $\K_y$, the length $l_\gamma = \log\L_\gamma$ of the geodesic $\gamma$, and the twist parameter $\theta_\gamma$. \vskip 2mm
 
 \begin{figure}[ht]
\centerline{\epsfbox{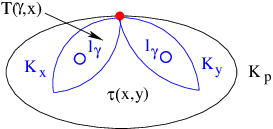}}
\caption{Unfolding  for the torus with a single cusped crown using trouser leg $\T(\gamma,x)$ and triangle $\tau(x,y)$.}
\label{sz10}
\end{figure}

 Applying Theorem \ref{T5.5} as illlustrated on Figure \ref{sz10},  we get
\begin{equation} \la{58}
\begin{aligned}
&\int_{{\cal M}_\bS}W^\sharp_pe^{-W_\bS}\Omega_\bS =\\
&4\int_{\mathbb{R}^2_+} Q(\gamma,x)
\cdot {\rm{Vol}}_{\mathcal{E}}{\cal M}_{\T(\gamma,x)}(\K_x,l_\gamma)
\cdot {\rm{Vol}}_{\mathcal{E}}{\cal M}_{\bS -  
\T(\gamma,x)}(\K_p, \K_x,l_\gamma)\cdot l_{\gamma}  \cdot d l_{\gamma} \wedge d \log \K_x
\\+& \int_{\mathbb{R}^2_+}  S(x,y)\cdot {\rm{Vol}}_{\mathcal{E}}{\cal M}_{\tau(x,y)}(\K_x, \K_y, \K_p) \cdot 
{\rm{Vol}}_{\mathcal{E}}{\cal M}_{\bS-\tau(x,y)}(\K_x, \K_y)  \cdot d\log \K_x \wedge d\log \K_y. 
\end{aligned}
\end{equation}
The coefficient $4$  reflects that the first integral in (\ref{58}) 
is equal to the one obtained by changing $x\to y$.  
According to (\ref{gapfunct}) and (\ref{Sfunct}) we have: 
$$
Q(\gamma,x)=\K_x^{\frac{1}{2}} \L_\gamma^{-\frac{1}{2}},\;\;\;\;\;\; S(x,y)= \left(\frac{\K_p}{\K_x \K_y}\right)^{-\frac{1}{2}}.
$$
\begin{figure}[ht]
	\centering
	\includegraphics[scale=0.3]{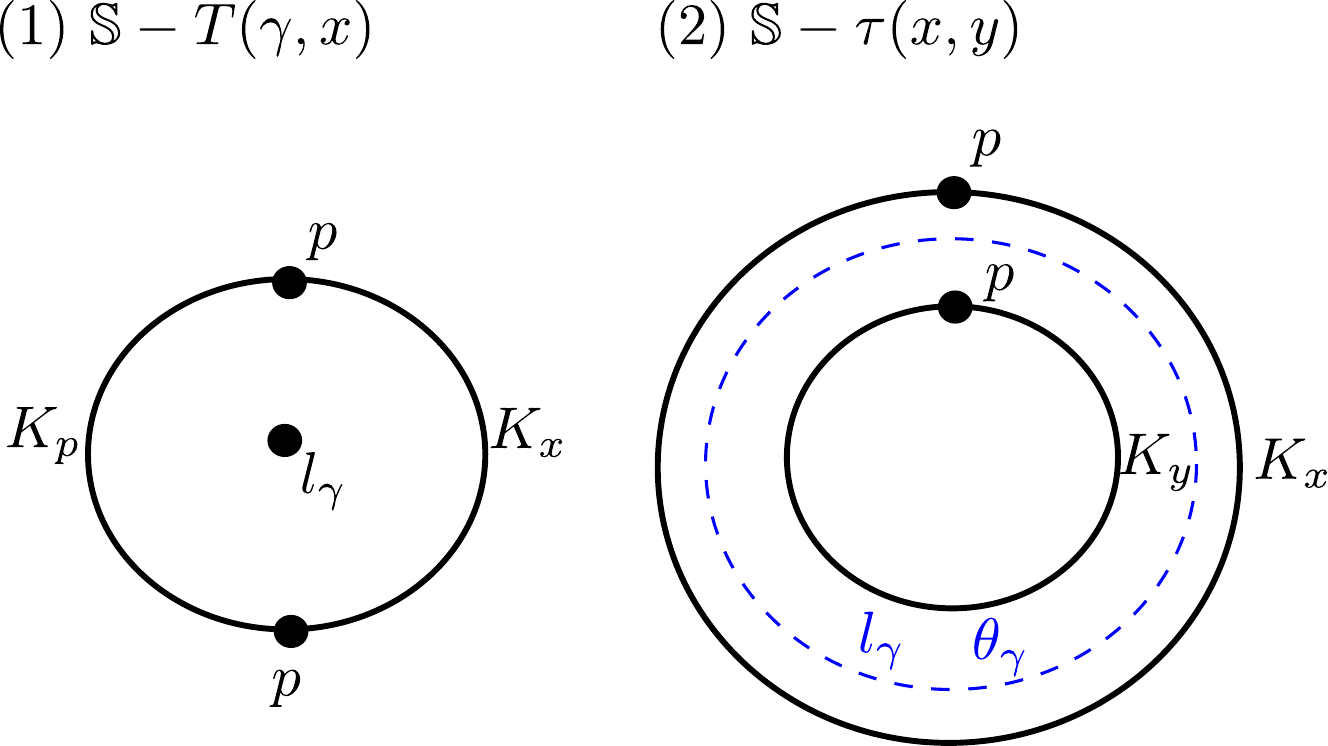}
	\small
	\caption{Decorated surfaces $\bS -  \T(\gamma,x)$ and $\bS-\tau(x,y)$, and coordinates on the related moduli spaces.}
\label{figure:1c1p1gcut}
\end{figure}

I) Let us elaborate the second line in (\ref{58}). Recall that 
$$
{\rm{Vol}}_{\mathcal{E}}{\cal M}_{  \T(\gamma,x)}(\K_x,l_\gamma)\ \stackrel{(\ref{EXVT})}{=}\ {\rm exp}\left(-\K^{\frac{1}{2}}_x(\L_\gamma^{\frac{1}{2}}+ \L_\gamma^{-\frac{1}{2}})\right).
$$
The surface $\bS -  \T(\gamma,x)$ is an annulus $\A_{0,2}$, see  Figure \ref{figure:1c1p1gcut}(1). We parameterize the space ${\cal M}_{\bS-  \T(\gamma,x)}$ as in Figure \ref{figure:1c1p1gcut}(1). 
Using   formula (\ref{5555}) for ${\rm{Vol}}_{\mathcal{E}}{\cal M}_{\bS -  \T(\gamma,x)}(\K_p, \K_x,l_\gamma)$, we write the second line in (\ref{58}) as

\begin{equation} \la{58a}
\begin{aligned}
&4 \cdot \int_{\mathbb{R}^3_{>0}} \K_x^{\frac{1}{2}} \L_\gamma^{-\frac{1}{2}}\cdot 
{\rm exp}\Bigl( -(\K^{\frac{1}{2}}_x +\K^{\frac{1}{2}}_y) (\L_\gamma^{\frac{1}{2}} + \L_\gamma^{-\frac{1}{2}} )  -W_\tau(\K_p, \K_x, \K_y) \Bigr)  
\cdot   l_{\gamma}  d l_{\gamma} \wedge \frac{d\K_x}{\K_x}\wedge \frac{d\K_y}{\K_y}.
\end{aligned}
\end{equation}

\begin{figure}[ht]
\centerline{\epsfbox{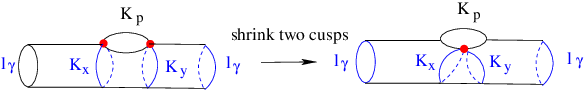}}
\caption{Let $\bS$ be  the torus with a single cusped crown.  Then $\bS -  \tau(x,y)$ is  an annulus $\A_{1,1}$. One can get it by shrinking one arc connected two cusps on the left.} 
\label{sz11}
\end{figure} 

II)  Let us elaborate the last line in (\ref{58}). The  exponential volume of the moduli space for the triangle $\tau(x,y)$:
$$
{\rm{Vol}}_{\mathcal{E}}{\cal M}_{\tau(x,y)}(\K_x, \K_y, \K_p)\ \stackrel{(\ref{EXVt})}{=}  \ \exp\left(-W_\tau(\K_x, \K_y, \K_p)\right).
$$
The surface $\bS -  \tau(x,y)$ is  an annulus $\A_{1,1}$,  see     Figure \ref{figure:1c1p1gcut}(2).  
So using    (\ref{101}),   the last line in (\ref{58})  
becomes  an integral over $\A,\B, \K_x, \K_y>0$:  
\be \la{101q}
\begin{split}
& \int  \left(\frac{\A\B}{\K_p}  \right)^{\frac{1}{2}}  \ \frac{
 {\rm exp}\Bigl(- W_\tau(\A,\B,\K_x) - W_\tau(\A,\B,\K_y)-W_\tau(\K_x, \K_y, \K_p)\Bigr)}{ (\L_\gamma^{\frac{1}{2}} - \L_\gamma^{-{\frac{1}{2}} })} \frac{d \A}{\A} \wedge  \frac{d \B}{\B} \wedge 
  \frac{d \K_x}{\K_x}\wedge  \frac{d \K_y}{\K_y}.  \\
\end{split}
\ee

\subsection{The McShane  identity for a pair of pants with cusps} \la{McPP}
Let $\bS$ be a pair of pants with one marked point on each boundary component as in Figure \ref{figure:cpp1}(1). Let $\gamma_p$, $\gamma_q$, $\gamma_r$ be the loops surrending the crown ends $k_p$, $k_q$, $k_r$ respectively, whose Dehn twists generate the mapping class group $\mathrm{Mod}(\bS)\cong\mathbb{Z} \times \mathbb{Z} \times \mathbb{Z}$.  Our  goal is to describe all terms of the McShane identity for the cusp $p$, and hence all terms of the unfolding formula. We consider all geodesics emitting from $p$.
\vskip 2mm

The bi-infinite geodesics $k_q'$, $k_r'$, $k_p$ form a $p$-narrowest ideal triangle as in Figure \ref{figure:cpp1}(2) with the potential $\theta_b:=W_p(k_q', k_r')$ at the cusp $p$. Any geodesic emitting from $p$ within the arc between $k_q'$ and $k_r'$ hits the boundary $k_p$. The annulus $\A^q_{1,1}$ is bounded by a bi-infinite geodesic $k_q'$ and the crown end $k_q$ with the potential $W_{k_q'}$ at $p$, while the annulus $\A^r_{1,1}$ is bounded by a bi-infinite geodesic $k_r'$ and the crown end $k_r$ with the potential $W_{k_r'}$ at $p$. 
\begin{figure}[ht]
	\centering
	\includegraphics[scale=0.6]{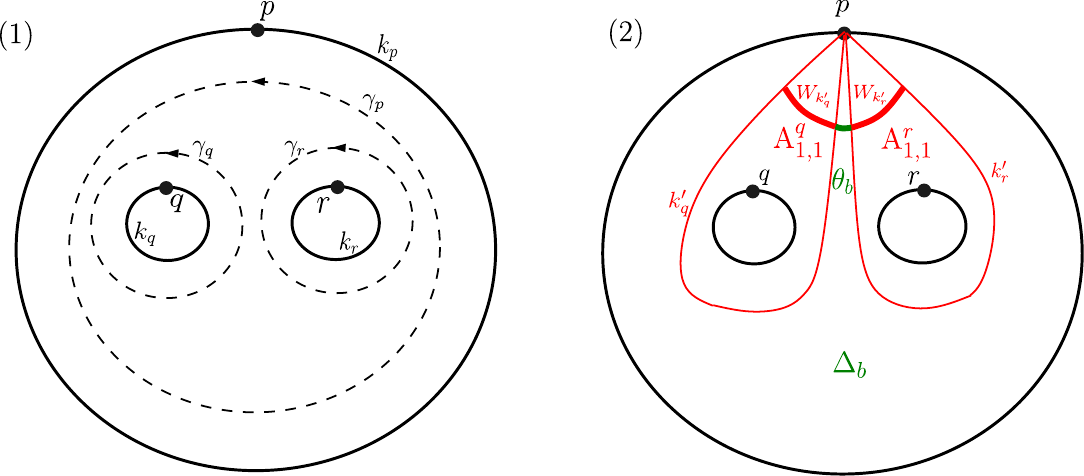}
	\small
	\caption{(1) The counterclockwise oriented loops $\gamma_p$, $\gamma_q$, $\gamma_r$ surround the crown ends $k_p$, $k_q$, $k_r$. 
	(2) The bi-infinite geodesics $k_q'$, $k_r'$, $k_p$ form a $p$-narrowest ideal triangle with the potential $\theta_b:=W_p(k_q', k_r')$ at the cusp $p$. The annulus $\A^q_{1,1}$ is bounded by the bi-infinite geodesic $k_q'$ and the crown end $k_q$ with the potential $W_{k_q'}$ at $p$, while the annulus $\A^r_{1,1}$ is bounded by the bi-infinite geodesic $k_r'$ and the crown end $k_r$ with the potential $W_{k_r'}$ at $p$.}
	\label{figure:cpp1}
\end{figure}

The Dehn twist of $a$ by $\gamma_p$ is denoted by $\gamma_p a$. Then  bi-infinite geodesics $\gamma_p k_r'$, $k_q'$, $k_p$ form a $p$-narrowest ideal triangle as in Figure \ref{figure:cpp2}(1) with the potential $\theta_a:=W_p(\gamma_p k_r', k_q')$ at the cusp $p$. 
\begin{figure}[ht]
	\centering
	\includegraphics[scale=0.6]{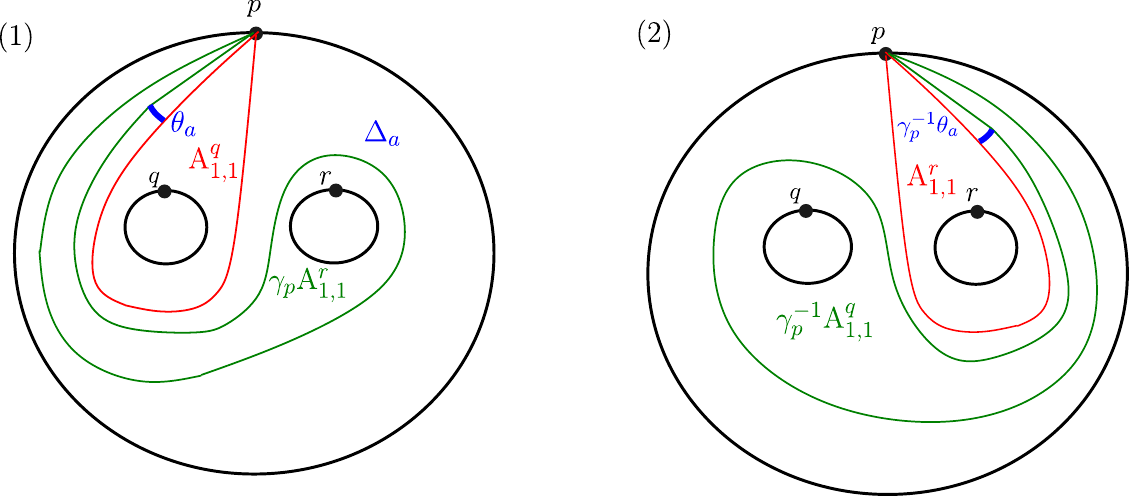}
	\small
	\caption{The potentials $\theta_a$ and $\gamma_p^{-1}\theta_a$ are indicated by the blue arcs.}
	\label{figure:cpp2}
\end{figure}
Any geodesic emitting from $p$ within the arc between $\gamma_p k_r'$ and $k_q'$ hits the boundary $k_p$. Similarly, the bi-infinite geodesics $k_r'$, $\gamma_p^{-1} k_q'$, $k_p$ form a $p$-narrowest ideal triangle as in Figure \ref{figure:cpp2}(2) with the potential $\gamma_p^{-1} \theta_a:=W_p(k_r', \gamma_p^{-1} k_q')$ at the cusp $p$. Any geodesic emitting from $p$ within the arc between $k_r'$ and $\gamma_p^{-1} k_q'$ hits the boundary $k_p$. Inductively, we have
\be
W_p=2 Q(\gamma_p, k_p)+ \sum_{i\in \mathbb{Z}}  (\gamma_p^i\theta_a+W_{\gamma_p^i k_q'}+ \gamma_p^i\theta_b+W_{\gamma_p^i k_r'}).  
\ee

Within $\gamma_p^i \A_{1,1}^q$, let $\{\ell_{i,q}^j\}_{j\in \mathbb{Z}}$ be a family of simple arcs connecting $p$ and $q$ where $\ell_{i,q}^{j+1} = \gamma_q \ell_{i,q}^j$. Inside the ideal triangle $\Delta_{i,q}^j$ formed by $\ell_{i,q}^j$, $\ell_{i,q}^{j+1}$ and $k_q$, every geodesic emitting from $p$ between $\ell_{i,q}^j$ and $\ell_{i,q}^{j+1}$ hits $k_q$. Then
\[W_{\gamma_p^i k_q'}\ =\ 2 Q(\gamma_q, \gamma_p^i k_q') + \sum_{j\in \mathbb{Z}} W_p(\ell_{i,q}^j, \ell_{i,q}^{j+1} ).\]
Similarly, we obtain 
\[W_{\gamma_p^i k_r'}\ =\ 2 Q(\gamma_r, \gamma_p^i k_r') + \sum_{j\in \mathbb{Z}} W_p(\ell_{i,r}^j, \ell_{i,r}^{j+1} ).\]
Combining the above three equations, we obtain the McShane identity
\be
\begin{aligned}
\label{eqcpp1}
W_p-2Q(\gamma_p, k_p)& \ = \ \sum_{i\in \mathbb{Z}} \gamma_p^i\theta_a+  \sum_{i\in \mathbb{Z}} \gamma_p^i\theta_b + 2 \sum_{i\in \mathbb{Z}} Q(\gamma_q, \gamma_p^i k_q')+ 
2 \sum_{i\in \mathbb{Z}} Q(\gamma_r,\gamma_p^i k_r')
\\& \ + \ \sum_{i\in \mathbb{Z}}  \sum_{j\in \mathbb{Z}} W_p(\ell_{i,q}^j, \ell_{i,q}^{j+1} )+ \sum_{i\in \mathbb{Z}}  \sum_{j\in \mathbb{Z}} W_p(\ell_{i,r}^j, \ell_{i,r}^{j+1}).
\end{aligned}
\ee

\section{Exponential volumes and     Whittaker--Hecke   algebra for ${\rm SL}_2(\R)$}   \la{WHalgebra}

 Recall the three elementary decorated surfaces: the triangle $\tau$, the trouser leg $\T$, and a pair of pants.  
The volume of the last is equal to $1$. The exponential volumes of the first two are non-trivial functions ${\cal E}_\tau(\K_1, \K_2, \K_3)$ and ${\cal E}_\T(\K, \L)$. The related 
${\cal B}-$function is the Bessel function:
 \be \la{BESS}
{\cal B}_\T(\K, s) = \int_{\R_{>0}} {\cal E}_\T(\K, \L)\L^{s/2}d\log \L \stackrel{\eqref{FBFa}}{=} 2J_s(\K).
\ee
 We define an algebra ${\cal E}$,  consisting of functions $f$ on  $\R_{>0}$ with exponential decay at infinity,  
with the product $\ast$   given by   
\be \la{PFE9}
(f\ast g)(\K_3):= \int_{\R_{>0}\times \R_{>0}} {\cal E}_\tau(\K_1, \K_2, \K_3)f(\K_1)g(\K_2) d\log\K_1 d\log \K_2.
\ee

 \bt The product $\ast$ is commutative and associative. 
  Functions  ${\cal B}_\T(\K, s)$     
  satisfy the  product formula 
\be \la{PF10}
 {\cal B}_\T(\K_1, s)\cdot {\cal B}_\T(\K_2, s) = \int_{\R_{>0}} {\cal E}_\tau(\K_1, \K_2, \K_3){\cal B}_\T(\K_3, s)d\log \K_3.
\ee
The function ${\cal B}_\T(\K, s)$ provides a homomorphism $\psi_s$ of the algebras $({\cal E}, \ast)$ to $\C$:
\be \la{PF13}
\psi_{s}(f):=  \int_{\R_{>0}} f(\K){\cal B}_\T(\K, s)d\log \K.
\ee
\et  

\bpr The product $\ast$ is  commutative since the function ${\cal E}_\tau(\K_1, \K_2, \K_3)$ is symmetric in $\K_1, \K_2, \K_3$. 

\vskip 1mm
{\it The associativity}. Recall  the exponential volume form for the rectangle $\square$:
\be \la{EVFREC}
 e^{-W_\square}\Omega_\square(\K_1, \K_2, \K_3, \K_4).
\ee
  There are  two ways to cut the rectangle  
  into two triangles, see Figure \ref{RT11}: $\square = \tau_1 \cup \tau_2 = \tau_3\cup \tau_4$.  
 So the cutting and gluing properties of  exponential volume forms provide two   presentations of the exponential volume form of the rectangle: 
   \be
   \begin{split}
(\ref{EVFREC}) 
= &{\cal E}_{\tau_1}(\K_1, \K_2, \K) {\cal E}_{\tau_2}(\K, \K_3, \K_4) d\log \K   = {\cal E}_{\tau_3}(\K_2, \K_3, \K) {\cal E}_{\tau_4}(\K, \K_4,  \K_1) d\log \K. \\
\end{split}
\ee
Multiplying it by  $f_1(\K_1)f_2(\K_2) f_3(\K_3)f_4(\K_4)$ and integrating we get the associativity of the product $\ast$. 

  \begin{figure}[ht]
\centerline{\epsfbox{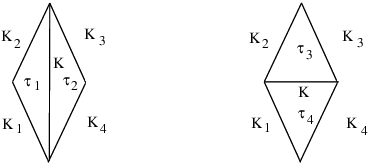}}
\caption{Associativity of the $\ast-$product  is equivalent  to the  independence of the exponential volume form of a  triangulation of the rectangle.}
\label{RT11}
\end{figure} 
\vskip 1mm

{\it The product formula}. 
Cut the disc $\D_1^*$ along the radius, getting a triangle $\tau'$ with a vertex $p$ corresponding to the puncture of $\D_1^*$.  Let ${\cal T}_{\tau'}$ be the space parametrising ideal geodesic triangles 
$[a,b,p]$ with horocycles $h_a, h_b$ at the vertices $a,b$. 
The notation $\tau'$ stresses that one of the cusps,  the cusp $p$,  does not carry a horocycle.    Pick any horocycle $h_p$ at  $p$. Recall  the signed geodesic length $\kappa_{xy}$  between horocycles 
$h_y$ and $h_y$.   The space ${\cal T}_{\tau'}$ carries a well defined  function, independent of the choice of horocycle $h_p$:  
\be \la{wdfT} 
 \ \L_{\tau'}= e^{-(\kappa_{ap}-\kappa_{pb})}.
\ee

  \begin{figure}[ht]
\centerline{\epsfbox{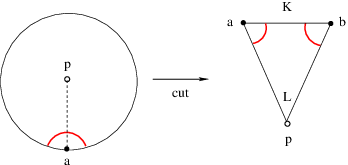}}
\caption{Cutting a trouser leg $\T = \D_1^*$ by a radius we get a triangle with potentials at two vertices.}
\label{RT8}
\end{figure} 

We fix the value of the coordinate $\K:= e^{-\kappa_{ab}}$. Then there is the exponential volume form
\be \la{WWWa}
e^{-W_{\tau'}}\Omega_{\tau'}(\K) := e^{-W_{\tau'}}d\log \L_{\tau'}, \ \ \ \  W_{\tau'}= W_a+W_b. 
\ee

Consider the space ${\cal T}_{\square'}$ parametrising ideal geodesic rectangles $(a,b,c,p)$ with horocycles $h_a, h_b, h_c$ at the vertices $a,b,c$. We fix the value of the coordinate 
$\K_1:= e^{-\kappa_{ab}}$ and   $\K_2:= e^{-\kappa_{bc}}$.  The space ${\cal T}_{\square'}$  carries  
 a  well defined function $  \L_{\square'}$ similar to (\ref{wdfT}), and the volume form:
\be
  \L_{\square'}:= e^{-(\kappa_{ap}-\kappa_{pc})}, \ \ \ \ \ \ \Omega_{\square'}(K_1, \K_2) := d\log \K\wedge d\log \L_{\square'}. \ee

 \begin{figure}[ht]
\centerline{\epsfbox{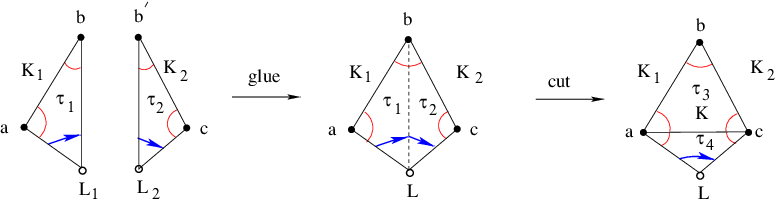}}
\caption{Gluing triangles with horocycles $\tau_1'$ and $\tau_2'$, and then cutting by the diagonal $ac$.}
\label{RT9}
\end{figure} 

Take two triangles with horocyles $\tau_1'=(a,b_1, p; h_a, h_{b_1})$ and $\tau_2'= (b_2,c,p; h_{b_2}, h_c)$. Glue them along the sides $b_1p$ and $b_2p$ matching the horocyles $h_{b_1}$ and $h_{b_2}$, see Figure \ref{RT9}. 
We get a rectangle  $\square'  = (a,b,c, p)$  with three  horocycles $h_a, h_b, h_c$. 
 The gluing provides an isomorphism of spaces with volume forms:
 \be
{\cal T}_{\square'} = {\cal T}_{\tau_1'}\times {\cal T}_{\tau_2'}; \ \ \ \Omega_{\square'}(\K_1, \K_2) = \Omega_{\tau_1'}(\K_1) \wedge \Omega_{\tau_2'}(\K_2).
\ee
Then we have the  gluing conditions for the potentials and the  functions  $\L$:

$$
W_{\square'} = W_{\tau_1} + W_{\tau_2} = W_{\tau_3} + W_{\tau_4},\ \ \L_{\square'} = \L_{\tau_1} \L_{\tau_2}.
$$ 
Therefore  using (\ref{WWWa}) we have a factorization of the exponential volume form: 
\be \la{evfa}
e^{-W_{\square'}}\Omega(\K_1, \K_2)  = e^{-W_{\tau_1}(\K_1,\L_{\tau_1})}  d\log(\L_{\tau_1}) \wedge  e^{-W_{\tau_2}(\K_2, \L_{\tau_2})}   d\log\L_{\tau_2}. 
\ee

 \begin{figure}[ht]
\centerline{\epsfbox{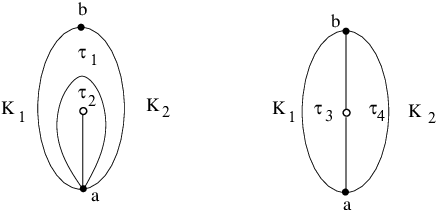}}
\caption{The product formula results from   calculating the  exponential volume form for $\D_2^*$ in two ways.}
\label{RT10}
\end{figure} 
On the other hand, cutting the rectangle $\square'$ by the diagonal $ac$ into two triangles $\tau_3$ and $\tau_4$, see Figure \ref{RT9}, where  $\tau_3 = (a,b,c)$ with horocycles $h'_a, h_b, h_c'$ and $\tau'_4=(a,c, p)$ with horocycles $h_a'', h_c''$, we get 
\be \la{evfb}
e^{-W_{\square'} }\Omega(\K_1, \K_2)   = e^{-W_{\tau_3}(\K_1, \K_2, \K) }\cdot e^{-W_{\tau_4}(\K, \L_{\square'})} d\log \L_{\square'} \wedge d\log \K.
\ee
Comparing (\ref{evfa}) and (\ref{evfb}) and multiplying by $\L_{\tau_1}^{s/2}\L_{\tau_2}^{s/2} = \L_{\square'}^{s/2}$ we get 
$$
 e^{-W(\K_1; \L_{\tau_1})}\L_{{\tau_1}}^{s/2} d\log \L_{\tau_1} \wedge e^{-W_2(\K_2; \L_{\tau_2} )}\L^{s/2}_{\tau_2} d\log \L_{\tau_2} =  {\cal E}_{\tau_3}(\K_1, \K_2, \K) \cdot e^{-W_{\tau_4}(\K,\L_{\square'})} 
 \L_{\square'}^{s/2}d\log (\L_{\square'}) \wedge d\log \K. 
 $$
 Integrating these exponential volume forms we get the product formula, see Figure \ref{RT10}.
 \vskip 2mm

{\it The multiplicativity of the map $\psi_s$}. Equality (\ref{PF13}) follows  from (\ref{PFE9}) and (\ref{PF10}):
\be \la{PF3}
\begin{split}
\kappa_{s}(f\ast g) = &\int_{\R_{>0}} (f\ast g)(\K_3) {\cal B}_\T(\K_3, s)   d\log \K_3 \\
=&\int_{\R_{>0}} {\cal E}_\tau(\K_1, \K_2, \K_3) f(\K_1)g(\K_2) {\cal B}_\T(\K_3, s) d\log \K_1 d\log \K_2 d\log \K_3  \\
=&\int_{\R_{>0}}  {\cal B}_\T(\K_1, s){\cal B}_\T(\K_2, s) f(\K_1)g(\K_2)   d\log \K_1 d\log \K_2  \\
=&\kappa_{s}(f) \cdot \kappa_s(g). \  \\
\end{split}
\ee
\epr

Few comments are in order.

\begin{enumerate} \item  Cutting the decorated surface $\D_2^*$ along the radius connecting the special point and the puncture we get the rectangle $\square'$. The product formula just means that
 calculating the exponential volume of the moduli space for  $\D_2^*$  
  using the two  triangulations of $\D_2^*$  on Figure \ref{RT10} leads to the same result.  
  
  \item The Cartan group $\H(\R)$ of  ${\rm PGL}_2(\R)$ has two components. One of them is  the positive part of the Cartan group $\H(\R_{>0}) = \R^\times_{>0}$. 
   Let $\N \subset {\rm PGL}_2$ be the upper triangular unipotent subgroup, so $\N(\R) = \R$. Let $\psi(a) = e^{2\pi i a}$ be an additive character of $\N(\R)$.  
    The Laplace operator $\Delta_{{\rm sl}_2}$ is  the generator of the center of the universal enveloping algebra ${\cal U}({\rm sl}_2(\R))$.  
    
    Recall the Whittaker function ${\cal S}(g, s)$, where $g \in {\rm PGL}_2(\R)$, of the principal series representation $V_s$ of  ${\rm PGL}_2(\R)$. 
  It has the following properties:
   
\begin{itemize} \item  The function ${\cal S}(g, s) $   is $(\N(\R), \psi)$ bi-invariant:  $$
   {\cal S}(n_1gn_2, s) = \psi(n_1)\psi(n_2){\cal S}(g, s)\ \ \ \ \ \ \forall n_{1}, n_2 \in \N(\R).
   $$
 
 \item  Its restriction to  the coset $w_0\H(\R_{>0}) $ of the positive Cartan subgroup  is  the Bessel function: 
   $$
    {\cal S}(w_0h(\K), s) = {\cal B}_\T(\K, s), \ \ \ \K\in \R^\times_{>0}, \ \ \ \ \ \ \ h(\K):=   \begin{pmatrix}  
      \K & 0 \\
      0 & \K^{-1}\\
   \end{pmatrix},   \ \ w_0= \begin{pmatrix}  
      0 & 1 \\
      -1 & 0\\ \end{pmatrix}. $$

 \item The function ${\cal S}(g, s) $ is an eigenfunction of the Laplace operator $\Delta_{{\rm sl}_2}$. 
   \end{itemize}
    
   Let us elaborate the analogy with the classical zonal spherical functions \cite{Ge50}.

  {\it The classical Hecke algebra  ${\cal H}_{\rm SO(2)}$}.   It is   
   given by compactly supported functions on ${\rm SL}_2(\R)$ which are left and right invariant under the action of the maximal compact subgroup ${\rm SO}(2)\subset {\rm SL}_2(\R)$. The product is given by the convolution. It is associative and  commutative.   Any irreducible unitary spherical principal series representation $V_s$ of the group ${\rm SL}_2(\R)$  contains a unique 
  normalised   spherical vector $v_s$. The corresponding matrix element  is called the {\it zonal spherical function}: 
   \be
   {\cal S}_{\rm SO(2)}(g, s):=  \langle v_s, gv_s\rangle.
      \ee
      It is an eigenfunction of the Laplace operator, 
     bi-invariant under the action of the group ${\rm SO}(2)$. So it is determined by its restriction to the subgroup $\H(\R)$. Therefore the product in the algebra ${\cal H}_{\rm SO(2)}$ can be written in terms 
   of the functions on  $\H(\R)$ using a certain kernel $a(\K_1, \K_2, \K_3)$:
  \be \la{SU2aa}
(f\ast g)(\K_3):= \int_{\R^\times\times \R^\times} {a}(\K_1, \K_2, \K_3)f(\K_1)g(\K_2) d\log\K_1 d\log \K_2.
\ee

  The zonal function  $  {\cal S}_{\rm SO(2)}(\K, s)$     
  satisfies the  product formula 
\be \la{PF10a}
   {\cal S}_{\rm SO(2)}(\K_1, s)\cdot   {\cal S}_{\rm SO(2)}(\K_2, s)= \int_{\R^\times} {a}(\K_1, \K_2, \K_3)  {\cal S}_{\rm SO(2)}(g, s)d\log \K_3.
\ee
It provides a homomorphism  of the algebras $({\cal H}_{\rm SO(2)}, \ast)$ to $\C$:
\be \la{PF13a}
\varphi_{s}(f):=  \int_{\R^\times} f(\K){\cal S}_{\rm SO(2)}(\K, s)d\log \K, \ \ \ \ \varphi_{s}(f\ast g) = \varphi_{s}(f)\varphi_{s}(g).
\ee
  
\item {\it Conclusion}. The algebra ${\cal E}$ is the analog of the Hecke algebra ${\cal H}_{\rm SO(2)}$ where the ${\rm SO}_2$ bi-invariance is replaced by the $(\N(\R), \psi)$ bi-invariance. 
However the  subgroup $\N(\R)$ is not compact, and so the convolution of $(\N(\R), \psi)$ bi-invariant functions is divergent. Nevertheless the restriction to the $w_0-$coset of the positive part of the Cartan torus is well defined, and there is  a commutative algebra with exactly the same   properties, which  we call    the {\it positive Hecke-Whittaker algebra}.

{\it The positive Hecke-Whittaker algebra ${\cal E}$ for ${\rm PGL}_2(\R)$  is given by the exponential volumes of elementary decorated surfaces.   
The exponential volumes of moduli spaces for all decorated surfaces, together with the unfolding formula, provide an extension  of the algebra ${\cal E}$ to all decorated surfaces}. 

 \end{enumerate}

\begin{appendices}

\section{Birman--Series theorem for ideal hyperbolic surfaces}

The main result of the Appendix is the following theorem, generalising the Birman--Series theorem \cite{BS85} to ideal hyperbolic surfaces. 
\begin{theorem}
\label{thm:bs}
Given an ideal hyperbolic structure on the decorated surface $\bS$, let $\mathcal{G}$ be the union of all bi-infinite geodesics without self-intersection. Then the area of $\mathcal{G}$ with respect to the measure on the surface induced by the hyperbolic structure  is equal to zero.
\end{theorem}
Let us start with the collar lemma.
Given $R>0$ and a cusp/puncture $p$, there is a unique horoarc/horocycle $h_R$ with the length $R$. Let us define the {\em collar neighbourhood} $C_R$ be the annular neighbourhood region bounded by $h_R$. 

\begin{lemma}
\label{lem:co}
Given an ideal hyperbolic surface, for any cusp/puncture $p$, there exists a collar neighborhood $C_r$ of a cusp/puncture $p$ such that for any bi-infinite geodesic $\ell$ entering and exiting $C_r$, the geodesic $\ell$ has self-intersection. 
\end{lemma}
\begin{figure}[ht]
	\centering
	\includegraphics[scale=0.7]{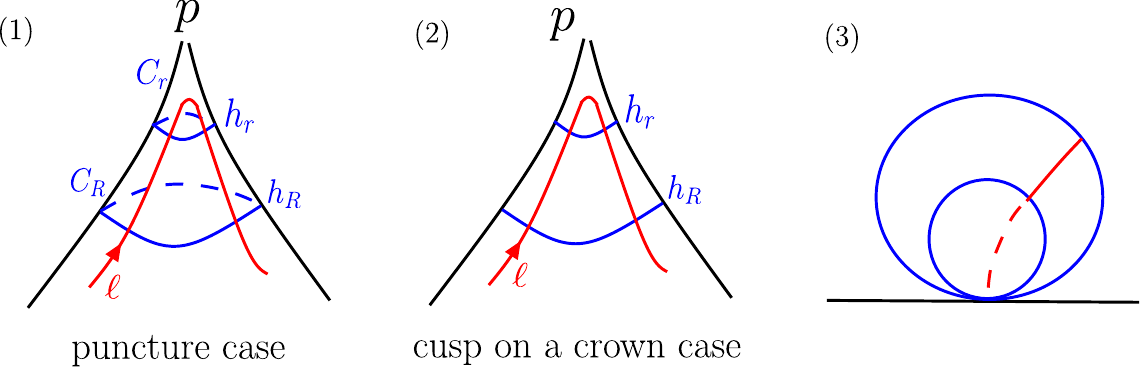}
	\small
	\caption{(1) The geodesic $\ell$ enters and exits $C_r$ for the puncture case. (2) The geodesic $\ell$ enters and exits $C_r$ for the cusp on a crown case case. (3) Lift of $\ell^{\pm}$ in the universal cover if stay within $C_R$ at infinity. }
	\label{figure:clem}
\end{figure}
\begin{proof}
Let us choose some horoarc/horocycle $h_R$ for $R>0$. Since the collar neighbourhood $C_R$ is infinitely long, we can choose a smaller collar neighbourhood $C_r\subset C_R$ such that the distance between $h_r$ and $h_R$ is at least $R/2$. If a geodesic $\ell$ enters and exits $C_r$ and does not have self-intersection, then $C_r\cap \ell$ contains at least two different points as in Figure \ref{figure:clem}(1)(2). We define $\ell^+$ ($\ell^-$ resp.) to be the geodesic ray starting from the first (last resp.) intersection point of $\ell \cap C_r$ towards the (opposite resp.) direction of $\ell$. We claim that both $\ell^+$ and $\ell^-$ must leave $C_R$ for any of the two directions. If $\ell^{\pm}$ does not leave $C_R$, the ideal end point of any lift of $\ell^{\pm}$ is the unique ideal boundary point of the horodisk in the universal cover as in Figure \ref{figure:clem}(3). This characterizes $\ell^{\pm}$ as a geodesic going straight up to the cusp, and thus hitting every horocycle at most once. This is a contradiction as $\ell^{\pm}$ meets $C_r$ in two places. Hence both $\ell^+$ and $\ell^-$ must leave $C_R$. Let $\bar{\ell}$ be the subarc of $\ell$ which 
lies completely within $C_R$, has both its endpoints on $h_R$, and enters and exits $C_r$. Since $\bar{\ell}$ has two subarcs between $h_r$ and $h_R$, it has length at least $R$. On the other hand, the geodesic arc $\bar{\ell}$ is endpoint-fixing homotopic to a horocyclic path along $h_R$ without wrapping around $h_R$. This implies the length of $\bar{\ell}$ is strictly less than $R$, leading to a contradiction. We conclude that any geodesic $\ell$ entering and exiting $C_r$ has self-intersection.
\end{proof}
As a consequence, any geodesic in $\mathcal{G}$ lies in a compact set $Q=\bS-\cup_p C_r$. Now, we fix an ideal triangulation $\mathcal{T}$ of the ideal hyperbolic surface. The ideal triangulation $\mathcal{T}$ cuts any geodesic in $\mathcal{G}$ into segments. Then we define $\mathcal{G}(N)$ to be the set of geodesic arcs in $\mathcal{G}$ that are cut up into $N$ geodesic segments by $\mathcal{T}$.
\begin{corollary}
\label{cor:comtop}
There exists a constant $C>0$, such that for any $\gamma \in \mathcal{G}(N)$, we have the length $l_\gamma\geq C \cdot N$ for any $N\geq 1$.
\end{corollary}
\begin{figure}[ht]
	\centering
	\includegraphics[scale=0.5]{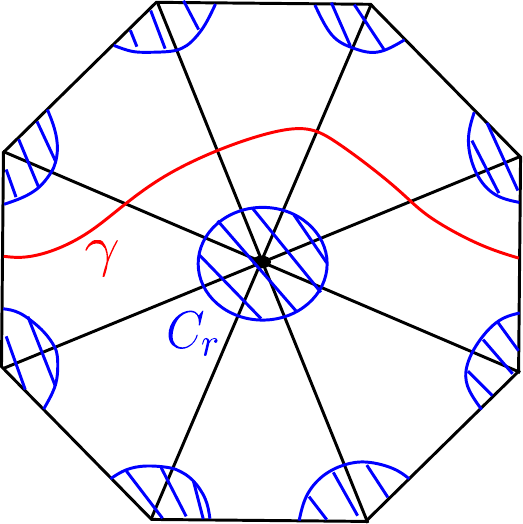}
	\small
	\caption{$\gamma \in \mathcal{G}(N)$ and $\gamma\subset \bS-\cup_p C_r$.}
	\label{figure:gammapdt}
\end{figure}

\begin{proof}

Any ideal triangle of the ideal triangulation $\mathcal{T}$ is cut into three compact intervals by $\cup_p C_r$. By Lemma \ref{lem:co}, the length of a segment of $\gamma \in \mathcal{G}(N)$ is determined by the pair of points on two of these three compact intervals. Since there are only finitely many ideal triangles, the length of a segment of $\gamma \in \mathcal{G}(N)$ is a function on a compact set. Thus the length of a segment is bounded below by a constant $C>0$. Hence $l_\gamma\geq C \cdot N$. 
\end{proof}

\begin{proof}[Proof of Theorem \ref{thm:bs}]
Given an ideal triangulation $\mathcal{T}$ of the ideal hyperbolic surface, any homotopy class $[\gamma]\in [\mathcal{G}(N)]$ is determined by 
\begin{enumerate}
\item the multiset of $N$ segments in $[\mathcal{G}(1)]$, thus $x_1+\ldots+ x_{\#[\mathcal{G}(1)]}=N$ for $x_i$ being the number of segments in these $N$ segments of the given type;
\item the starting and ending segments of $[\gamma]$. We have
\end{enumerate}
\[\#[\mathcal{G}(N)]\leq N^2 \cdot \tbinom{\#[\mathcal{G}(1)]+N-1}{N-1}  =p_0(N),\]
where $p_0$ is a polynomial in $N$.

Consider one fundamental domain $F$ of the universal cover of the ideal hyperbolic surface in the Poincar\'e disk model. Let $\tilde{\mathcal{T}}$ be the universal cover of the ideal triangulation $\mathcal{T}$. Then $\tilde{\mathcal{T}}$ cuts $F$ into finitely many ideal triangles $F\backslash \tilde{\mathcal{T}}=\cup_{i=1}^l \Delta_i$.
\[I=\{\sigma=\tilde{\gamma}\cap \Delta_i \;|\; \text{ for some } i=1,...,l, \;\;\tilde{\gamma} \text{ lift of } \gamma \in \mathcal{G}\}.\]
To prove the theorem is equivalent to prove the Euclidean area of $I$ is zero. For any integer $N>0$, any $\sigma \in I$ is 
\begin{enumerate}
\item uniquely expressed as the $(N+1)$-th segment of $\gamma \in \mathcal{G}(2N+1)$;
\item only one side of $\sigma$ could extend to $N$ segments, while the other side ends at some cusp/puncture;
\item both sides end at some cusp/puncture before $N$ segments.
\end{enumerate}
\begin{figure}[ht]
	\centering
	\includegraphics[scale=0.7]{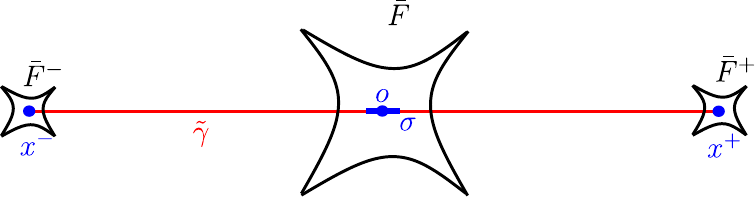}
	\small
	\caption{$\tilde{\gamma}$ is a lift of $\gamma \in \mathcal{G}(2N+1)$ and $\sigma \subset \tilde{\gamma}$ is the lift of $(N+1)$-th segment.}
	\label{figure:bsp}
\end{figure}
For Case (1),  consider the lift $\tilde{\gamma}$ of $\gamma$ where $\sigma \subset \tilde{\gamma}$ as in Figure \ref{figure:bsp}. The segment $\sigma$ is contained in the closure $\bar{F}$ of the fundamental domain $F$. Let $x^-$ and $x^+$ be two end points of $\gamma$ which are contained in two fundamental domains $\bar{F}^-$ and $\bar{F}^+$ respectively. Choose any point $o\in \sigma$, by Corollary \ref{cor:comtop}, we have $d(x^-,o)\geq C\cdot N$ and  $d(x^+,o)\geq C\cdot N$. The {\em Euclidean diameter} $\rm{diam}_E{\bar{F}}$ of $\bar{F}$ is finite in the Poincar\'e disk model. Since the ideal hyperbolic surface has negative constant curvature $-1$, we get the Euclidean diameters $\rm{diam}_E{\bar{F}^-}\leq C_0 e^{-C_1 N}$ and $\rm{diam}_E{\bar{F}^+}\leq C_0 e^{-C_1 N}$ for some constant $C_0, C_1>0$. Then for any representive $\gamma$ of $[\gamma]$, we cover any $\sigma(\subset\gamma)$ by the convex hull $\square_{[\gamma]}$ of $\bar{F}^-$ and $\bar{F}^+$. The Euclidean area of $\square_{[\gamma]}$ is less than $C_2 e^{-C_1 N}$ for some constant $C_2>0$. Since
\[\lim_{N\rightarrow +\infty} \# [\mathcal{G}(2N+1)] \cdot C_2 e^{-C_1 N}=0,\]
for any integer $N>0$, we could cover all these kinds of $\sigma$ by all these convex hulls.

For Case (2), let $\sigma\subset \gamma$ and $\gamma\in \mathcal{G}(M)$ for $M\leq 2N$, where $\gamma$ is a geodesic ray ending at the cusp/puncture $p$ and $\sigma$ is the $(N+1)$-th segment of $\gamma$. We cover $\gamma$ by the convex hull of $\bar{F}^-$ and $p$ which has the Euclidean area less than $C_2 e^{-C_1 N}$. Then we have 
\[\lim_{N\rightarrow +\infty} \sum_{M=N+1}^{2N} \# [\mathcal{G}(M)] \cdot C_2 e^{-C_1 N}=0,\]

For Case (3), we cover these finitely many bi-infinite geodesics by themselves.

Combining all these cases, we cover $I$ by a sequence of measuable sets $U_N$ such that the Euclidean area of $U_N$ converges to zero.
\end{proof}

\section{Relative volume forms} \la{APPB}
Recall the relative volume form in  \eqref{RVF2}. Let  $S$ be a sphere with $n$ punctures,  $\ell$  a simple loop on $S$, and $S':=S-\ell$. The monodromy map ${\cal X}_S\lra (\C^\times)^n$ maps an element in ${\cal X}_S$ into $(\L_1,..., \L_n)$ where, given any ideal triangulation, the $\L_i$  is the product of the cluster $\mathcal{X}$-coordinates at the edges sharing the $i-$th puncture. The $\L_i$  does not depend on the triangulation. 
The fibers of the monodromy map 
 ${\cal X}_S\lra (\C^\times)^n$ are the generic symplectic fibers on the cluster Poisson space ${\cal X}_S$. 
 Given a point  of ${\cal X}_S(\mathbb{R}_{>0})$, we have $\L_i=e^{l_i}$ where $l_i\in \mathbb{R}$ is the signed boundary geodesic  length. Let $l$ be the length  of $\ell$ and $\theta$  the twist parameter of $\ell$. 
Let $\L=(l_1,...,l_n)$ and $\L'=(l_1,...,l_n, l,l)$.

 \bl \la{WP=cl} 
 The cluster volume forms $\Omega_S(\L)$ on ${\cal X}_S(\mathbb{R}_{>0})(\L)$ and $\Omega_{S'}(\L')$ on ${\cal X}_{S'}(\mathbb{R}_{>0})(\L')$ 
 are related by
 \be
 \Omega_S(\L) = \frac{1}{2} \cdot \Omega_{S'}(\L') \wedge dl \wedge d\theta.
  \ee
 \el
 
 \bpr

 By Wolpert's formula \cite{Wol83}, we have the Weil--Petersson form and the corresponding volume form: 
 \be
\begin{split}
&\omega_{\rm WP} =  \sum_{i=1}^{n-3} d l_i\wedge d\theta_i,\;\;\; \Omega^{\rm WP}_S(\L) := \omega_{\rm WP}^{n-3}/(n-3)!.\\
 &\Omega^{\rm WP}_S(\L) =    \Omega^{\rm WP}_{S'}(\L')  \wedge dl \wedge d\theta.\\
    \end{split}
 \ee
 
   \begin{figure}[ht]
	\centering
	\includegraphics[scale=0.5]{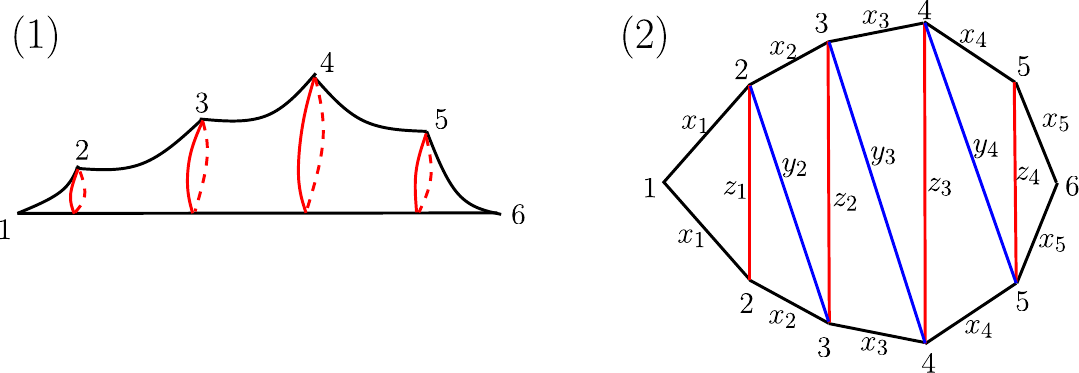}
	\small
	\caption{(1) The $6$-th punctured sphere. (2) The ideal triangulation.}
	\label{figure:catep}
\end{figure}

Let us relate the cluster volume form $\Omega_S(\L)$ to $\Omega^{\rm WP}_S(\L)$.
Given any ideal triangulation as in Figure \ref{figure:catep}, we have
\[l_1=x_1, l_n=x_{n-1},\]
\[l_i=2 z_{i-1}+x_{i-1}+x_{i}+y_{i-1}+y_{i}\]
for $i=2,..., n-1$ with $y_1=y_{n-1}:=0$.
By definition
\[\Omega_S(\L)\wedge \bigwedge_{i=1}^{n} d l_i=2\bigwedge_{i=2}^{n-2} \left(d x_i \wedge d y_i \right) \bigwedge_{i=1}^{n-2}d z_i \wedge d x_1 \wedge d x_{n-1}.\]
Thus 
\[\Omega_S(\L)= \frac{1}{2^{n-3}}\bigwedge_{i=2}^{n-2} \left(d x_i \wedge d y_i \right).\] 
So we get
\[\omega_{\rm cl}= \frac{1}{2} \sum_{i=2}^{n-2} d x_i \wedge d y_i.\]
By \cite[Appendix A]{Pen92}, $2\omega_{\rm cl} = \omega_{\rm WP}$. So we obtain
\be \la{WPCLn}
\Omega^{\rm WP}_S(\L) = \omega_{\rm WP}^{n-3}/(n-3)!=(2\omega_{\rm cl})^{n-3}/(n-3)!=\bigwedge_{i=2}^{n-2} \left(d x_i \wedge d y_i \right)=2^{n-3}\Omega_S(\L) .
\ee
Since $\alpha$ cuts $S$ into $S'=S_{0,n_1}\cup S_{0,n_2}$ where $n=n_1+n_2-2$, we obtain 
\[\Omega_S^{\rm WP}(\L)=\Omega_{S_{0,n_1}}^{\rm WP}(\L'') \wedge \Omega_{S_{0,n_2}}^{\rm WP}(\L''') \wedge d l_\alpha \wedge d \theta_\alpha.\]
Thus
\[2^{n-3}\Omega_S(\L)= 2^{n_1-3}\Omega_{S_{0,n_1}}(\L'') \wedge 2^{n_2-3}\Omega_{S_{0,n_2}}(\L''') \wedge d l_\alpha \wedge d \theta_\alpha,\]
which implies
 \be
 \Omega_S(\L) = \frac{1}{2}\cdot \Omega_{S'}(\L') \wedge dl \wedge d\theta.
  \ee
 \epr

\end{appendices}

 \end{document}